\documentclass[a4paper, 10pt]{article}
\usepackage[utf8]{inputenc}
\usepackage[top=3cm,bottom=2cm,left=3cm,right=3cm,marginparwidth=1.75cm]{geometry}

\usepackage[backend=biber, hyperref=true, style=numeric, sorting=nyt, doi=true, isbn=false, url=false, maxbibnames=9, sortcites=true]{biblatex}
\usepackage[bookmarksnumbered, bookmarks]{hyperref}

\usepackage{setspace} 

\renewenvironment{abstract}
  {\quotation
  {\bfseries\noindent{\abstractname:}}}
  {\endquotation}

\usepackage{sectsty} 
\sectionfont{\centering}

\usepackage{amssymb,amsmath,amsthm}
\allowdisplaybreaks 

\theoremstyle{definition}                            
\newtheorem{thm}{Theorem}[section]                   
\newtheorem{defn}[thm]{Definition}
\newtheorem{lem}[thm]{Lemma}
\newtheorem{prop}[thm]{Proposition}
\newtheorem{cor}[thm]{Corollary}
\newtheorem*{rmk_temp}{Remark}                       
\newtheorem*{nts_temp}{{\color{cyan}Note to self}}   
\numberwithin{equation}{section} 

\newenvironment{rmk}
  {\pushQED{\qed}\begin{rmk_temp}}
  {\popQED\end{rmk_temp}}

\newenvironment{nts}
  {\pushQED{\qed}\begin{nts_temp}\color{cyan}}
  {\popQED\end{nts_temp}}
\usepackage{version}        
\excludeversion{nts}        

\usepackage{lipsum}         
\usepackage{bbm}            
\usepackage{graphicx}       
\usepackage{xcolor}         
\usepackage{upgreek}        
\usepackage{cancel}         
\usepackage[notref, notcite, final]{showkeys} 
\usepackage{mymacros}       
\usepackage{tikz}           
\usetikzlibrary{calc,decorations,calligraphy}
\newcommand\tikzmark[2]{\tikz[remember picture,baseline=(#1.base)]{\node[inner sep=0pt] (#1) {#2};}}
\usepackage[margin=1.8cm]{caption}  
\usepackage{enumitem}               
\usepackage{authblk}                

\title{\LARGE\MakeUppercase{\textbf{A high-contrast composite with \\ annular inclusions: \\Norm-resolvent asymptotics}}}
\author[1]{Yi-Sheng Lim}

\affil[1]{Department of Mathematical Sciences, University of Bath, Claverton Down, Bath BA2 7AY, United Kingdom (Email: ysl64@bath.ac.uk)}

\date{}

\begin{document}
\maketitle

\vspace{-0.8cm}

\begin{abstract}
    We investigate the operator-norm resolvent asymptotics of a high-contrast composite, consisting of a ``stiff" material, with annular ``soft" inclusions (a ``stiff-soft-stiff" setup). This setup is derived from two models with very different effective wave propagation behaviors. Our analysis is based on an operator-framework proposed by Cherednichenko, Ershova, and Kiselev in \cite{eff_behavior}. Then, as a first step towards studying wave propagation on the stiff-soft-stiff composite, we use the effective description to derive analogous ``dispersion functions".

    \vskip 0.5cm


    \vskip 0.5cm
    
\end{abstract}



\onehalfspacing
\section{Introduction}\label{sect:intro}

The theory of homogenization refers to the study of approximating a highly heterogeneous media with a homogeneous one. This has been a subject of intense study since the 1970s, and to date has amassed an extensive literature. We mention for instance, the books \cite{zhikov,bakhvalov_panasenko,bensoussan_lions_papanicolaou}. Recently, there has been a push towards making these approximations \textit{quantitative}. Aside from the clear benefits to numerical applications, obtaining convergence rates is closely linked to the large-scale regularity of solutions \cite{qshomo_book,shen_book}. In another direction, we could also ask for convergence rates \textit{in the operator norm}, as that will immediately imply the same error for functions of the operator, by the functional calculus. On the other hand, there is a development to extend the techniques of \cite{zhikov, bakhvalov_panasenko, bensoussan_lions_papanicolaou} to account for more ``degenerate" situations, for instance when there is a lack of uniform ellipticity \cite{zhikov2000}. This brings us to the setting of the present paper.

\subsection*{Problem outline}

In this paper, we will look at the problem of homogenization for a high-contrast $\eps \mathbb{Z}^d$-periodic composite on $\mathbb{R}^d$, $d\geq 2$. Our composite will consist of two materials which we will refer to as ``soft" and ``stiff", adopting the terminology of elasticity theory. We think of ``soft" having small material coefficients relative to ``stiff".

For $\eps>0$, consider the operator $A_\eps = -\Div(a_{\eps} \nabla \cdot)$, on $L^2(\mathbb{R}^d)$. The coefficient matrix $a_{\eps}$ is defined as $a_{\eps}(x) := \widetilde{a}_{\eps^2}(\frac{x}{\eps})$, where $\widetilde{a}_{\eps^2}$ is a $\mathbb{Z}^d$-periodic matrix with values given by
\begin{align}\label{eqn:intro_coeffmatrix}
    \widetilde{a}_{\eps^2}(y) = \begin{cases}
        \eps^2 I, \quad & y \in \cup_{n \in \mathbb{Z}^d} ( Q_\soft + n ),\\
        I, & y \in \cup_{n \in \mathbb{Z}^d} \left( (Q_\stin \cup Q_\stls) + n \right).
    \end{cases}
\end{align}
Here, the sets $Q_\soft$, $Q_\stin$, and $Q_\stls$ partition the reference period cell $Q = [0,1)^d$, and are arranged in a ``stiff-soft-stiff" setup as follows: We have a simply connected ``stiff-interior" region $Q_\stin$, surrounded by an annular shaped ``soft" region $Q_\soft$, with the remaining region filled by the ``stiff-landscape" part $Q_\stls$. See Figure \ref{fig:intro_toymodel} for a pictorial description of $\widetilde{a}_{\eps^2}$ when restricted to the period cell $Q$. The choice $\eps^2$ in (\ref{eqn:intro_coeffmatrix}) is referred to as the ``double porosity" scaling \cite{arbogast_douglas_hornung_1990}. We impose transmission boundary conditions on the soft-stiff interfaces $\Gamma_\interior$ and $\Gamma_\ls$. See Section \ref{sect:mainmodel} for the precise definition of $A_\eps$.

\begin{figure}[h]
  \centering
  \includegraphics[page=1, clip, trim=4cm 8cm 7cm 5cm, width=1.00\textwidth]{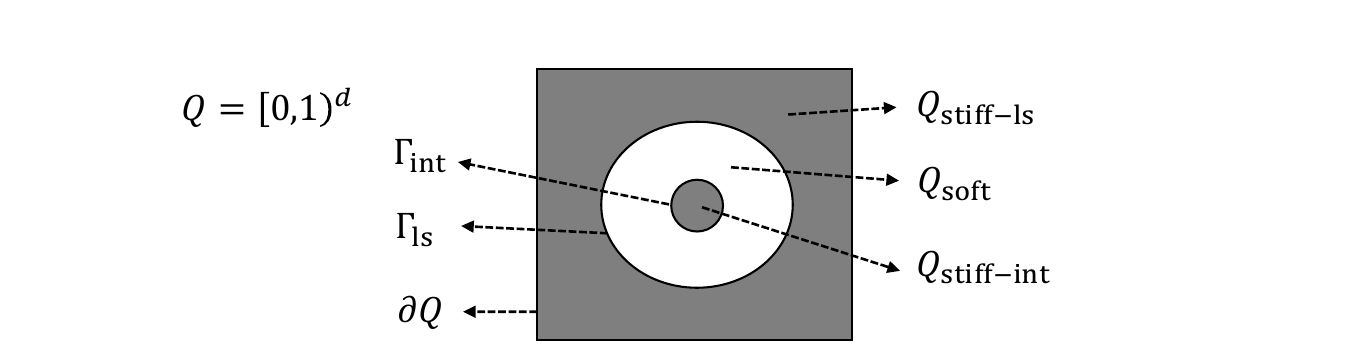}
  \caption{The period cell $Q=[0,1)^d$. The subscript ``stiff-int" stands for stiff interior, and ``stiff-ls" stands for stiff landscape.} \label{fig:intro_toymodel}
\end{figure}

We are interested in the limiting behavior as $\eps \downarrow 0$ of $A_\eps$, in the \textit{norm-resolvent} sense. The reason for this particular choice of (operator) topology is that it gives us direct access to the spectrum, in that it implies spectral convergence/asymptotics of $\sigma(A_\eps)$, in the sense of Hausdorff convergence/asymptotics on compact subsets of the real line $\mathbb{R}$ (see \cite[Section VIII.7]{reed_simon1} and \cite[Definition 4.4.13 and Proposition 4.4.14]{ambrosio_tilli_book}). We intend to use the spectral information to study the effective behavior of waves in high-contrast media.

We employ an operator framework recently proposed by Cherednichenko, Ershova, and Kiselev in \cite{eff_behavior}, which is based on the following key ingredients (see Section \ref{sect:mainmodel} for precise definitions):
\begin{enumerate}[label=(\Alph*)]
    \item The (rescaled) \textit{Gelfand/Floquet transform} $G_\eps$, which helps take the $\eps\mathbb{Z}^d$-periodic operator $A_\eps$ on $L^2(\mathbb{R}^d)$ to a family of operators $A_\eps^{(\tau)}$ on $L^2(Q)$, indexed by $\tau \in Q' = [-\pi,\pi)^d$.
    
    \item Boundary triples $(A_0,\Lambda,\Pi)$ in the sense of Ryzhov \cite{ryzhov2009}, to obtain norm-resolvent estimates for each $A_\eps^{(\tau)}$.
    
    \item Perturbation theory in the sense of Kato \cite{kato_book}, and Reed and Simon \cite[Chapter XII]{reed_simon4}, to ensure that the estimates in (2) are uniform in $\tau$.
    
    \item Generalised resolvents, such as the operator $R_\eps^{(\tau)}(z) = P_\soft (A_\eps^{(\tau)} - z)^{-1} P_\soft$, where $P_\soft$ is the projection of $L^2(Q)$ onto $L^2(Q_\soft)$ (see Section \ref{sect:identifying_homo_op}). Here, the norm-resolvent asymptotics of $R_\eps^{(\tau)}(z)$, which we denote as $R_{\eps, \text{hom}}^{(\tau)}(z)$, is identified with a compression of some $(\mathcal{A}_{\eps,\text{hom}}^{(\tau)} - z)^{-1}$ (Theorem \ref{thm:self_adjoint}).
\end{enumerate}

\subsection*{Existing literature}

Let us give an overview on the literature of operator norm estimates in homogenization. The first operator norm estimates were obtained by Birman and Suslina in \cite{birman_suslina_2004}, who employed the Gelfand transform to obtain the family $A_\eps^{(\tau)}$, and then proceeded with a spectral analysis of $A_\eps^{(\tau)}$ using perturbation theory, studying the behavior of the resolvent of $A_\eps$ near the bottom of the spectrum. Their approach was later extended to include other related setups, for instance, bounded domains \cite{suslina_2013_dirichlet, suslina_2013_neumann} and perforated domains \cite{suslina_2018_perforated}. Other methods that appeared thereafter include: the periodic unfolding method, introduced by Griso in \cite{griso_2004, griso_2006}; the shift method, introduced by Zhikov and Pastukhova in \cite{zhikov_pastukhova_2005_operator} (see also their survey paper \cite{zhikov_pastukhova_2016_opsurvey}); and a refinement of the two-scale expansion method by Kenig, Lin, and Shen \cite{kenig_lin_shen_2012}, which directly dealt with the case of bounded domains (see also the recent book by Shen \cite{shen_book}). Finally, we mention the recent work of Cooper and Waurick \cite{cooper_waurick_2019}, proposing an abstract framework under which uniform in $\tau$ norm-resolvent estimates for the family $A_\eps^{(\tau)}$ can be achieved.

Let us remark that this list is non-exhaustive, and is growing at the point of writing. These methods work well in the moderate-contrast setting (meaning that $\widetilde{a}(y)$ is positive definite and bounded, independently of $\eps$), but cannot be used in the high-contrast case $\widetilde{a}_\eps(y)$ (see (\ref{eqn:intro_coeffmatrix})), at least without serious modifications. This brings us to the approach of \cite{eff_behavior}. One of the goals of this paper is to demonstrate how the approach of \cite{eff_behavior} can be extended from simple geometries like those in Figure \ref{fig:intro_auxmodel} to a geometry like Figure \ref{fig:intro_toymodel}.

Let us make a few historical remarks on the ingredients (A) and (B). The use of Gelfand transform in the mathematical analysis of periodic homogenization problems can be traced back to Zhikov \cite{zhikov1989}, and Conca and Vanninathan \cite{conca_vanninathan1997}. However, they did not pursue the goal of obtaining operator norm estimates. As for ingredient (B), the Ryzhov boundary triple is a generalization of the (``classical") boundary triple introduced independently by Kochubei \cite{kochubei_bdrytriples} and Bruk \cite{bruk_bdrytriples} (see also \cite{behrndt_hassi_desnoo_book}, \cite[Chapter 3]{gorbachuk_gorbachuk_book}, and \cite[Chapter 14]{konrad_book}), and this generalization is more suited for the PDE setting, as it allows the trace operators to be defined on a smaller set than what is required of a classcal boundary triple. A simplified version of the operator framework of \cite{eff_behavior} was initially used in \cite{time_dispersive} (by the same authors) to study a high-contrast homogenization problem on a periodic quantum graph (ODE on the full space $\mathbb{R}$). In the quantum graph setting, the classical triple suffices as the new ingredient in (B). In \cite{split_second}, Cherednichenko, Kiselev and Silva demonstrate the use of Ryzhov triples in a PDE setting on a bounded domain. By combining the techniques of \cite{time_dispersive} and \cite{split_second}, one is able to treat the PDE setting on the full space $\R^d$. This is the content of \cite{eff_behavior}.

\begin{figure}[t]
  \centering
  \includegraphics[page=2, clip, trim=4cm 7cm 7cm 5cm, width=1.00\textwidth]{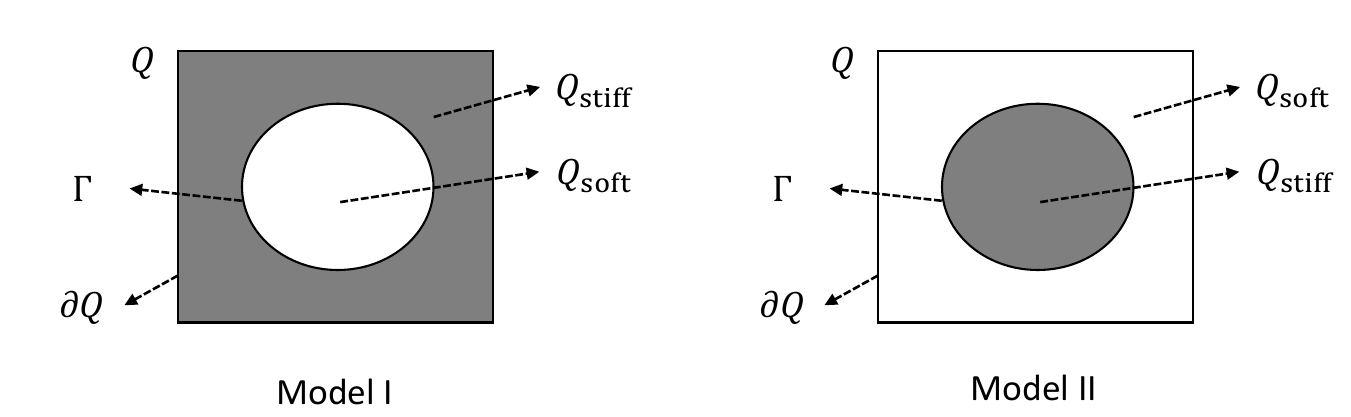}
  \caption{Auxiliary models from \cite{eff_behavior}.}\label{fig:intro_auxmodel}
\end{figure}

\subsection*{Motivation for the stiff-soft-stiff setup}

Next, let us comment on the choice of our setup. The stiff-soft-stiff model (Figure \ref{fig:intro_toymodel}) is derived from the two auxiliary models studied in \cite{eff_behavior}, referred to as Model I and Model II (Figure \ref{fig:intro_auxmodel}). As depicted in the figure, we see that the auxiliary models are geometrically identical, having only one inclusion inside the period cell $Q$, with smooth boundary $\Gamma$, and at a positive distance from the boundary of the cube $\partial Q$. Models I and II differ only in the choice of the soft and stiff components. 

Due to the similarity in geometry, one might naively guess that the homogenized descriptions of $A_\eps$ are almost identical. This is true to a certain extent. Indeed, it was shown in \cite[Section 4.2]{eff_behavior} that for both Models I and II, the fibres $A_\text{hom}^{(\tau)}$ of the homogenized operator $A_\text{hom}$ is one that, roughly speaking, stays unchanged as $-(\nabla + i\uptau)^2$ on $L^2(Q_\soft)$, and acts only on a $1D$ subspace of $L^2(Q_\st)$. A further study on how the constant of multiplication $c_\tau$ in $L^2(Q_\st)$ depends on $\tau$ reveals a non-local behavior of $A_\text{hom}$ on $L^2(\mathbb{R}^d)$, see \cite[Section 5.4]{eff_behavior}.

But $A_\text{hom}$ of Models I and II differ in many respects. For starters, $A_\text{hom}$ depends on $\eps$ for Model I, and does not for Model II. The fact that we only obtain \textit{asymptotics} for Model I is to be expected, because it is known from \cite[p.~1447]{hempellienau} that $A_\eps$ in Model I does not have a norm-resolvent limit. Model I does however possess a strong-resolvent limit $A_{\text{hom,sr}}$ (the Dirichlet Laplacian on the soft parts of $\mathbb{R}^d$ \cite[Proposition 2.2]{hempellienau}), and also a two-scale strong resolvent limit $A_{\text{hom,2sr}}$ \cite[Section 3]{zhikov2005}.

However, the operator $A_{\text{hom,sr}}$ does not capture the spectral information of $A_\eps$, since $\sigma(A_\eps) \nrightarrow \sigma(A_{\text{hom,sr}})$. Furthermore, the manner in which the limit $A_{\text{hom,sr}}$ was obtained in \cite{hempellienau} does not provide us with a rate of convergence. On the other hand, the operator $A_{\text{hom,2sr}}$ satisfies $\sigma(A_\eps) \rightarrow \sigma(A_{\text{hom,2sr}})$, and we even know the spectral decomposition of $\sigma(A_{\text{hom,2sr}})$. But the proof is again qualitative in nature, relying on an additional compactness argument to prove spectral convergence. An alternative route taken by \cite{hempellienau} is to study $\sigma(A_\eps)$ directly, without characterising the limiting behaviour as the spectrum of some $A_{\text{hom},\eps}$.

As for Model II, it was shown in \cite[formulae (5.7) and (7.1), and Theorem 5.1]{zhikov2000} that the two-scale strong resolvent limit $A_\text{hom,2sr}$ exists, using qualitative arguments. But in contrast to Model I, we do not know if there is spectral convergence of $A_\eps$ to $A_\text{hom,2sr}$, and we do not know the decomposition of $\sigma(A_{\text{hom,2sr}})$. The norm-resolvent limit $A_\text{hom}$ (which also the strong-resolvent limit, but different from $A_\text{hom,2sr}$) is obtained in \cite{eff_behavior}, together with a rate of convergence.

It is possible to apply the result of \cite{zhikov2000} to the stiff-soft-stiff model and obtain a two-scale strong resolvent limit $A_\text{hom,2sr}$, but there is work to be done. That includes: verifying if there is spectral convergence; finding the spectral decomposition of $A_\text{hom,2sr}$; and turning the qualitative arguments into quantitative ones. We will not pursue that route here.

Returning to norm-resolvent asymptotics, one might wonder the sense in which the norm resolvent asymptotics $A_\text{hom}$, obtained from \cite{eff_behavior}, provides a simplified description of the high-contrast composite. We attempt to provide an answer in the following context: just as how we may study the dispersion relation of a periodic operator, we could also ask for the limiting dispersion relation of $A_\eps$. We will do so by taking a closer look at the non-local part of $A_{\text{hom}}^{(\tau)}$, in particular at how the constant of multiplication $c_\tau$ in the $1D$ subspace of $L^2(Q_\st)$ depends on $\tau$. The key object that is extracted from this study is referred to as the ``dispersion function" $K(\tau,z)$. As shown in \cite[Section 5]{eff_behavior}, $K(\tau,z)$ are very different for Models I and II. In Section \ref{sect:closer_look}, we will derive $K(\tau,z)$ for the stiff-soft-stiff model, and compare it with $K(\tau,z)$ of Models I and II.

\subsection*{Main results}

The main results of this paper are collected in Sections \ref{sect:homo_result} and \ref{sect:closer_look}. These are as follows:
\begin{itemize}
    \item A homogenization result for the composite material in Figure \ref{fig:intro_toymodel}. This is Theorem \ref{thm:fibre_homo_result}, stated in terms of norm-resolvent asymptotics for the operator $A_\eps$.
    
    \item Derivation and justification of the existence of ``dispersion functions", analogous to that of \cite{eff_behavior}. These are the objects $K_\stin(\uptau,z)$ (Theorem \ref{thm:stiff_int_dispersion_reln}) and $K_\stls(\uptau,z)$ (Theorem \ref{thm:stiff_ls_dispersion_reln}).
\end{itemize}

\subsection*{Structure of paper}
This paper is structured as follows: In Section \ref{sect:preliminaries_main}, we fix some notations and introduce the notion of a periodic Sobolev space, following \cite{cioranescu_donato}. Sections \ref{sect:mainmodel} and \ref{sect:bdry_triple_setup} involve setting up the ``stiff-soft-stiff" problem. In Section \ref{sect:mainmodel}, we define the operator $A_\eps$ on $L^2(\mathbb{R}^d)$ and then explain why we can equivalently study the family of operators $\{ A_\eps^{(\tau)} \}_{\tau \in [-\pi,\pi)^d}$ on $L^2(Q)$, obtained by the Gelfand transform. Section \ref{sect:bdry_triple_setup} further casts the problem in the language of boundary triples. 

Section \ref{sect:norm_resolvent_asymp} studies the resolvent asymptotics of $A_\eps^{(\tau)}$, and is at the heart of the analysis. Section \ref{sect:identifying_homo_op} combines the result of Section \ref{sect:norm_resolvent_asymp} (Theorem \ref{thm:m_inverse_est}) with the boundary triple setup of Section \ref{sect:bdry_triple_setup} to give a self-adjoint operator $\mathcal{A}_{\eps,\text{hom}}^{(\tau)}$ that captures the norm-resolvent asymptotics of $A_\eps^{(\tau)}$. Section \ref{sect:homo_result} unpacks the notation and summarizes the boundary triple approach for homogenization, yielding the main result of the paper, Theorem \ref{thm:fibre_homo_result}.

Section \ref{sect:closer_look} places focus on the limiting operator $\mathcal{A}_{\eps,\text{hom}}^{(\tau)}$ itself. We derive the dispersion functions $K_\stin(\uptau,z)$ and $K_\stls(\uptau,z)$, as a first step towards studying wave propagation on the ``stiff-soft-stiff" composite.

\section{Preliminaries}\label{sect:preliminaries_main}
\subsection{Notation, assumptions, abbreviations}

Fix the dimension $d\geq 2$. $\mathbb{N} = \{1,2,3,\cdots \}$ and $\mathbb{N}_0 = \{ 0 \} \cup \mathbb{N}$. The indicator function of a set $U \subset \mathbb{R}^d$ is denoted by $\mathbf{1}_U$. $\oplus$ refers to an orthogonal sum of Hilbert spaces, or of operators on Hilbert spaces. $\dot{+}$ refers to a direct sum of vector spaces. For $a,b\in \R^d$, we write $a\cdot b = \sum_{i=1}^d a_ib_i$ for the inner product on $\R^d$, and $|a| = \sqrt{a\cdot a}$ for the corresponding norm. For $c \in \C$, the real and imaginary components of $a$ are denoted by $\text{Re}(a)$ and $\text{Im}(a)$ respectively.

\textit{Spaces.} We will assume that our Hilbert spaces $(\mathcal{H}, (\cdot,\cdot)_{\mathcal{H}})$ are complex, and write $\| \cdot \|_{\mathcal{H}}$ for its corresponding norm. Let $U \subset \mathbb{R}^d$ be open. Denote by $C^{\infty}(U)$ the space of smooth functions $f:U\rightarrow \mathbb{C}$, and $C_c^{\infty}(U)$ for the vector subspace of functions that have compact support in $U$. For $k\in\mathbb{N}_0$ and $p\in [1,\infty]$, we will need $L^p(U)$, $L^p_{loc}(U)$, $W^{k,p}(U)$, $W_0^{k,p}(U) := \overline{C_c^{\infty}(U)}^{W^{k,p}}$, with the important special cases being $H_0^1(U) := W_0^{1,2}(U)$, and $H^k(U) := W^{k,2}(U)$. We will also need \textit{periodic function spaces}, which are defined in Section \ref{sect:periodic_sobolev}.

\textit{Operators.} We will mainly follow the notations of \cite{konrad_book}. Let $\mathcal{H}_1$, $\mathcal{H}_1$ and $\mathcal{H}_2$ be Hilbert spaces. By an (unbounded) operator $T$ from $\mathcal{H}_1$ to $\mathcal{H}_2$, we mean a linear mapping $T:\mathcal{D}(T) \rightarrow \mathcal{H}_2$, where $\mathcal{D}(T)$ is a linear subspace of $\mathcal{H}_1$. The set $\mathcal{D}(T)$ is referred to as the domain of $T$, and we also write $(T,\mathcal{D}(T))$ to mean the operator $T$, whenever we would like to place an emphasis on the domain. $S \subset T$ means that $T$ is an extension of $S$. If $(T,\mathcal{D}(T))$ is an operator on $\mathcal{H}$ (i.e.~from $\mathcal{H}$ to $\mathcal{H}$), the spectrum of $T$ is denoted by $\sigma(T)$, and the resolvent set by $\rho(T)$. For $\lambda \in \rho(T)$, the resolvent $(T-\lambda I)^{-1}$ will be abbreviated as $(T-\lambda)^{-1}$. $\mathcal{L}(\mathcal{H}_1,\mathcal{H}_2)$ denotes the space of bounded linear operators from $\mathcal{H}_1$ and $\mathcal{H}_2$, and $\mathcal{L}(\mathcal{H}) := \mathcal{L}(\mathcal{H},\mathcal{H})$. The operator norm of $T \in \mathcal{L}(\mathcal{H}_1,\mathcal{H}_2)$ is denoted either by $\| T \|_{\mathcal{H}_1 \rightarrow \mathcal{H}_2} $, or by $\| T \|_{op}$ if the spaces are clear from the context.

\textit{Special families of operators.} Let $\uptau \in \mathbb{R}^d$. The operator $-(\nabla + i\tau)^2$ or $(\frac{1}{i}\nabla + \uptau)^2$ (with appropriately defined boundary conditions) are both shorthand for $-\Delta - 2i\tau \cdot \nabla + |\tau|^2$, as opposed to a composition of operators. The multiplication operator on $L^2(\R^d)$ by an almost everywhere finite function $f$ is denoted by $M_f$. Similarly, multiplication on $\C$ by a constant $c$ is denoted by $M_c$. Our operators of interest are typically defined through a sesquilinear form. If $(A,\mathcal{D}(A))$ is constructed from a form $(\mathfrak{t}, \mathcal{D}(\mathfrak{t}) )$, then we will write $\mathcal{D}[A] := \mathcal{D}(\mathfrak{t})$ to distinguish between the form domain and the operator domain. All our projections will be orthogonal. If $H$ is a subspace of $\mathcal{H}$, then the projection onto $H$ will be denoted either by $P_H$ or $\mathcal{P}_H$. 

\textit{Conventions.} We will be dealing with a multitude of projections on two Hilbert spaces, $\mathcal{H}$ and $\mathcal{E}$. The straight font (e.g.~$P_H$) is reserved for projections on $\mathcal{H}$, and the calligraphic font (e.g.~$\mathcal{P}_E$) is reserved for projections on $\mathcal{E}$. We omit the differential $``dx"$ in integrals, where it is understood.

\textit{Abbreviations.} We will be using the following abbreviations: LHS/RHS for left/right hand side (of an equation); BVP/BC for boundary value problem/boundary condition.

\subsection{Periodic Sobolev spaces}\label{sect:periodic_sobolev}
Fix a reference cell $Q = [0,1)^d$.
\begin{defn}
    A function $f$, defined a.e.~on $\mathbb{R}^d$ is called $\mathbb{Z}^d$-periodic if for all $k\in\mathbb{Z}$ and $i\in \{1,\cdots,d\}$, we have $f(x+k e_i) = f(x)$ a.e. Here $\{e_1,\cdots,e_d \}$ denotes the standard basis of $\mathbb{R}^d$.
\end{defn}

We will also require a notion of periodicity up to and including the boundary $\partial Q$. Since we want to talk about traces of measurable functions on $Q$, we need at least one weak derivative. This prompts us to make the following definition:
\begin{defn}
    $C^{\infty}_{per}(Q) := \{u\in C^{\infty}(\mathbb{R}^d) : u \text{ is $\Z^d$-periodic}  \}$. We will identify $u\in C^{\infty}_{per}(Q)$ with its restriction to $\overline{Q}$.
\end{defn}

The key definition of this section is the following Hilbert space:

\begin{defn} 
    $H^1_{per}(Q) := \overline{C^{\infty}_{per}(Q)}^{\| \cdot \|_{H^1(Q)} }$. We identify space as a subspace of $L^2(Q)$.
\end{defn}

We list here several equivalent characterizations of $H^1_{per}(Q)$:
\begin{alignat}{2}
    H^1_{per}(Q) &= \{ u \in H^1_{loc}(\mathbb{R}^d) : u \text{ is $\Z^d$-periodic} \} \span\span \\
    &= \{ u\in L^2(Q) :\, &&\partial_i u \in L^2(Q), \text{ and $u$, $\partial_i u$ have equal trace on} \nonumber \\
    & &&\text{opposite faces of $Q$, $1\leq i \leq d$.}\} \\
    &= \left\{ u \in L^2(Q) : \sum_{k\in\mathbb{Z}^d} (1+|k|^2) |\hat{u}(k)|^2 < \infty \right\}, \span\span
\end{alignat}
where $\hat{u}$ denotes the Fourier transform of $u$. For an explanation of the equalities, we refer to \cite[p.~6]{zhikov} and \cite[Proposition 3.50]{cioranescu_donato} for the first, \cite[Proposition 3.49]{cioranescu_donato} for the second, and \cite[p.~137]{david_borthwick} for the third expression. Note that $H^1_{per}(Q) = \overline{C^{\infty}_{per}(Q)}^{\| \cdot \|_{H^1(Q)} }$ is in general a subspace of $H^1(Q)$, and $\overline{C^{\infty}_{per}(Q)}^{\| \cdot \|_{L^2(Q)} } = L^2(Q)$.

\begin{rmk}
    We will sometimes rely on the compactness of $\overline{Q}$ to make our arguments. Notably, this is used in Proposition \ref{prop:m_components_unif_bound}, Theorem \ref{thm:m_inverse_est}, and Proposition \ref{prop:cts_dependence}.
\end{rmk}

\section{Problem formulation}\label{sect:mainmodel}

This section is structured as follows: In Section \ref{sect:mainmodel_aeps}, we define the operator $A_\eps$ on $L^2(\mathbb{R}^d)$. Next, we introduce the scaled Gelfand transform $G_\eps$ in Section \ref{sect:scalings}, and use it to obtain a family of operators $A_\eps^{(\tau)}$ on $L^2(Q)$ indexed by $\tau \in [-\pi,\pi)^d$. With the help of $G_\eps$, our study of the norm-resolvent asymptotics of $A_\eps$ can be restated in terms of the family $A_\eps^{(\tau)}$. This allows to reformulate our problem in terms of $A_\eps^{(\tau)}$, which we will do in Section \ref{sect:mainmodel_defn}. 

The operator $A_\eps^{(\tau)}$ will be the main object of study in the remainder of the text. We will refer to the setup in Section \ref{sect:mainmodel_defn} as the ``main model", and $A_\eps^{(\tau)}$ as the ``main model operator".

\subsection{Operator on the full space}\label{sect:mainmodel_aeps}

In this section we will define the operator $A_\eps$. On the reference cell $Q=[0,1)^d$, consider the setup as shown in Figure \ref{fig:intro_toymodel}. That is, $Q$ is split into three connected components: a simply connected ``stiff interior" part $Q_\stin$, surrounded by an annular ``soft" region $Q_\soft$, with the remaining region filled by the ``stiff landscape" part $Q_\stls$. We require that the boundaries $\Gamma_\interior$ and $\Gamma_\ls$ are smooth, and $\Gamma_\interior$, $\Gamma_\ls$, and $\partial Q$ are of positive distance from each other.

Recall from \eqref{eqn:intro_coeffmatrix} that our coefficient matrix $\widetilde{a}_{\eps^2}:\mathbb{R}^d \rightarrow \mathbb{R}^{d\times d}$ is given by
\begin{align}
    \widetilde{a}_{\eps^2}(y) = \begin{cases}
        \eps^2 I, \quad & y \in \cup_{n \in \mathbb{Z}^d} ( Q_\soft + n ),\\
        I, & y \in \cup_{n \in \mathbb{Z}^d} \left( (Q_\stin \cup Q_\stls) + n \right),
    \end{cases}
\end{align}
where $Q_\soft + n = \{ y + n : y \in Q_\soft \}$, and similarly for $(Q_\stin \cup Q_\stls) + n$. The matrix $\widetilde{a}_{\eps^2}$ is $\mathbb{Z}^d$-periodic, and thus the matrix $a_{\eps} := \widetilde{a}_{\eps^2}\left(\frac{\cdot}{\eps}\right)$ is $\eps\mathbb{Z}^d$-periodic. 

The operator $A_{\eps} \equiv -\text{div}(a_{\eps}\nabla\cdot)$ on $L^2(\mathbb{R}^d)$ is defined through its sesquilinear form:
\begin{equation}
    (u,v) \mapsto \int_{\mathbb{R}^d} \widetilde{a}_{\eps^2}\left(\tfrac{\tilde{x}}{\eps}\right) \nabla u(\tilde{x}) \cdot \overline{\nabla v(\tilde{x})} \text{ }d\tilde{x}, \quad u,v \in \mathcal{D}[A_\eps] :=  H^1(\mathbb{R}^d).
\end{equation}
Let us emphasize that $A_\eps$ is not uniformly strongly elliptic, in the sense that the coefficient matrices $\widetilde{a}_{\eps^2}(\frac{\cdot}{\eps})$ cannot be bounded away from zero, independently of $\eps$.

If $z \in \rho(A_\eps)$, then the resolvent equation
\begin{align}
    -\Div(\widetilde{a}_{\eps^2} \left(\tfrac{\cdot}{\eps}\right) \nabla u_\eps) - z u_\eps = f, \quad f \in L^2(\mathbb{R}^d), \quad z \in \mathbb{C},
\end{align}
has a unique solution $u_{\varepsilon}$, which can be written as $u_{\eps} = (A_{\eps}-z)^{-1} f$. In terms of the weak formulation, the resolvent equation is given by:
\begin{equation}
    \int_{\mathbb{R}^d} \bigg[ \widetilde{a}_{\eps^2} \left( \tfrac{\tilde{x}}{\eps} \right) \nabla u(\tilde{x}) \cdot \overline{\nabla v(\tilde{x})} - z u(\tilde{x}) \overline{v(\tilde{x})} \bigg] d\tilde{x} = \int_{\mathbb{R}^d} f(\tilde{x})\overline{v(\tilde{x})} \text{ }d\tilde{x}, \quad \text{ for all } v \in H^1(\mathbb{R}^d). 
\end{equation}

\subsection{Passing from the full space to the unit cell}\label{sect:scalings}

Let $Q' = [-\pi,\pi)^d$. It is customary in the study of periodic differential operators (\textit{Floquet Theory}, see e.g.~\cite[Section~XIII.16]{reed_simon4}) to begin the analysis of a $\mathbb{Z}^d$-periodic operator $T$ by applying the Gelfand transformation to $T$, giving us family of operators $T^{(\uptau)}$, $\uptau \in Q'$. Let us now summarize the necessary elements from Floquet theory that will be of use here.

First, it would be more convenient to introduce a scaled version of the Gelfand transform since $A_\varepsilon$ is $\eps \mathbb{Z}^d$-periodic rather than $\mathbb{Z}^d$-periodic:

\begin{defn} 
    The \textit{scaled Gelfand transform} is the operator $G_\varepsilon$ defined first for $u\in C^{\infty}_c(\mathbb{R}^d)$ by the formula
    \begin{align}
        (G_\varepsilon u) (\tilde{x},\theta) := \left( \frac{\varepsilon}{2\pi}\right)^{d/2} \sum_{n\in \mathbb{Z}^d} u(\tilde{x} + \varepsilon n) e^{-i\theta \cdot (\tilde{x} + \varepsilon n)}, \qquad \tilde{x} \in \varepsilon Q \text{, } \theta \in \varepsilon^{-1}Q',
    \end{align}
    and extended by continuity to an operator from $L^2(\mathbb{R}^d)$ to $L^2(\varepsilon Q \times \varepsilon^{-1} Q')$, which will we still denote by $G_\eps$.
\end{defn}

\begin{rmk}
    In fact, $G_\varepsilon$ is unitary, with the following inversion formula:
    \begin{align}
        u(\tilde{x}) = \left( \frac{\varepsilon}{2\pi}\right)^{d/2} \int_{\varepsilon^{-1} Q'} (G_\varepsilon u) (\tilde{x}, \theta) e^{i\theta \cdot \tilde{x}} d\theta, \qquad \tilde{x}\in \mathbb{R}^d,
    \end{align}
    where we have extended $G_\varepsilon u$ by $\varepsilon Q$-periodicity to a function on $\mathbb{R}^d\times \varepsilon^{-1} Q'$.
\end{rmk}

Now introduce a new notation for the Bochner spaces $L^2(M,\mu;\mathcal{H}')$, following \cite[Section VIII.16]{reed_simon4}:

\begin{defn}\label{defn:directintspace}
    Let $(M, \mu)$ be a $\sigma$-finite measure space, and $\mathcal{H}'$ a separable Hilbert space. We define the \textit{(constant fiber) direct integral space} $\mathcal{H} = \int_{M}^{\oplus} \mathcal{H}' d\mu(m)$ to be the Bochner space $L^2(M,\mu; \mathcal{H}')$. Recall that this is a Hilbert space, equipped with inner product
    \begin{align}
        ( s,t )_{\mathcal{H}} := \int_M \left( s(m), t(m) \right)_{\mathcal{H}'} d\mu(m).
    \end{align}
    Elements of this space $s\in \mathcal{H}$ are called \textit{(measurable cross-)sections}.
\end{defn}

This notation places emphasis on the fibers $\mathcal{H}'$, and therefore on operators on $\mathcal{H}'$ indexed by the set $(M,\mu)$, in a measurable way. This requires us to define a notion of measurability. We continue with the notation of Definition \ref{defn:directintspace} for the remaining definitions of this section:

\begin{defn}
    We say that $T(\cdot): M \rightarrow \mathcal{L}{(\mathcal{H'})}$ is measurable if for all $x,y\in \mathcal{H}'$, the mapping
    \begin{align}
        M \ni m \mapsto \left(x, T(m) y \right)_{\mathcal{H}'} \in \mathbb{C} 
    \end{align}
    is measurable.
\end{defn}

To work with unbounded operators, the following definition will be neessary:

\begin{defn}
    Let $\{ A(m) \}_{m \in M}$ be a collection of unbounded self-adjoint operators on $\mathcal{H}'$, we say that $A(\cdot)$ is measurable if the mapping
    \begin{align}
        M \ni m \mapsto (A(m) + i)^{-1} \in \mathcal{L}(\mathcal{H}')
    \end{align}
    is measurable.
\end{defn}

We are now ready to introduce the notion of a ``continuous direct sum of bounded operators". 

\begin{defn}[Decomposable operator] Let $T$ be a bounded operator on $\mathcal{H} = \int_M^{\oplus} \mathcal{H}' d\mu$. Suppose there exists a measurable family $T(\cdot) \in L^{\infty}(X,\mu;\mathcal{L}(\mathcal{H}'))$ such that for all sections $s\in \mathcal{H}$,
    \begin{align}
        (Ts)(x) = T(x) s(x).
    \end{align}
    Then we call $T$ \textit{decomposable}, write $T = \int_M^{\oplus} T(m)d\mu(m)$, and call $T(m)$ the \textit{fibers of $T$.}
\end{defn}

Similarly, we will need to extend this notion to the unbounded self-adjoint case:

\begin{defn}
    Suppose $A(\cdot)$ is a measurable family of unbounded self-adjoint operators on $\mathcal{H}'$. We define the operator $A \equiv \int_M^{\oplus} A(m) d\mu(m)$ by
    \begin{align}
        &\mathcal{D}(A) = \left\{ u\in\mathcal{H} ~\bigg|~ u(m)\in \mathcal{D}\left(A(m)\right) ~\text{$\mu$-a.e.,~with}~\int_M \| A(m)u(m) \|_{\mathcal{H}'}^2 d\mu(m) < \infty \right\},\\
        &(Au)(m) = A(m) u(m).
    \end{align}
    This is an unbounded self-adjoint operator, by \cite[Theorem XIII.85(a)]{reed_simon4}.
\end{defn}

\begin{rmk}
    With the notation of a direct integral, the Gelfand transform may now be written as
    \begin{align*}
        &G_\eps:L^2(\mathbb{R}^d) \longrightarrow L^2(\eps Q \times \eps^{-1}Q') \cong  \int_{\eps^{-1}Q'}^{\oplus} L^2(\eps Q) d\theta, \\
        &(G_\eps u)(\theta) = (G_\eps u)(\cdot,\theta) = G_\eps (\theta) u (\theta) \in L^2(\eps Q).
    \end{align*}
    Note also the special case $\int_{\varepsilon^{-1}Q'}^{\oplus} \mathbb{C} d\theta \cong L^2(\eps^{-1}Q')$.
\end{rmk}

\begin{defn}\label{defn:a_eps_theta}
    For $\theta \in \eps^{-1}Q'$, define $A^{(\theta)}_\eps$ to be the operator on $L^2(\eps Q)$ corresponding to the sesquilinear form
    \begin{align}
        (u,v) \mapsto \int_{\eps Q} \widetilde{a}_{\eps^2}\left( \tfrac{\tilde{x}}{\eps} \right) (\tfrac{1}{i} \nabla_{\tilde{x}} + \theta) u(\tilde{x}) \cdot \overline{ (\tfrac{1}{i} \nabla_{\tilde{x}} + \theta) v(\tilde{x})} d\tilde{x}, \quad u,v \in \mathcal{D}[A^{(\theta)}_\eps] := H_{\text{per}}^1(\eps Q).
    \end{align}
    That is, $A^{(\theta)}_\eps$ corresponds to the differential expression $(\frac{1}{i} \nabla_{\tilde{x}} + \theta) \widetilde{a}_{\eps^2}(\frac{\tilde{x}}{\varepsilon}) (\frac{1}{i} \nabla_{\tilde{x}} + \theta)$ with periodic BCs on $\eps Q$. 
\end{defn}

Since $A_\eps$ (from Section \ref{sect:mainmodel_aeps}) has $\eps \mathbb{Z}^d$-periodic coefficients, the scaled Gelfand transform sets up a unitary equivalence between $A_\eps$ and a family of operators $A^{(\theta)}_\eps$:

\begin{prop} \label{prop:a_eps_theta_defn}
    With $A_\eps$ as defined in Section \ref{sect:mainmodel_aeps} and $A_\eps^{(\theta)}$ as in Definition \ref{defn:a_eps_theta}, we have the following identity:
    \begin{align}
        A_\varepsilon = G_\varepsilon^{*} \left( \int_{\varepsilon^{-1}Q'}^{\oplus} A_\varepsilon^{(\theta)} d \theta \right) G_\varepsilon.
    \end{align}
\end{prop}

\begin{proof}
    This is a direct consequence of the product rule, see for example \cite[Theorem 2.5]{muthukumar} for the short computation. Periodic BCs follows from the fact that $G_\varepsilon u(\tilde{x},\theta)$ is $\eps\Z^d$-periodic in $\tilde{x}$.
\end{proof}

While we have shifted our perspective to consider a ($\theta$ dependent) operator on a bounded subset of $\mathbb{R}^d$, this is rather inconvenient as the space $L^2(\varepsilon Q)$ varies with $\varepsilon$. This motivates us to define:

\begin{defn}
    For each $\varepsilon>0$, define the unitary rescaling operators $\Phi_\eps$ and $\Psi_\eps$
    \begin{align}
        &\Phi_\eps: L^2(\eps Q) \rightarrow L^2(Q), 
        &&(\Phi_\eps u)(x) = \eps^{d/2} u(\eps x), \\
        &\Psi_\eps : L^2(\eps^{-1}Q') \rightarrow L^2(Q'), 
        &&(\Psi_\eps v) (\tau) = \left( \tfrac{2\pi}{\eps} \right)^{d/2} v(\tfrac{\tau}{\eps}).
    \end{align}
\end{defn}

\begin{defn}\label{defn:a_eps_tau}
    For $\tau \in Q'$, define $A_\eps^{(\tau)}$ to be the operator on $L^2(Q)$ corresponding to the sesquilinear form
    \begin{align}
        (u,v) \mapsto \frac{1}{\eps^2} \int_{Q} \widetilde{a}_{\eps^2}\left( x \right) (\tfrac{1}{i} \nabla_{x} + \theta) u(x) \cdot \overline{ (\tfrac{1}{i} \nabla_{x} + \theta) v(x)} dx, \quad u,v \in \mathcal{D}[A^{(\tau)}_\eps] := H_{\text{per}}^1(Q).
    \end{align}
    That is, $A^{(\tau)}_\eps$ corresponds to the differential expression $(\frac{1}{i} \nabla_x + \uptau) \frac{1}{\varepsilon^2} \widetilde{a}_{\eps^2}(x) (\frac{1}{i} \nabla_x + \uptau)$ with periodic BCs on $Q$. We will refer to $A_\eps^{(\tau)}$ as the \textit{main model operator}.
\end{defn}

\begin{lem} \label{lem:rescaling}
    Let $\tau = \eps \theta$. Then, $A^{(\uptau)}_{\varepsilon} = \Phi_\varepsilon A^{(\theta)}_\varepsilon \Phi_\varepsilon^*$.
\end{lem}

\begin{proof}
    Equivalently, we need to show that $A^{(\uptau)}_{\varepsilon} \Phi_\varepsilon = \Phi_\varepsilon A^{(\theta)}_\varepsilon$ as an operator from $L^2(\varepsilon Q)$ to $L^2(Q)$. It suffices to check this on a form core $C^{\infty}_\text{per}(\varepsilon Q)$. Let $u \in C^{\infty}_\text{per}(\varepsilon Q)$. We use $\tilde{x}$ for the variable on $\eps Q$, and $x$ for the variable on $Q$. First we see that
    \begin{align}
        \left( A^{(\theta)}_\varepsilon u \right) (\tilde{x}) &= \left( \frac{1}{i} \nabla_{\tilde{x}} + \theta \right) \cdot \left( \widetilde{a}_\varepsilon\left(\frac{\tilde{x}}{\varepsilon}\right) \left(\frac{1}{i} \nabla_{\tilde{x}} + \theta\right) u(\tilde{x}) \right) \nonumber \\
        &= \sum_{j,k=1}^d \left( \frac{1}{i}\frac{\partial}{\partial\tilde{x}_j} + \theta_j \right) \left[ \widetilde{a}_\varepsilon^{jk}\left( \frac{\tilde{x}}{\varepsilon}\right) \left( \frac{1}{i} \frac{\partial}{\partial \tilde{x}_k} + \theta_k \right) u(\tilde{x}) \right]
    \end{align}
    Therefore, if $\tau = \eps \theta$ and $x = \frac{\tilde{x}}{\eps}$, then
    \begin{align*}
        RHS &= \left( \Phi_\varepsilon A^{(\theta)}_\varepsilon u \right) (x) = \varepsilon^{d/2} \left(A^{(\theta)}_\varepsilon u \right) (\varepsilon x) \\
        &= \varepsilon^{d/2} \sum_{j,k=1}^d \left( \frac{1}{i} \frac{\partial}{\partial \tilde{x}_j} + \theta_j \right) \left[ \widetilde{a}_{\eps^2}^{jk}\left( \frac{\varepsilon x}{\varepsilon}\right) \left( \frac{1}{i} \frac{\partial}{\partial \tilde{x}_k} + \theta_k \right) u(\varepsilon x) \right] \\
        &= \varepsilon^{d/2} \sum_{j,k=1}^d \left( \frac{1}{i\varepsilon} \frac{\partial}{\partial x_j} + \frac{\uptau_j}{\varepsilon} \right) \left[ \widetilde{a}_{\eps^2}^{jk}\left( x \right) \left( \frac{1}{i\varepsilon} \frac{\partial}{\partial x_k} + \frac{\uptau_k}{\varepsilon} \right) u(\varepsilon x) \right] \\
        &= \varepsilon^{d/2} \sum_{j,k=1}^d \left( \frac{1}{i} \frac{\partial}{\partial x_j}  + \uptau_j \right) \left[ \frac{1}{\varepsilon^2} \widetilde{a}^{jk}_{\eps^2}(x) \left( \frac{1}{i} \frac{\partial}{\partial x_k} + \uptau_k \right) u(\varepsilon x)  \right] \\
        &= \varepsilon^{d/2} \left( \frac{1}{i} \nabla_x + \uptau  \right) \cdot \left( \frac{1}{\varepsilon^2} \widetilde{a}_{\eps^2}(x) \left( \frac{1}{i} \nabla_x + \uptau \right) u(\varepsilon x) \right) \\
        &= \left( \frac{1}{i} \nabla_x + \uptau  \right) \cdot \left( \frac{1}{\varepsilon^2} \widetilde{a}_{\eps^2}(x) \left( \frac{1}{i}\nabla_x + \uptau \right) (\Phi_\varepsilon u)(x) \right)
        = \left( A^{(\uptau)}_\varepsilon \Phi_\varepsilon u \right) (x) = LHS. \qedhere
    \end{align*}

\end{proof}

\begin{cor}\label{cor:unitary_equiv}
    $A_\eps$ is unitarily equivalent to $\int_{Q'}^{\oplus} A_\eps^{(\tau)}d\tau$.
\end{cor}

\begin{proof}
    Using expressions relating $A_\eps$, $A_\eps^{(\theta)}$, and $A_\eps^{(\tau)}$, $\tau = \eps\theta$, we have
    \begin{align}
        A_\varepsilon &\stackrel{\text{Prop \ref{prop:a_eps_theta_defn}}}{=} G_\varepsilon^{*} \left( \int_{\varepsilon^{-1}Q'}^{\oplus} A_\varepsilon^{(\theta)} d \theta \right) G_\varepsilon
        \stackrel{\text{Lemma \ref{lem:rescaling}}}{=} G_\varepsilon^{*} \left( \int_{\varepsilon^{-1}Q'}^{\oplus} \Phi_\varepsilon^{*} A_\varepsilon^{(\varepsilon\theta)} \Phi_\varepsilon d \theta \right) G_\varepsilon \nonumber\\
        &= G_\varepsilon^{*} \left( \int_{\varepsilon^{-1} Q'}^{\oplus} \Phi_\varepsilon^* d\theta \right)
        \left( \int_{\eps^{-1}Q'}^{\oplus} A^{(\eps\theta)}_\varepsilon d\theta \right)
        \left( \int_{\varepsilon^{-1} Q'}^{\oplus} \Phi_\varepsilon d\theta \right)G_\varepsilon \nonumber \\
        &= G_\varepsilon^{*} \left( \int_{\varepsilon^{-1} Q'}^{\oplus} \Phi_\varepsilon^* d\theta \right)
        \left( \int_{Q}^{\oplus} \Psi_\eps^* dx \right)
        \left( \int_{Q'}^{\oplus} A^{(\tau)}_\varepsilon d\tau \right)
        \left( \int_{Q}^{\oplus} \Psi_\eps dx \right)
        \left( \int_{\varepsilon^{-1} Q'}^{\oplus} \Phi_\varepsilon d\theta \right)G_\varepsilon,
    \end{align}
    where in the last equality, we have used $\int_{\eps^{-1}Q'}^{\oplus} L^2(Q) d\theta \cong L^2(Q\times \eps^{-1}Q') \cong \int_Q^{\oplus} L^2(\eps^{-1}Q') dx$. Note for instance, that $\| \int_{\eps^{-1}Q'}^{\oplus} \Phi_\eps d\theta \|_{op} = \text{esssup}_\theta \| \Phi_\eps \|_{L^2(\eps Q) \to L^2(Q)} = 1$ \cite[Theorem~XIII.83]{reed_simon4}.
\end{proof}

We will therefore turn our attention towards the operator $A_\varepsilon^{(\uptau)}$. In the following section, we recast our problem of studying the norm-resolvent asymptotics of $A_\eps$ in terms of $A_\varepsilon^{(\uptau)}$.

\subsection{Reformulation in terms of operators on the unit cell}\label{sect:mainmodel_defn}

With Corollary \ref{cor:unitary_equiv} in mind, our goal can now be restated as follows:

\begin{quote}
    Identify, \textbf{uniform in $\tau$}, \textbf{norm-resolvent asymptotics} for $A_\eps^{(\tau)}$, as $\eps \downarrow 0$.
\end{quote}

Having turned our focus towards $A_\eps^{(\tau)}$, let us collect several ways of describing $A_\eps^{(\tau)}$ that will be useful for purposes of interpretation.

First, we recall from Definition \ref{defn:a_eps_tau} that $A_\eps^{(\tau)}$ is an operator on $L^2(Q)$ that corresponds to the differential expression $(\frac{1}{i} \nabla_x + \uptau) \frac{1}{\varepsilon^2} \widetilde{a}_{\eps^2}(x) (\frac{1}{i} \nabla_x + \uptau)$. Recall that the coefficient matrix is given by:
\begin{equation}\label{eqn:matrix_coeff}
    \frac{1}{\varepsilon^2} \widetilde{a}_{\varepsilon^2}(x) = 
    \begin{cases}
        \eps^{-2} I, \quad &x \in Q_{\text{stiff-ls}},\\
        I, \quad &x \in Q_{\text{soft}},\\
        \eps^{-2} I, \quad &x \in Q_{\text{stiff-int}},
    \end{cases}
\end{equation}
where the subscripts ``ls" and ``int" stands for landscape and interior respectively. 

Second, we have obtained $A_\eps^{(\tau)}$ from $A_\eps$ through a combination of the Gelfand transform $G_\eps$ and rescaling $\Phi_\eps$ in the previous subsection. Figure \ref{fig:toymodel} gives a description of this process, where we pass from the full space $\R^d$ to the unit cell $Q$.
\begin{figure}[h]
  \centering
  \includegraphics[page=4, clip, trim=1cm 4.5cm 5cm 2.5cm, width=1.00\textwidth]{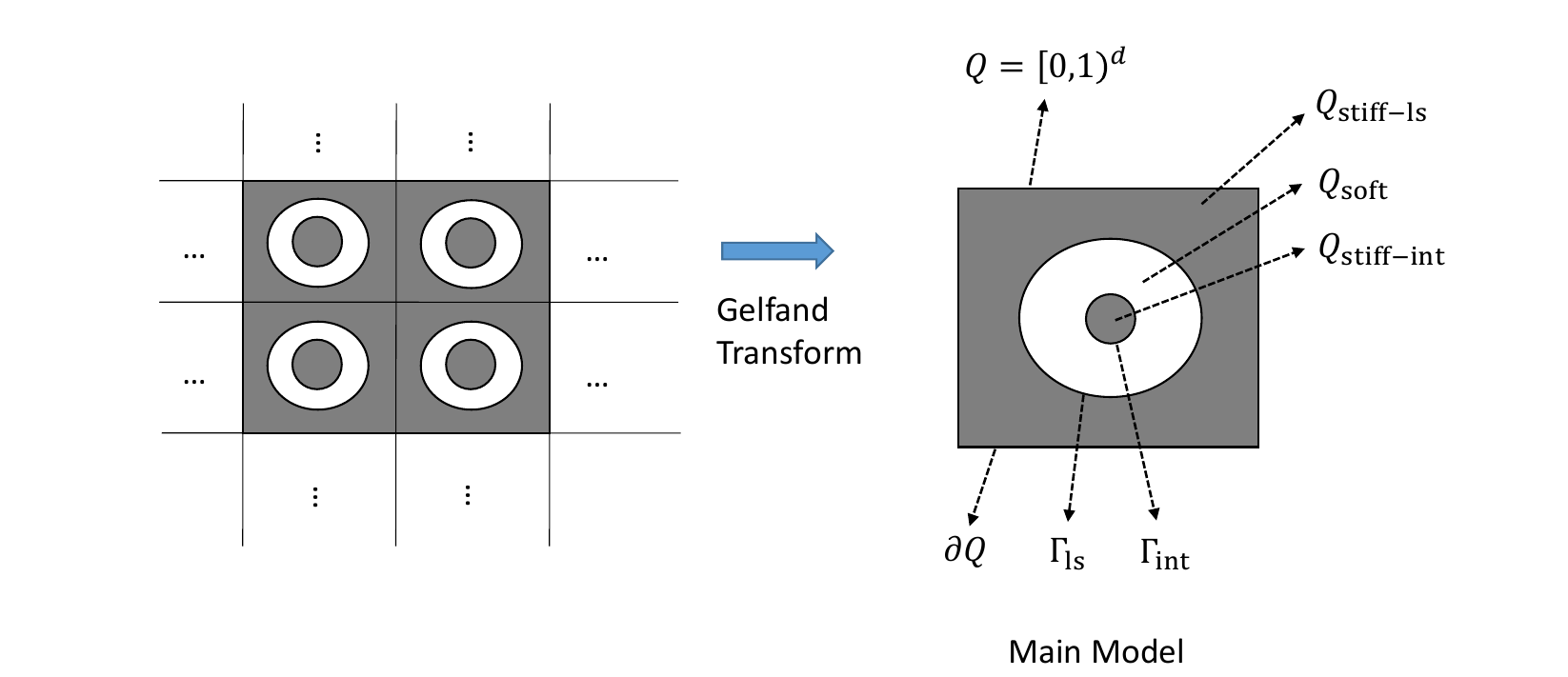}
  \caption{Obtaining the main model operator $A_\eps^{(\tau)}$ via Gelfand transform.} \label{fig:toymodel}
\end{figure}

Third, recall from Section \ref{sect:intro} that $A_\eps$ (and thus $A_\eps^{(\tau)}$) are defined with transmission BCs on the soft-stiff interfaces. That Definition \ref{defn:a_eps_tau} implies transmission BCs on $\Gamma_\ls$ and $\Gamma_\interior$ can be seen from the form domain $\mathcal{D}[A^{(\uptau)}_\varepsilon] = H^1_{\text{per}}(Q)$. Alternatively, we may write the BVP for the resolvent equation for $A_\eps^{(\tau)}$: The equation $(A_{\eps}^{(\tau)} - z)u = f \in L^2(Q)$ has a unique solution $u\equiv u_{\eps}^{(\tau)} = u_\stls + u_\soft + u_\stin$ whenever the following BVP can be solved uniquely in the weak sense:
\begin{equation}\label{eqn:mainmodel_bvp}
    \begin{cases}
        \varepsilon^{-2} \bigg(\frac{1}{i} \nabla + \uptau \bigg)^2 u_{\text{stiff-ls}} - zu_{\text{stiff-ls}} = f, \quad &\text{ in } Q_{\text{stiff-ls}},\\
        \bigg(\frac{1}{i} \nabla + \uptau \bigg)^2 u_{\text{soft}} - zu_{\text{soft}} = f, \quad &\text{ in } Q_{\text{soft}},\\
        \varepsilon^{-2} \bigg(\frac{1}{i} \nabla + \uptau \bigg)^2 u_{\text{stiff-int}} - zu_{\text{stiff-int}} = f, \quad &\text{ in } Q_{\text{stiff-int}}, \\
        u_{\text{stiff-ls}}=u_{\text{soft}} \quad &\text{ on } \Gamma_\ls, \\
        \varepsilon^{-2} \bigg[ \frac{\partial u_{\text{stiff-ls}}}{\partial n_{\text{stiff-ls}}} + i (\uptau \cdot n_{\text{stiff-ls}}) u_{\text{stiff-ls}} \bigg] + \bigg[ \frac{\partial u_{\text{soft}}}{\partial n_{\text{soft}}} + i (\uptau \cdot n_{\text{soft}}) u_{\text{soft}} \bigg] = 0
        &\text{ on } \Gamma_\ls, \\
        u_{\text{soft}}=u_{\text{stiff-int}} \quad &\text{ on } \Gamma_\interior, \\
        \bigg[ \frac{\partial u_{\text{soft}}}{\partial n_{\text{soft}}} + i (\uptau \cdot n_{\text{soft}}) u_{\text{soft}} \bigg] + \varepsilon^{-2} \bigg[ \frac{\partial u_{\text{stiff-int}}}{\partial n_{\text{stiff-int}}} + i (\uptau \cdot n_{\text{stiff-int}}) u_{\text{stiff-int}} \bigg] = 0
        &\text{ on } \Gamma_\interior, \\
        u_{\text{stiff-ls}} \text{ periodic} &\text{ on $\partial Q$},
    \end{cases}
\end{equation}
where $n_{\bigstar}$ denotes the outward unit normal vector with respect to $Q_{\bigstar}$, $\bigstar \in \{ \text{stiff-int}, \text{soft}, \text{stiff-ls} \} $.

Finally, we conclude this section by introducing the following notation:

\begin{defn}
    Let $(\bigstar, \bullet) \in \{(\text{soft}, \text{ls}), (\text{stiff-ls}, \text{ls}), (\text{soft}, \text{int}), (\text{stiff-int}, \text{int}) \}$, we denote by $n_\bigstar$ the outward pointing unit normal vector with respect to the component $Q_\bigstar$. Also, let $\partial^{(\uptau)}_{n_\bigstar,\bullet} u$ be the trace of the co-normal derivative of $u$, with respect to $\bigstar$, on the boundary $\Gamma_{\bullet}$. This is defined for $u\in H^{3/2}(Q)$, by
    \begin{align}
        \partial^{(\uptau)}_{n_\bigstar,\bullet} u := - \left( \frac{\partial u}{\partial n_\bigstar} + i(\uptau \cdot n_{\bigstar}) u \right) \bigg|_{\Gamma_{\bullet}}.
    \end{align}
    (Note the minus sign convention.)
\end{defn}


\section{Boundary triple theory setup}\label{sect:bdry_triple_setup}
\subsection{Preliminaries}\label{sect:bdry_triple_prelim}

In this section, we will discuss the three ingredients that make up the Ryzhov boundary triple \cite{ryzhov2009}, namely, the $\uptau$-Dirichlet decoupling, the $\uptau$-harmonic lift, and the $\uptau$-Dirichlet-to-Neumann ($\uptau$-DtN) operator. First, we introduce a new notation for the spaces:

\begin{defn}
    Set $\mathcal{H} := L^2(Q) = L^2(Q_\stin) \oplus L^2(Q_\soft) \oplus L^2(Q_\stls)$ and $\mathcal{E} := L^2(\Gamma_\interior)\oplus L^2(\Gamma_\ls)$. We refer to $\mathcal{E}$ as the \textit{boundary space}.
\end{defn}

\begin{rmk}[On notation]
    We will view $L^2(Q_\stin)$, $L^2(Q_\soft)$, and $L^2(Q_\stls)$ as subspaces of $\mathcal{H}$. This means, for instance, that a function $u \in L^2(Q_\soft)$ may be viewed as an element of $\mathcal{H}$
    \begin{itemize}
        \item by an extension by zero onto $Q_\stin \cup Q_\stls$, in which we write $0 + u + 0$ or simply $u$,
        \item or by an identification with the second component of $(0,u,0)$.
    \end{itemize}
    We will switch between the two notations where convenient. A similar remark applies to $\mathcal{E}$ and its subspaces $L^2(\Gamma_\interior)$ and $L^2(\Gamma_\ls)$.
\end{rmk}

\begin{defn}[Projections on $\mathcal{H}$ and $\mathcal{E}$]\label{defn:projections}
    For $\bigstar \in \{\text{stiff-int}, \text{soft}, \text{stiff-ls} \}$, we write $P_\bigstar \in \mathcal{L}(\mathcal{H})$ for the orthogonal projection of $\mathcal{H}$ onto $L^2(Q_{\bigstar})$. Similarly, for $\bullet \in \{ \text{int}, \text{ls} \}$ we write $\mathcal{P}_\bullet \in \mathcal{L}(\mathcal{E})$ for the orthogonal projection of $\mathcal{E}$ onto $L^2(\Gamma_{\bullet})$. (Note the calligraphic font for projections on the boundary space.)
\end{defn}

\subsection*{The \texorpdfstring{$\tau$}{tau}-Dirichlet decoupling}

The first ingredient, the $\tau$-Dirichlet decoupling, is constructed using the Dirichlet operators $-(\nabla + i\uptau)^2$, appropriately rescaled, on each connected component of $Q$.

\begin{defn}
    The \textit{$\uptau$-Dirichlet decoupling} is the operator on $\mathcal{H} = L^2(Q)$ defined by 
    \begin{align}
        A^{(\uptau)}_{\varepsilon,0} = A^{\text{stiff-int},(\uptau)}_{\varepsilon,0} \oplus A^{\text{soft},(\uptau)}_{0} \oplus A^{\text{stiff-ls},(\uptau)}_{\varepsilon,0}
    \end{align}
    where, 
    \begin{itemize}
        \item $A^{\text{stiff-int},(\uptau)}_{\varepsilon,0}$ is the operator $- \eps^{-2}(\nabla + i\uptau)^2$ on $L^2(Q_{\text{stiff-int}})$ with Dirichlet BC on $\Gamma_\interior$. That is, the operator defined through its sesquilinear form having form domain $\mathcal{D}[A^{\text{stiff-int},(\uptau)}_{\varepsilon,0}] = H^1_0(Q_{\text{stiff-int}})$ and action $(u,v)\mapsto \int_{Q_{\text{stiff-int}}} \eps^{-1} (\frac{1}{i}\nabla + \uptau)u \cdot \overline{ \eps^{-1} (\frac{1}{i}\nabla + \uptau) v}$.
        
        \item $A^{\text{soft},(\uptau)}_{0}$ is the operator $-(\nabla + i\uptau)^2$ on $L^2(Q_{\text{soft}})$ with Dirichlet BCs on $\Gamma_\interior \cup \Gamma_\ls$. That is, $\mathcal{D}[A^{\text{soft},(\uptau)}_0] = H^1_0(Q_{\text{soft}})$.
        
        \item $A^{\text{stiff-ls},(\uptau)}_{\varepsilon,0}$ is the operator $- \eps^{-2}(\nabla + i\uptau)^2$ on $L^2(Q_{\text{stiff-ls}})$ with Dirichlet BC on $\Gamma_\ls$ and periodic BCs on $\partial Q$. That is, $\mathcal{D}[A^{\text{stiff-ls},(\uptau)}_{\varepsilon,0}] = H^1_{0,\text{per}}(Q_{\text{stiff-ls}})$, the closure of smooth functions that are periodic on $\partial Q$ and with compact support in $\partial Q \cup Q_{\text{stiff-ls}}$, under the $H^1$ norm.
    \end{itemize}
    Write $\widetilde{A}^{\text{stiff-int}, (\uptau)}_0 := \varepsilon^2 A^{\text{stiff-int},(\uptau)}_{\varepsilon,0}$ and $\widetilde{A}^{\text{stiff-ls}, (\uptau)}_0 := \varepsilon^2 A^{\text{stiff-ls},(\uptau)}_{\varepsilon,0}$ for the unweighted operators.
\end{defn}

We record some properties of $\azero{main}$ that will be useful to us.

\begin{prop}\label{prop:decoupling_properties}
    For all $\uptau \in \overline{Q'}=[-\pi,\pi]^d$, $A^{(\uptau)}_{\varepsilon,0}$ is self-adjoint, positive definite, has purely discrete spectrum, and $0\in\rho(A^{(\uptau)}_{\varepsilon,0})$. Moreover, $A^{\text{soft},(\uptau)}_0$ and $A^{\bigstar,(\uptau)}_{\varepsilon,0}$ are bounded below, \textit{uniformly in $\uptau$ and $\varepsilon$}, assuming $\varepsilon$ is small enough, $\bigstar \in \{ \text{stiff-int, stiff-ls}\}$. We also have the following estimates: For some $C>0$, independent of $\uptau$ and $\varepsilon$, assuming $\varepsilon$ is small enough,  
    \begin{align}
        \|(A^{\text{soft},(\uptau)}_{0})^{-1} \|_{L^2(Q_{\text{soft}})\rightarrow L^2(Q_{\text{soft}})} \leq C, \\
        \|(A^{\bigstar,(\uptau)}_{\varepsilon,0})^{-1} \|_{L^2(Q_{\bigstar})\rightarrow L^2(Q_{\bigstar})} \leq C \varepsilon^2.
    \end{align}
\end{prop}

\begin{proof}
    The self-adjointness, positive-semi-definiteness, and spectral type follows immediately as it is the orthogonal sum of operators with these properties. The positive-definiteness will then follow from $0\in\rho(A^{(\uptau)}_{\varepsilon,0})$. To show this, we first note that the case $\uptau=0$ follows from the Poincar\'e inequality applied to each of the three operators in $A^{(\uptau)}_{\varepsilon,0}$, since the first/lowest eigenvalue $\lambda_1$ is related to the optimal Poincar\'e constant $\gamma$ by $\lambda_1 = \gamma^{-2}>0$. ($\gamma$ can be taken to be independent of $\varepsilon$, if we assume $\varepsilon$ is small.)
    
    For general $\uptau$, the lowest eigenvalue is always greater than or equal to the $\uptau=0$ case. This is due to the diamagnetic inequality $|\nabla|f|(x)|\leq |(\nabla + i\uptau)f(x)|$ a.e., $f\in H^1$, and the fact that we can always choose the first Dirichlet eigenfunction (for $\tau=0$) to be strictly positive \cite[Theorem 8.38]{gilbarg_trudinger}. This shows the claim on being uniformly bounded below.
    
    Since the norm of $(A^{\text{soft},(\uptau)}_0)^{-1}$ is bounded above by $\left(\text{dist}(0, \sigma(A^{\text{soft}, (\uptau)}_0)) \right)^{-1}$, the estimate follows. A similar argument applies to the ``stiff" decouplings. 
\end{proof}


\subsection*{The \texorpdfstring{$\tau$}{tau}-harmonic lift}
The second ingredient, the \textbf{$\tau$-harmonic lift}, generalizes the map that takes boundary data to harmonic functions.

\begin{defn}
    The \textit{$\uptau$-harmonic lift} is the operator $\Pi^{(\uptau)}:\mathcal{E} \rightarrow \mathcal{H}$, defined by
    \begin{align}
        \Pi^{(\uptau)} = \Pi^{\text{stiff-int}, (\uptau)} \mathcal{P}_\interior + \Pi^{\text{soft}, (\uptau)} + \Pi^{\text{stiff-ls}, (\uptau)}\mathcal{P}_\ls,
    \end{align}
    where
    \begin{itemize}
        \item $\Pi^{\text{stiff-int}, (\uptau)}: L^2(\Gamma_\interior) \rightarrow L^2(Q_{\text{stiff-int}})$ is the operator $\phi \mapsto u_\phi$, where $u_\phi$ is the unique solution to the BVP
        \begin{equation}
            \begin{cases}
                -(\nabla + i\uptau)^2 u_\phi = 0 & \text{ in $Q_{\text{stiff-int}}$,} \\
                u_\phi = \phi &\text{ on $\Gamma_\interior$.}
            \end{cases}
        \end{equation}
        
        \item $\Pi^{\text{soft}, (\uptau)}: \mathcal{E} \rightarrow L^2(Q_{\text{soft}})$ is the operator $(\phi,\varphi) \mapsto u_{\phi,\varphi}$, where $u_{\phi,\varphi}$ is the unique solution to the BVP
        \begin{equation}
            \begin{cases}
                -(\nabla + i\uptau)^2 u_{\phi,\varphi} = 0 & \text{ in $Q_{\text{soft}}$,} \\
                u_{\phi,\varphi} = \phi &\text{ on $\Gamma_\interior$,} \\
                u_{\phi,\varphi} = \varphi &\text{ on $\Gamma_\ls$.}
            \end{cases}
        \end{equation}
        
        \item $\Pi^{\text{stiff-ls}, (\uptau)}: L^2(\Gamma_\ls) \rightarrow L^2(Q_{\text{stiff-ls}})$ is the operator $\varphi \mapsto u_\varphi$, where $u_\phi$ is the unique solution to the BVP
        \begin{equation}
            \begin{cases}
                -(\nabla + i\uptau)^2 u_\phi = 0 & \text{ in $Q_{\text{stiff-ls}}$,} \\
                u_\varphi = \varphi &\text{ on $\Gamma_\ls$,} \\
                u_\varphi \text{ periodic} &\text{ on $\partial Q$.}
            \end{cases}
        \end{equation}
    \end{itemize}
\end{defn}

Note that while $\Pi^{(\tau)}$ is not a direct sum of $\Pi^{\stin, (\uptau)}$, $\Pi^{\soft, (\uptau)}$, and $\Pi^{\soft, (\uptau)}$, their lifts into the components $L^2(Q_\stin)$, $L^2(Q_\soft)$, and $L^2(Q_\stls)$ are mutually orthogonal:
\begin{align*}
    \Pi^{(\uptau)}(\phi + \varphi) = \left( \Pi^{\text{stiff-int}, (\uptau)}\phi,  \Pi^{\text{soft}, (\uptau)}(\phi + \varphi),  \Pi^{\text{stiff-ls}, (\uptau)}\varphi \right), \quad \text{for $\phi \in L^2(\Gamma_\interior)$ and $\varphi \in L^2(\Gamma_\ls)$.}
\end{align*}

Below, we give a sketch on how the lifts are constructed, and refer the reader to \cite[Theorem 4.25]{strongly_elliptic_sys} for the full details. For concreteness, we focus on $\Pi^{\text{stiff-int}, (\uptau)}$. We remark that the construction applies to $\Pi^{\soft,(\tau)}$, as $Q_\soft$ is connected with Lipschitz domain $\Gamma_\interior \cup \Gamma_\ls$.

\begin{quote}
    The lift $\Pi^{\text{stiff-int}, (\uptau)}$ is initially defined as a mapping from $H^{1/2}(\Gamma_\interior)$ to $H^1(Q_\stin)$. This is possible because the fully homogeneous problem (zero RHS and zero on the boundary) is uniquely solved by $u\equiv 0$, as $A^{\text{stiff-int}, (\uptau)}_{\varepsilon,0}$ is injective (Proposition \ref{prop:decoupling_properties}).
    
    We then show that $\Pi^{\text{stiff-int}, (\uptau)}$ admits a continuous extension to $L^2(\Gamma_\interior)$, by verifying an $L^2$ estimate for $u = \Pi^{\text{stiff-int}, (\uptau)}\phi$, where $\phi \in H^{1/2}(\Gamma_\interior)$. To do this, consider the adjoint problem $``L^* w = f"$ corresponding to our harmonic lift problem $``Lu = 0"$. The idea now is to combine both problems with Green's identity, giving
    \begin{align*}
        0 = (u, f)_{L^2(Q_\stin)} + (\phi, \partial^{(\uptau)}_n w)_{L^2(\Gamma_\interior)}.
    \end{align*}
    Pick $f=u$. Bound $\| \partial^{(\uptau)}_n w \|_{L^2}$ in terms of $\| f \|_{L^2}$ (this step is tedious, and requires elliptic regularity) and we get the required inequality.
\end{quote}

By definition, the lifts do not depend on $\eps$. As for $\uptau$, we have the following:

\begin{prop}\label{prop:lift_properties}
    There is some $C>0$, independent of $\uptau$ (and $\varepsilon$), such that
    \begin{align}
        \|\Pi^{(\uptau)}\|_{\mathcal{E}\rightarrow \mathcal{H}} < C.
    \end{align}
\end{prop}

\begin{proof}
    This follows from the continuity of the mapping $\overline{Q'}\ni \uptau \to \| \Pi^{(\tau)} \|_{\mathcal{E} \to \mathcal{H}}$, which was proved in \cite[Proposition~5.5]{ys_thesis}.
\end{proof}

Let us record two more properties of $\Pi^{(\tau)}$ that are necessary for constructing boundary triples. First, owing to the fact that the decoupling $A^{(\uptau)}_{\varepsilon,0}$ has Dirichlet BCs, one has
\begin{align}
    \mathcal{D}(A^{(\uptau)}_{\varepsilon,0}) \cap \text{ran}(\Pi^{(\uptau)}) = \{ 0 \}.
\end{align}

Second, the individual lifts $\Pi^{\stin, (\uptau)}$, $\Pi^{\soft, (\uptau)}$, and $\Pi^{\stls, (\uptau)}$ are injective, and hence 
\begin{align}
    \text{ker}(\Pi^{(\uptau)}) = \{ 0 \}.
\end{align}
To prove the injectivity of, say $\Pi^{\stin, (\uptau)}$, one first observes that $\Pi^{\stin, (\uptau)}$ can be characterized as the adjoint of the operator
\begin{align*}
    L^2(Q_\stin) \ni f \mapsto \partial^{(\uptau)}_{n_\stin, \interior} (\widetilde{A}_{0}^{\stin, (\tau)})^{-1} f.
\end{align*}
(This is a prequel to the identity $\Pi^*=\Gamma_1 A_0^{-1}$ of Proposition \ref{prop:auxops_properties}.) Since $\Gamma_\interior$ is smooth, an argument using elliptic regularity implies that the range of this operator contains $C^\infty (\Gamma_\interior)$, which is dense in $L^2(\Gamma_\interior)$.


\subsection*{The \texorpdfstring{$\tau$}{tau}-Dirichlet-to-Neumann operator}
The final ingredient of the boundary triple is the \textbf{$\tau$-Dirichlet-to-Neumann operator}.

\begin{defn}
    The $\uptau$-Dirichlet-to-Neumann ($\uptau$-DtN) operator is the (unbounded) operator $\Lambda^{(\uptau)}_\varepsilon$ on $\mathcal{E}$ defined with domain $\mathcal{D}(\Lambda^{(\uptau)}_\varepsilon) = H^1(\Gamma_\interior) \oplus H^1(\Gamma_\ls)$ and action
    \begin{align}
        (\phi,\varphi) \mapsto &-\varepsilon^{-2} \left[ \frac{\partial u_{\text{stiff-int}}} {\partial n_{\text{stiff-int}}} + i(\uptau \cdot n_{\text{stiff-int}}) u_{\text{stiff-int}}  \right] - \left[ \frac{\partial u_{\text{soft}}}{\partial n_{\text{soft}}} + i(\uptau \cdot n_{\text{soft}}) u_{\text{soft}} \right] \nonumber \\
        &-\varepsilon^{-2} \left[ \frac{\partial u_{\text{stiff-ls}}} {\partial n_{\text{stiff-ls}}} + i(\uptau \cdot n_{\text{stiff-ls}} ) u_{\text{stiff-ls}}  \right] \nonumber \\
        = &\eps^{-2}\partial^{(\uptau)}_{n_\stin,\interior} u_\stin + \partial^{(\uptau)}_{n_\soft,\interior} u_\soft + \partial^{(\uptau)}_{n_\soft, \ls} u_\soft + \eps^{-2}\partial^{(\uptau)}_{n_\stls,\ls} u_\stls,
    \end{align}
    where $u_{\phi,\varphi} = u = u_{\text{stiff-int}} + u_{\text{soft}} + u_{\text{stiff-ls}}$ is the solution to the BVP 
    
    \begin{equation}
        \begin{cases}
            -(\nabla + i\uptau)^2 u_{\text{stiff-int}} = 0 & \text{ in $Q_{\text{stiff-int}}$,} \\
            -(\nabla + i\uptau)^2 u_{\text{soft}} = 0 & \text{ in $Q_{\text{soft}}$,} \\
            -(\nabla + i\uptau)^2 u_{\text{stiff-ls}} = 0 & \text{ in $Q_{\text{stiff-ls}}$,} \\
            u_{\text{stiff-int}} = u_{\text{soft}} = \phi &\text{ on $\Gamma_\interior$,} \\
            u_{\text{stiff-ls}} = u_{\text{soft}} = \varphi &\text{ on $\Gamma_\ls$,} \\
            u_{\text{stiff-int}} \text{ periodic} &\text{ on $\partial Q$.}
        \end{cases}
    \end{equation}
    For convenience, we introduce the following notation for the DtN on the individual components:
    \begin{itemize}
        \item Denote by $\Lambda^{\text{stiff-int}, (\uptau)}_{\varepsilon}$ the operator with domain $H^1(\Gamma_\interior)$ and action $\phi \mapsto \varepsilon^{-2} \partial^{(\uptau)}_{n_\stin,\interior} u_\phi$, where $u_\phi = u_{\text{stiff-int}}$ solves the BVP
        \begin{equation}
            \begin{cases}
                -(\nabla + i\uptau)^2 u_{\text{stiff-int}} = 0 & \text{ in $Q_{\text{stiff-int}}$,} \\
                u_{\text{stiff-int}} = \phi &\text{ on $\Gamma_\interior$.} \\
            \end{cases}
        \end{equation}
        \item Denote by $\Lambda^{\text{soft}, (\uptau)}$ the operator with domain $H^1(\Gamma_\interior) \oplus H^1(\Gamma_\ls)$ and action \newline $(\phi, \varphi) \mapsto \partial^{(\uptau)}_{n_\soft,\interior} u_{\phi,\varphi} + \partial^{(\uptau)}_{n_\soft,\ls} u_{\phi,\varphi}$, where $u_{\phi,\varphi} = u_{\text{soft}}$ solves the BVP
        \begin{equation}
            \begin{cases}
                -(\nabla + i\uptau)^2 u_{\text{stiff-ls}} = 0 & \text{ in $Q_{\text{soft}}$,} \\
                u_{\text{soft}} = \phi &\text{ on $\Gamma_\interior$,} \\
                u_{\text{soft}} = \varphi &\text{ on $\Gamma_\ls$.}
            \end{cases}
        \end{equation}
        \item Denote by $\Lambda^{\text{stiff-ls}, (\uptau)}_{\varepsilon}$ the operator with domain $H^1(\Gamma_\ls)$ and action $\varphi \mapsto \varepsilon^{-2} \partial^{(\uptau)}_{n_\stls,\ls} u_\varphi$, where $u_\varphi = u_{\text{stiff-ls}}$ solves the BVP
        \begin{equation}
            \begin{cases}
                -(\nabla + i\uptau)^2 u_{\text{stiff-ls}} = 0 & \text{ in $Q_{\text{stiff-ls}}$,} \\
                u_{\text{stiff-ls}} = \varphi &\text{ on $\Gamma_\ls$,} \\
                u_{\text{stiff-ls}} \text{ periodic} &\text{ on $\partial Q$.}
            \end{cases}
        \end{equation}
        \item Write $\widetilde{\Lambda}^{\text{stiff-int},(\uptau)} := \varepsilon^2 \Lambda^{\text{stiff-int},(\uptau)}_\varepsilon$ and $\widetilde{\Lambda}^{\text{stiff-ls},(\uptau)} := \varepsilon^2 \Lambda^{\text{stiff-ls},(\uptau)}_\varepsilon$ for the unweighted operators.
        
    \end{itemize}
    
\end{defn}

We may thus write $\Lambda^{(\uptau)}_\varepsilon$ as a sum of operators on $L^2(\Gamma_\interior)$, $\mathcal{E}$, and $L^2(\Gamma_\ls)$ respectively:
\begin{align}
    \Lambda^{(\uptau)}_\varepsilon = \Lambda^{\text{stiff-int}, (\uptau)}_{\varepsilon} \mathcal{P}_\interior + \Lambda^{\text{soft}, (\uptau)} + \Lambda^{\text{stiff-ls}, (\uptau)}_{\varepsilon} \mathcal{P}_\ls.
\end{align}

\begin{rmk}
    We have used the assumption that the boundaries $\Gamma_\interior$ and $\Gamma_\ls$ are at least $C^{1,1}$, so that by \cite[Theorem~4.21]{strongly_elliptic_sys}, the co-normal derivatives are well-defined.
\end{rmk}

We refer the reader to \cite[p.~145]{strongly_elliptic_sys} and \cite{ardent_mazzeo} for general properties of $\tau$-DtN maps. Of note in the construction of the DtN maps, is the requirement that $u\equiv 0$ is the unique solution to the fully homogeneous problem, similarly to the lifts $\Pi$. This refers to formulae (4.35) to (4.38) of \cite{strongly_elliptic_sys}, and Section 3 of \cite{ardent_mazzeo} (the assumption that $0$ belongs to the resolvent set of the Dirichlet Laplacian).

To construct our boundary triple, we require the DtN map to be self-adjoint \cite[Assumption 2]{ryzhov2009}. This is immediate for $\Lambda_\eps^{\stin,(\uptau)}$, $\Lambda^{\soft,(\uptau)}$, and $\Lambda_\eps^{\stls,(\uptau)}$, by \cite[Theorem~3.1]{ardent_mazzeo}.

\begin{lem}\label{lem:dtn_selfadjoint}
    For $\varepsilon$ small enough, independently of $\uptau$, $\Lambda_\eps^{(\uptau)}$ is self-adjoint on $\mathcal{E}$.
\end{lem}

\begin{proof}[Sketch of proof.]
    We will outline the idea of \cite[Lemma 2.1]{eff_behavior} and state the modifications needed. 
    
    First, it suffices to discuss the $\uptau=0$ case, as general $\uptau$ can be viewed as a relatively bounded perturbation of the $\uptau=0$ case. We hence omit writing $\tau$ for the remainder of the proof.

    Second, we note that
    \begin{align}
        \Lambda^\stin_\eps \mathcal{P}_\interior + \Lambda^\stls_\eps \mathcal{P}_\ls 
        = \Lambda^\stin_\eps \oplus \Lambda^\stls_\eps,
    \end{align}
    which is an orthogonal sum of self-adjoint operators, hence it is self-adjoint on $\mathcal{E}$ with domain $H^1(\Gamma_\interior)\oplus H^1(\Gamma_\ls)$.
    
    Third, we view $\Lambda_\varepsilon$ as the operator $\Lambda^{\text{stiff-int}}_\varepsilon \oplus \Lambda^{\text{stiff-ls}}_\varepsilon$ perturbed by the ``soft" DtN operator $\Lambda^{\text{soft}}$. In fact, we may verify the following estimate: there exist some $\alpha, \beta > 0$, independent of $\eps$,  such that for all $(\phi, \varphi) \in \mathcal{D}(\Lambda^{\text{stiff-int}}_\varepsilon \oplus \Lambda^{\text{stiff-ls}}_\varepsilon)$,
    $$
        \| \Lambda^{\text{soft}}(\phi + \varphi) \|_{\mathcal{E}} \leq \alpha \varepsilon^2 \| (\Lambda^{\text{stiff-int}}_\varepsilon \oplus \Lambda^{\text{stiff-ls}}_\varepsilon) (\phi + \varphi) \|_{\mathcal{E}} + \beta \| \phi + \varphi \|_{\mathcal{E}}.
    $$
    Therefore, if $\varepsilon$ is small enough, then $\Lambda^{\text{soft}}$ is relatively $(\Lambda^{\text{stiff-int}}_\varepsilon \oplus \Lambda^{\text{stiff-ls}}_\varepsilon)$-bounded with bound strictly less than $1$, hence the sum $\Lambda_\eps = \Lambda^\soft + (\Lambda_\eps^{\stin} \oplus \Lambda_\eps^{\stls})$ is self-adjoint by the Kato-Rellich theorem \cite[Theorem~8.5]{konrad_book}.

    Finally, we note that $\Lambda^{\text{soft}}$ and $\Lambda^{\text{stiff-int}}_\varepsilon \oplus \Lambda^{\text{stiff-ls}}_\varepsilon$ have common domain $H^1(\Gamma_\interior) \oplus H^1(\Gamma_\ls)$, so this is also the domain of the sum $\Lambda_\eps$.
\end{proof}

\begin{quote}
    \textbf{We will henceforth assume that $\varepsilon>0$ is small enough to satisfy Proposition \ref{prop:decoupling_properties} and Lemma \ref{lem:dtn_selfadjoint}.}
\end{quote}

The DtN operator is an important object for our analysis. Not only is it one of the main ingredients of the boundary triple, it also features prominently in \textit{Krein's formula}, a key result in the boundary triples theory. In particular, of interest to us are spectral properties of the unweighted ``stiff" DtN components $\widetilde{\Lambda}^{\text{stiff-int},(\uptau)} = \varepsilon^2 \Lambda^{\text{stiff-int},(\uptau)}_\varepsilon$ and $\widetilde{\Lambda}^{\text{stiff-ls},(\uptau)} = \varepsilon^2 \Lambda^{\text{stiff-ls},(\uptau)}_\varepsilon$. The following proposition collects the relevant properties for these DtN maps.

\begin{prop}\label{prop:dtn_properties}
    For all $\uptau \in \overline{Q'} = [-\pi,\pi]^d$, $\widetilde{\Lambda}^{\text{stiff-int},(\uptau)}$, $\Lambda^{\text{soft}, (\uptau)}$, and $\widetilde{\Lambda}^{\text{stiff-ls},(\uptau)}$ are unbounded self-adjoint operators on $L^2(\Gamma_\interior)$, $\mathcal{E}$, and $L^2(\Gamma_\ls)$ respectively. They are semibounded from above (note our sign convention for $\partial^{(\uptau)}_n$), and have compact resolvents. Moreover, if we order the eigenvalues of $\widetilde{\Lambda}^{\text{stiff-int},(\uptau)}$ and $\widetilde{\Lambda}^{\text{stiff-ls},(\uptau)}$ in descending order counting multiplicities, then
    
    \begin{itemize}
        \item The eigenvalues of $\widetilde{\Lambda}^{\text{stiff-int},(\uptau)}$ satisfy
        \begin{equation}
            \text{For all $\uptau$,} \quad 0 = \mu^{\text{stiff-int}, (\uptau)}_1 > \mu^{\text{stiff-int}, (\uptau)}_2 \geq \mu^{\text{stiff-int}, (\uptau)}_3 \geq \cdots \rightarrow -\infty.
        \end{equation}
        
        The eigenfunction $\psi_1^{\text{stiff-int},(\uptau)}$ corresponding to the first eigenvalue is \newline $\psi_1^{\text{stiff-int},(\uptau)}(x) = |\Gamma_\interior|^{-\frac{1}{2}} e^{-i\uptau\cdot x}$. In particular this is a constant when $\uptau=0$.

        \item The eigenvalues of $\widetilde{\Lambda}^{\text{stiff-ls},(\uptau)}$ satisfy
        \begin{align}
            \text{If } \uptau=0, \text{ then } &0 = \mu^{\text{stiff-ls},(\uptau)}_1 > \mu^{\text{stiff-ls},(\uptau)}_2 \geq \mu^{\text{stiff-ls},(\uptau)}_3 \geq \cdots \rightarrow -\infty. \\
            \text{If } \uptau\neq 0, \text{ then } &0 > \mu^{\text{stiff-ls},(\uptau)}_1 \geq \mu^{\text{stiff-ls},(\uptau)}_2 \geq \mu^{\text{stiff-ls},(\uptau)}_3 \geq \cdots \rightarrow -\infty.
        \end{align}

        Moreover, $\mu_1^{\stls,(\tau)}$ is simple when $|\tau|$ is small enough.
        
        The first eigenvalue admits an asymptotic expansion in $\uptau$ with the leading order being quadratic in $\uptau$. That is, there exist a (strictly) negative-definite matrix $\mu_*^{\stls}$ satisfying
        \begin{align}
            \mu^{\text{stiff-ls},(\uptau)}_1 = \mu_*^{\stls} \uptau \cdot \uptau + O\left(|\uptau|^3\right).
        \end{align}
        For the case $\uptau=0$, the eigenfunction $\psi_1^{\text{stiff-ls},(\uptau)}$ corresponding to the first eigenvalue is constant, $\psi_1^{\text{stiff-ls},(\uptau)}(x) = |\Gamma_\ls|^{-\frac{1}{2}}\mathbf{1}_{\Gamma_\ls}(x)$. For general $\uptau$, the eigenfunction $\psi_1^{\text{stiff-ls},(\uptau)}$ admits an expansion: there exist some $\psi_*^{\stls} = (\psi_{(1)}^{\stls}, \cdots , \psi_{(d)}^{\stls}) \in \left(L^2(\Gamma_\ls)\right)^d$ such that
        \begin{align}
            \psi_1^{\text{stiff-ls},(\uptau)} = |\Gamma_\ls|^{-\frac{1}{2}} \left( 1 + \uptau \cdot \psi_*^{\stls} + O\left(|\uptau|^2 \right)  \right).
        \end{align}
        
    \end{itemize}
    
\end{prop}

\begin{proof}
    See \cite{ardent_mazzeo}, which discusses the DtN operator corresponding to $-\Delta$ (the case $\uptau=0$), and with a bounded connected domain $\Omega \subset \mathbb{R}^d$ with Lipschitz boundary. It was proven in \cite{ardent_mazzeo} that the DtN map is self-adjoint, semibounded, and has compact resolvent \cite[Theorem 3.1]{ardent_mazzeo}. The key fact allowing us to conclude the compactness of the resolvent is the compactness of trace operator $H^1(\Omega)\rightarrow L^2(\partial \Omega)$. An easy modification to the case $-(\nabla + i\tau)^2$ allows us to conclude the same for $\widetilde{\Lambda}^{\stin,(\uptau)}$, $\widetilde{\Lambda}^{\stls,(\uptau)}$, and $\Lambda^{\text{soft}, (\uptau)}$.
    
    Simplicity of the first DtN eigenvalue is proved in the follow-up work in \cite[Theorem 1.2]{ardent_elst}. Note that $(\psi_k^{\stin,(0)},\mu_k^{\stin,(0)})$ is an eigenpair for $\widetilde{\Lambda}^{\stin,(0)}$ if and only if $(\psi_k^{\stin,(0)}e^{-i\tau\cdot x},$ $\mu_k^{\stin,(0)})$ is an eigenpair for $\widetilde{\Lambda}^{\stin,(\tau)}$.

    Since $Q_\stls$ (with edges identified) is connected and $\Gamma_\ls$ is smooth, the arguments of \cite[Proposition 4.1]{ardent_elst} can be modified to the setting of $H^1_\text{per}(Q_\stls)$ to give the simplicity of $\mu_1^{\stls,(0)}$. The claim that $\mu_1^{\stls,(\tau)}<0$ for $\tau\neq 0$ is a consequence of Corollary \ref{cor:stls_evalue1}. Corollary \ref{cor:stls_evalue1}, combined with the fact that $\mu_2^{\stls,(\tau)}$ can be bounded away from zero uniformly in $\tau$ implies the simplicity of $\mu_1^{\stls,(\tau)}$ for small $|\tau|$. We postpone the self-contained argument on $\mu_2^{\stls,(\tau)}$ to the proof of Theorem \ref{thm:m_inverse_est}.

    The claims that lowest eigenvalue is zero for $\widetilde{\Lambda}^{\text{stiff-int},(\uptau)}$ for all $\uptau$ and for $\widetilde{\Lambda}^{\text{stiff-ls},(\uptau)}$ for $\uptau=0$ is a direct check on the expression for the eigenfunctions. The claims on the asymptotic expansions for $\mu_1^{\text{stiff-ls},(\uptau)}$ and $\psi_1^{\text{stiff-ls},(\uptau)}$ are proven in \cite[Proposition~5.5]{ys_thesis}.
\end{proof}


\subsection{Applying the triple framework. General properties.}

We will now use the three ingredients provided in the previous section to define boundary triples and several auxiliary operators. This construction is done for each $\varepsilon>0$ and $\uptau \in Q$.

\begin{defn}
    (\cite{ryzhov2009}.) By a \textit{(Ryzhov) boundary triple}, we mean two separable Hilbert spaces $\mathcal{H}$ and $\mathcal{E}$ ($\mathcal{E}$ is called the \textit{boundary space}), and a triple of operators $(A_0,\Lambda, \Pi)$ such that:
    
    \begin{itemize}
        \item (Dirichlet decoupling) $A_0$ is an unbounded self-adjoint operator on $\mathcal{H}$, with $0\in\rho(A_0)$,
        \item (DtN operator) $\Lambda$ is an unbounded self-adjoint operator on $\mathcal{E}$,
        \item (Lift) $\Pi:\mathcal{E}\rightarrow\mathcal{H}$ is a bounded linear map such that $\text{ker}(\Pi) = \{ 0 \}$,
        \item $\mathcal{D}(A_0) \cap \text{ran}(\Pi) = \{ 0 \}$.
    \end{itemize}
    
    When the underlying Hilbert spaces are clear form the context, we will simply refer to $(A_0,\Lambda,\Pi)$ as the boundary triple.
\end{defn}

We now proceed to define the auxiliary operators $\widehat{A}$, $\Gamma_0$, $\Gamma$, $S(z)$, and $M(z)$, corresponding to a boundary triple $(A_0,\Lambda,\Pi)$.

\begin{defn}\label{defn:aux_operators}
    Let $(A_0,\Lambda, \Pi)$ be a boundary triple with spaces $\mathcal{H}$ and $\mathcal{E}$. Define the following operators:
    \begin{itemize}
        \item $\widehat{A}: \mathcal{H} \supset \mathcal{D}(\widehat{A}) \rightarrow \mathcal{H}$, with domain $\mathcal{D}(\widehat{A}) = \mathcal{D}(A_0) \dot{+} \text{ran}(\Pi)$ and action
        \begin{align}
            \widehat{A}(A^{-1}_0 f + \Pi \phi) = f, \quad f\in\mathcal{H}, \phi \in \mathcal{E}.
        \end{align}
        \item $\Gamma_0: \mathcal{H} \supset \mathcal{D}(\Gamma_0) \rightarrow \mathcal{E}$, with domain $\mathcal{D}(\Gamma_0) = \mathcal{D}(A_0) \dot{+} \text{ran}(\Pi)$ and action
        \begin{align}
            \Gamma_0(A^{-1}_0 f + \Pi \phi) = \phi, \quad f\in\mathcal{H}, \phi \in \mathcal{E}.
        \end{align}
        
        (We have used the assumptions $\mathcal{D}(A_0) \cap \text{ran}(\Pi) = \{ 0 \}$, $0\in\rho(A_0)$, and $\text{ker}(\Pi) = \{ 0\}$.)

        \item $\Gamma_1: \mathcal{H} \supset \mathcal{D}(\Gamma_1) \rightarrow \mathcal{E}$, with domain $\mathcal{D}(\Gamma_1) = \mathcal{D}(A_0) \dot{+} \Pi \mathcal{D}(\Lambda)$ and action
        \begin{align}
            \Gamma_1(A^{-1}_0 f + \Pi \phi) = \Pi^* f + \Lambda \phi, \quad f\in\mathcal{H}, \phi \in \mathcal{D}(\Lambda) \subset \mathcal{E}.
        \end{align}

        \item (Solution operator) For $z\in \rho(A_0)$, define the bounded linear operator $S(z):\mathcal{E}\rightarrow \mathcal{H}$ by
        \begin{align}
            S(z)\phi := \left( I + z(A_0 - z)^{-1} \right) \Pi\phi.
        \end{align}
        
        \item (M-operator) For $z\in\rho(A_0)$, we define the operator $M(z): \mathcal{E} \supset \mathcal{D}(M(z)) \rightarrow \mathcal{E}$, with domain $\mathcal{D}(M(z)) = \mathcal{D}(\Lambda)$ (independent of $z$), and action
        \begin{align}
            M(z)\phi := \Gamma_1 S(z) \phi, \quad \phi \in \mathcal{D}(M(z)).
        \end{align}
    \end{itemize}
\end{defn}

Let us now provide a motivation for the operators in Definition \ref{defn:aux_operators}. Given $f \in \mathcal{H}$, $\phi\in\mathcal{E}$, and $z\in\rho(A_0)$, we would like to solve the following system of linear equations
\begin{equation}\label{eqn:abstract_system}
    \begin{cases}
        (\widehat{A}-z) u = f, \\
        \Gamma_0 u = \phi.
    \end{cases}
\end{equation}
The system bears resemblance to BVPs, with one equation on the main Hilbert space $\mathcal{H}$ and another on the boundary space $\mathcal{E}$. Here, $\Gamma_0$ has the interpretation of the (Dirichlet) trace, since by definition $\Gamma_0 \Pi \phi = \phi$ and $\Pi$ will be the harmonic lift in Section \ref{sect:bdry_triple_prelim}.

Choosing $\Lambda$ to be the DtN map from Section \ref{sect:bdry_triple_prelim}, the identity $\Lambda = \Gamma_1 \Pi$ then implies that $\Gamma_1$ has the interpretation of the Neumann trace.

As for $S(z)$, \cite[Theorem 3.2]{ryzhov2009} says that the system has a unique solution $u_z^{f,\phi} = (A_0 - z)^{-1} f + (I - zA_0^{-1})^{-1}\Pi \phi$, where the two terms solve the following systems respectively:

\vspace{.5em}

    \begin{minipage}{.45\linewidth}
    \begin{equation*}
        \begin{cases}
            (\widehat{A}-z)u = f, \\
            \Gamma_0 u=0.
        \end{cases}
    \end{equation*}
    \end{minipage}
    \begin{minipage}{.5\linewidth}
    \begin{equation}
        \begin{cases}
            (\widehat{A}-z)u = 0, \\
            \Gamma_0 u = \phi.
        \end{cases}
    \end{equation}
    \end{minipage}

\vspace{1em}

We then set $S(z)$ to be the operator solving the second system, justifying the name ``solution operator". One should compare this with the BVP for the harmonic lift in the previous section. Notice that $S(z)$ is not merely any generalization of $\Pi$ from $z=0$ to $z\in\rho(A_0)$ in the sense that $S(0) = \Pi$, but one with an additional property that the dependence on $z$ is reflected explicitly in the first equation of our system, $(\widehat{A}-z)u=0$.

Combining the interpretations for $\Gamma_1$ and $S(z)$, we see that $M(z) = \Gamma_1 S(z)$ could be interpreted as the DtN map with spectral parameter $z$. Similarly to $S(z)$, we have $M(0) = \Lambda$.

Returning to our setting, we have four triples of interest:

\begin{itemize}
    \item (Full cube) $( A^{(\uptau)}_{\varepsilon, 0}, \Lambda^{(\uptau)}_\varepsilon , \Pi^{(\uptau)} )$ with $\mathcal{H}= L^2(Q)$ and $\mathcal{E} = L^2(\Gamma_\interior) \oplus L^2(\Gamma_\ls)$.
    
    \item (Stiff interior) $( A^{\text{stiff-int}, (\uptau)}_{\varepsilon, 0}, \Lambda^{\text{stiff-int}, (\uptau)}_\varepsilon , \Pi^{\text{stiff-int}, (\uptau)} )$ with $L^2(Q_{\text{stiff-int}})$ and boundary space $L^2(\Gamma_\interior)$.
    
    \item (Soft annulus) $( A^{\text{soft}, (\uptau)}_{0}, \Lambda^{\text{soft}, (\uptau)} , \Pi^{\text{soft}, (\uptau)} )$ with $L^2(Q_{\text{soft}})$ and boundary space $L^2(\Gamma_\interior) \oplus L^2(\Gamma_\ls)$.
    
    \item (Stiff landscape) $( A^{\text{stiff-ls}, (\uptau)}_{\varepsilon, 0}, \Lambda^{\text{stiff-ls}, (\uptau)}_\varepsilon , \Pi^{\text{stiff-ls}, (\uptau)} )$ with $L^2(Q_{\text{stiff-ls}})$ and boundary space $L^2(\Gamma_\ls)$.
\end{itemize}

Now apply Definition \ref{defn:aux_operators} to get the following auxiliary operators

\begin{itemize}
    \item (Full cube) $\widehat{A}^{(\uptau)}_\varepsilon$, $\Gamma_0^{(\uptau)}$, $\Gamma^{(\uptau)}_{\varepsilon,1}$, $S^{(\uptau)}_\varepsilon(z)$, and $M^{(\uptau)}_\varepsilon(z)$.
    
    \item (Stiff interior) $\widehat{A}^{\text{stiff-int}, (\uptau)}_\varepsilon$, $\Gamma_0^{\text{stiff-int}, (\uptau)}$, $\Gamma^{\text{stiff-int}, (\uptau)}_{\varepsilon,1}$, $S^{\text{stiff-int}, (\uptau)}_\varepsilon(z)$, and $M^{\text{stiff-int}, (\uptau)}_\varepsilon(z)$.
    
    \item (Soft annulus) $\widehat{A}^{\text{soft}, (\uptau)}$, $\Gamma_0^{\text{soft}, (\uptau)}$, $\Gamma^{\text{soft}, (\uptau)}_{1}$, $S^{\text{soft}, (\uptau)}(z)$, and $M^{\text{soft}, (\uptau)}(z)$.
    
    \item (Stiff landscape) $\widehat{A}^{\text{stiff-ls}, (\uptau)}_\varepsilon$, $\Gamma_0^{\text{stiff-ls}, (\uptau)}$, $\Gamma^{\text{stiff-ls}, (\uptau)}_{\varepsilon,1}$, $S^{\text{stiff-ls}, (\uptau)}_\varepsilon(z)$, and $M^{\text{stiff-ls}, (\uptau)}_\varepsilon(z)$.
\end{itemize}

\begin{rmk}
    Our main model operator $A^{(\uptau)}_\varepsilon$ defined in Section \ref{sect:mainmodel_defn} is \textit{not} $\widehat{A}^{(\uptau)}_\varepsilon$, but as we see shortly, will coincide with an operator denoted by $\widehat{A}^{(\uptau)}_{\varepsilon, 0,I}$. $\widehat{A}^{(\uptau)}_{\varepsilon, 0,I}$ will be derived from $\widehat{A}^{(\uptau)}_\varepsilon$.
\end{rmk}

In the next section we will discuss some extra properties that arise from our specific setup. Here, we collect some properties which are applicable to a general boundary triple $(A_0,\Lambda,\Pi)$. Some of these have already been used to motivate the definition of the triple.

\begin{prop}[Properties of auxiliary operators]\label{prop:auxops_properties}
    Let $(A_0,\Lambda,\Pi)$ be a boundary triple with spaces $\mathcal{H}$ and $\mathcal{E}$. Construct the operators $\widehat{A}$, $\Gamma_0$, $\Gamma_1$, $S(z)$, and $M(z)$. Let $z \in \rho(A_0)$. Then,
    \begin{enumerate}
        \item $\text{ker}(\Gamma_0) = \mathcal{D}(A_0)$, and $\text{ran}(S(z)) = \text{ker}(\widehat{A}-z)$.
        \item $\Gamma_0 S(z) = I_\mathcal{E}$. In particular, since $S(0) = \Pi$, we have $\Gamma_0 S(0) = \Gamma_0 \Pi = I_\mathcal{E}$.
        \item $\Lambda = \Gamma_1 \Pi$, and $\Pi^* = \Gamma_1 A_0^{-1}$.
        \item $S(z) = (I-zA_0^{-1})^{-1}\Pi$. In particular, $S(\overline{z})^* = \Gamma_1 (A_0-z)^{-1}$.
        \item $M(z) = \Lambda + z\Pi^* (I-zA_0^{-1})^{-1} \Pi$. In particular, $M(0) = \Lambda$, and $M(z)^* = M(\overline{z})$.
        \item $\rho(A_0) \ni z \mapsto M(z)$ is an analytic operator-valued function where the operators are closed in $\mathcal{E}$ and have common ($z$-independent) domain $\mathcal{D}(\Lambda)$.
        \item For $z,\zeta \in \rho(A_0)$, $M(z)-M(\zeta)$ is bounded, and $M(z) - M(\zeta) = (z-\zeta)S(\overline{z})^* S(\zeta)$. In particular, $\text{Im}M(z) := \text{Im}(M(z) - M(0)) = \text{Im}(z) S(\overline{z})^* S(\overline{z})$.
        \item $I_{\text{ker}(\widehat{A}-z)} \subset S(z)\Gamma_0$.
    \end{enumerate}
\end{prop}

\begin{proof}
    We refer the reader to \cite[Section 3]{ryzhov2009} for the proofs of these identities.
\end{proof}

One more property deserves mention. Because of its importance we place it in a separate result.

\begin{thm}(Green's second identity)\label{thm:greens_id}
    For all $u,v \in \mathcal{D}(\Gamma_1) = \mathcal{D}(A_0) \dot{+} \Pi\mathcal{D}(\Lambda)$, we have
    \begin{align}
        (\widehat{A}u,v)_{\mathcal{H}} - (u,\widehat{A}v)_{\mathcal{H}} = (\Gamma_1 u, \Gamma_0 v)_{\mathcal{E}} - (\Gamma_0 u, \Gamma_1 v)_{\mathcal{E}}.
    \end{align}
\end{thm}

\begin{proof}
    See \cite[Theorem 3.6]{ryzhov2009}.
\end{proof}

We now outline the key steps in constructing the operator $\widehat{A}_{\beta_0,\beta_1}$ (see \cite[Section 4-5]{ryzhov2009} for details). Given $f \in \mathcal{H}$, $\phi\in\mathcal{E}$, and $z\in\rho(A_0)$, we would like to uniquely solve the following system:
\begin{equation}\label{eqn:abstract_system2}
    \begin{cases}
        (\widehat{A}-z) u = f, \\
        (\beta_0 \Gamma_0 + \beta_1 \Gamma_1) u = \phi.
    \end{cases}
\end{equation}

The biggest set where this would make sense is $u \in \mathcal{D}(\Gamma_1)$. That is, $u = A_0^{-1}g + \Pi \varphi$, with $g \in \mathcal{H}$ and $\varphi \in \mathcal{D}(\Lambda)$. Furthermore, we would like to make sense of $\beta_0$, $\beta_1$ not only as numbers, but also as \textit{operators} on $\mathcal{E}$, as doing so would allow us to greatly expand our interpretation of a BVP. To figure out reasonable conditions on $\beta_0$, $\beta_1$, we observe that for $u$ as above, 
\begin{align}
    (\beta_0 \Gamma_0 + \beta_1 \Gamma_1) (A_0^{-1}g + \Pi \varphi) = \beta_1 \Pi^* g + (\beta_0 + \beta_1 \Lambda) \varphi.
\end{align}

Therefore, we make the following assumptions:
\begin{defn}
    We assume that $\beta_0$ and $\beta_1$ are linear operators on $\mathcal{E}$ such that $\mathcal{D}(\beta_0) \supset \mathcal{D}(\Lambda)$, $\beta_1 \in \mathcal{L}(\mathcal{E})$, and $\beta_0 + \beta_1 \Lambda$ is closable. 
\end{defn}

The closability condition has been added because in what follows, we would like $\overline{\beta_0 + \beta_1 \Lambda}$, or equivalently $\overline{\beta_0 + \beta_1 M(z)}$ by Proposition \ref{prop:auxops_properties}(5), to be boundedly invertible. As a byproduct, we have expanded our solution space from $u\in \mathcal{D}(\Gamma_1) = \{ u=A_0^{-1}g + \Pi \varphi \text{ }|\text{ } g \in \mathcal{H}, \varphi \in \mathcal{D}(\Lambda) \}$ to 
\begin{align*}
    \mathcal{H}_{\overline{\beta_0 + \beta_1 \Lambda}} := \{ u=A_0^{-1}g + \Pi \varphi \text{ }|\text{ } g \in \mathcal{H}, \varphi \in \mathcal{D}(\overline{\beta_0 + \beta_1 \Lambda}) \}.    
\end{align*}
This space, together with the closability of $\beta_0 + \beta_1 \Lambda$, enable the subsequent steps in the construction: 

\begin{enumerate}
    \item $\mathcal{H}_{\overline{\beta_0 + \beta_1 \Lambda}}$ is a Hilbert space with norm $\| u \|_{\overline{\beta_0 + \beta_1 \Lambda}}^2 := \|g \|_\mathcal{H}^2 + \| \varphi \|_{\mathcal{E}}^2 + \|(\overline{\beta_0 + \beta_1 \Lambda}) \varphi \|_{\mathcal{E}}^2$, 
    
    \item $\beta_0 \Gamma_0 + \beta_1 \Gamma_1$ may be extended to a bounded operator from $(\mathcal{H}_{\overline{\beta_0 + \beta_1 \Lambda}}, \| \cdot \|_{\overline{\beta_0 + \beta_1 \Lambda}})$ to $\mathcal{E}$.
    
    \item If $\overline{\beta_0 + \beta_1 M(z)}$ defined on $\mathcal{D}(\overline{\beta_0 + \beta_1 M(z)}) = \mathcal{D}(\overline{\beta_0 + \beta_1 \Lambda})$ is boundedly invertible on $\mathcal{E}$, then the system (\ref{eqn:abstract_system2}) has a unique solution in $\mathcal{H}_{\overline{\beta_0 + \beta_1 \Lambda}}$.
    
    \item There exist an operator $\widehat{A}_{\beta_0,\beta_1}$ constructed from $\beta_0$, $\beta_1$, and the triple $(A_0,\Lambda,\Pi)$:
    \begin{thm}\label{thm:kreins_formula}
        (\cite[Theorem 5.5]{ryzhov2009}) Assume $z\in\rho(A_0)$ is such that $\overline{\beta_0 + \beta_1 M(z)}$ defined on $\mathcal{D}(\overline{\beta_0 + \beta_1 \Lambda})$ is boundedly invertible on $\mathcal{E}$. Define
        \begin{align}
            R_{\beta_0,\beta_1}(z) := (A_0 - z)^{-1} + S(z) Q_{\beta_0,\beta_1}(z) S(\overline{z})^*,
        \end{align}
        where $Q_{\beta_0,\beta_1}(z) := - (\overline{\beta_0 + \beta_1 M(z)})^{-1} \beta_1$. Then $R_{\beta_0,\beta_1}(z)$ is the resolvent at $z$ of a closed densely defined operator $\widehat{A}_{\beta_0,\beta_1}$ on $\mathcal{H}$. Its domain satisfies the following inclusion
        \begin{align}
            \mathcal{D}(\widehat{A}_{\beta_0,\beta_1}) \subset \{ u \in \mathcal{H}_{\overline{\beta_0 + \beta_1 \Lambda}}\text{ }|\text{ } (\beta_0 \Gamma_0 + \beta_1 \Gamma_1)u = 0 \} = \text{ker}(\beta_0 \Gamma_0 + \beta_1 \Gamma_1).
        \end{align}
        Furthermore, we have $\widehat{A}_{\beta_0,\beta_1} \subset \widehat{A}$. That is, $\widehat{A}_{\beta_0,\beta_1}u = \widehat{A}u$ whenever $u\in\mathcal{D}(\widehat{A}_{\beta_0,\beta_1})$.
    \end{thm}
\end{enumerate}

\begin{rmk} 
\begin{itemize}
    \item We refer to the formula for $(\widehat{A}_{\beta_0,\beta_1}-z)^{-1} \equiv R_{\beta_0,\beta_1}(z)$ as \textit{Krein's formula}.
    
    \item We did not give a complete description of $\mathcal{D}(\widehat{A}_{\beta_0,\beta_1})$. The best we have is $\mathcal{D}(\widehat{A}_{\beta_0,\beta_1}) = \text{ran}(R_{\beta_0,\beta_1}(z))$, where the RHS can be expressed in the triple $(A_0,\Lambda,\Pi)$ by Krein's formula. We also note here that $\mathcal{D}(\widehat{A}_{\beta_0,\beta_1})$ fits into the following chain of inclusions:
    $$\text{ker}(\Gamma_0) \cap \text{ker}(\Gamma_1) \subset \mathcal{D}(\widehat{A}_{\beta_0,\beta_1}) \subset \text{ker}(\beta_0 \Gamma_0 + \beta_1 \Gamma_1) \subset \mathcal{H}_{\overline{\beta_0 + \beta_1 \Lambda}} \subset \mathcal{D}(A) \subset \mathcal{H}.$$
    
    \item We do not claim that $\widehat{A}_{\beta_0,\beta_1}$ is self-adjoint. See \cite[Corollary 5.8]{ryzhov2009} for a sufficient condition.
    
    \item The construction of the closed operator $\widehat{A}_{\beta_0,\beta_1}$ from resolvents is a general result of ``pseudoresolvents" which can be found in \cite[Chapter VIII.1.1]{kato_book}.
    
    \item For the purposes of this paper, it is important that $\beta_0$ and $\beta_1$ are allowed to depend on $z$. Correspondingly, the operator $\widehat{A}_{\beta_0,\beta_1}$ depends on $z$ as well. \qedhere
\end{itemize}
\end{rmk}

\begin{nts}
    (On the issue of pesudoresolvents)
    The reference \cite[Chapter 4, Proposition 1.6]{engel_nagel_short} (Engel-Nagel) is just \cite[Chapter VIII.1.1]{kato_book} (Kato) but explained better. To summarize the discussion: the main issue here not existence, since we can take $R(z)^{-1} + zI$. Rather, the issue is consistency ... how to we know that $R(z)^{-1} + zI$ and $R(w)^{-1} + wI$ defines the same operator? (The family $\{R(z)\}_z$ needs to satisfy the first resolvent identity for every pair $R(z)$ and $R(w)$.)
\end{nts}

Theorem \ref{thm:kreins_formula} says that for $f\in\mathcal{H}$, the equation $(\widehat{A}_{\beta_0,\beta_1}-z)u = f$ can be solved uniquely if and only if the same holds for the system
\begin{equation}\label{eqn:abstract_system_b0b1}
    \begin{cases}
        (\widehat{A}-z) u = f, \\
        (\beta_0 \Gamma_0 + \beta_1 \Gamma_1) u = 0.
    \end{cases}
\end{equation}

A solution to this system implies a ``weak solution" in the sense of \cite[Definition 3.8]{ryzhov2009}, which coincides with the typical definition of a weak solution. Therefore we can relate the operators constructed in this way to a typical BVP, such as our main model in Section \ref{sect:mainmodel_defn}. To be precise, for the case $\beta_0 = 0$ and $\beta_1 = I$, we conclude that

\begin{cor}\label{cor:main_model_equiv}
    $A_{\varepsilon}^{(\uptau)} = \widehat{A}_{\varepsilon,0,I}^{(\uptau)}$, and $(A_{\varepsilon}^{(\uptau)}-z)^{-1} = (A_{\varepsilon,0}^{(\uptau)}-z)^{-1} - S_\varepsilon^{(\uptau)}(z) M_\varepsilon^{(\uptau)}(z)^{-1} S_\varepsilon^{(\uptau)}(\overline{z})^*$ whenever $z\in \C\setminus\R$.
\end{cor}

\begin{proof}
    $A_{\varepsilon}^{(\uptau)}$ is self-adjoint, and so is $\widehat{A}_{\varepsilon,0,I}^{(\uptau)}$ by \cite[Corollary 5.8]{ryzhov2009}. Now consider the resolvents of both operators at $z = i$. The range of $(\widehat{A}_{\varepsilon,0,I}^{(\uptau)} - i)^{-1}$ is $\mathcal{D}(\widehat{A}_{\varepsilon,0,I}^{(\uptau)})$, and is also the set of solutions to \eqref{eqn:abstract_system_b0b1} for some $f\in \mathcal{H}$. Similarly, the range of $(A_{\varepsilon}^{(\uptau)} - i)^{-1}$ is $\mathcal{D}(A_{\varepsilon}^{(\uptau)})$, and is also the set of weak solutions to \eqref{eqn:abstract_system_b0b1}. By the preceding paragraph, we have $\mathcal{D}(\widehat{A}_{\varepsilon,0,I}^{(\uptau)}) \subset \mathcal{D}(A_{\varepsilon}^{(\uptau)})$, and so $\widehat{A}_{\varepsilon,0,I}^{(\uptau)} \subset A_{\varepsilon}^{(\uptau)}$. Since self-adjoint operators are maximally symmetric, they must be equal.
\end{proof}

\subsection{Properties of the triple arising from our setup}\label{sect:specific_setup}

First, we state the actions of the auxiliary operators in a convenient form. We skip $\widehat{A}$ and $\Gamma_0$ as they are just null extensions of $A_0$ and the left inverse of $\Pi$ respectively. We will only do this for the the triple on $Q_{\text{stiff-ls}}$, omitting similar statements for $Q_{\text{stiff-int}}$ and $Q_{\text{soft}}$ for brevity.

By the identity $\Lambda_\varepsilon^{\text{stiff-ls},(\uptau)} = \Gamma_{\varepsilon,1}^{\text{stiff-ls},(\uptau)} \Pi^{\text{stiff-ls},(\uptau)}$, we see that $\Gamma_{\varepsilon,1}^{\text{stiff-ls}, (\uptau)}$ takes $u = \Pi^{\text{stiff-ls}, (\uptau)} \phi \in \Pi^{\text{stiff-ls}, (\uptau)} \mathcal{D}(\Lambda_\varepsilon^{\text{stiff-ls},(\uptau)})$ to $\varepsilon^{-2} \partial^{(\uptau)}_{n_\stls,\ls} u$, where $u$ solves the BVP
\begin{equation*}
    \begin{cases}
        -(\nabla + i\uptau)^2 u = 0 & \text{ in $Q_{\text{stiff-ls}}$,} \\
        u = \phi &\text{ on $\Gamma_\ls$,} \\
        u \text{ periodic} &\text{ on $\partial Q$.}
    \end{cases}
\end{equation*}

\begin{rmk}
    The action of $\Gamma_1$ is characterized by two equations, $\Lambda = \Gamma_1 \Pi$ and $\Pi^* = \Gamma_1 A_0^{-1}$. The above description only discusses $\Lambda = \Gamma_1 \Pi$.
\end{rmk}

The operator $S_\varepsilon^{\text{stiff-ls},(\uptau)}(z) = (I - z(A_{\varepsilon,0}^{\text{stiff-ls},(\uptau)})^{-1}) \Pi^{\text{stiff-ls},(\uptau)}$ takes $\phi \in L^2(\Gamma_\ls)$ to $u_\phi \in \mathcal{D}(A_{\varepsilon,0}^{\text{stiff-ls},(\uptau)}) \dot{+}$ $\text{ran}(\Pi^{\text{stiff-ls},(\uptau)}) \subset L^2(Q_{\text{stiff-ls}})$,
where $u=u_\phi$ solves the BVP (in the sense of system \eqref{eqn:abstract_system})
\begin{equation*}
    \begin{cases}
        ( -\varepsilon^{-2}(\nabla + i\uptau)^2 -z ) u = 0, \\
        \Gamma_0^{\text{stiff-ls},(\uptau)} u = \phi. \\
    \end{cases}
\end{equation*}

The operator $M_\varepsilon^{\text{stiff-ls},(\uptau)}(z) = \Gamma_{\varepsilon,1}^{\text{stiff-ls},(\uptau)} S_\varepsilon^{\text{stiff-ls},(\uptau)}(z)$ takes $\phi \in H^1(\Gamma_\ls) = \mathcal{D}(\Lambda_\varepsilon^{\text{stiff-ls},(\uptau)})$ to $\varepsilon^{-2}\partial^{(\uptau)}_{n_\stls,\ls} u$, where $u = u_\phi$ solves the BVP
\begin{equation*}
    \begin{cases}
        (-\varepsilon^{-2}(\nabla + i\uptau)^2 - z ) u = 0 & \text{ in $Q_{\text{stiff-ls}}$,} \\
        u = \phi &\text{ on $\Gamma_\ls$,} \\
        u \text{ periodic} &\text{ on $\partial Q$.}
    \end{cases}
\end{equation*}

\begin{nts}
    I chose $Q_\stls$ instead of $Q_\soft$, so that we can see the dependence on $\eps$.
\end{nts}

In view of keeping our notation compact, we make the following convention:

\begin{rmk}[On notation]
    In the remainder of the text, we will often abuse notation and write for instance, the operator $``P_1AP_2"$, for projections $P_1$ and $P_2$, to mean any one of the following:
    \begin{itemize}
        \item the composition of the three operators, $P_1 A P_2 : \mathcal{H} \rightarrow \mathcal{H}$,
        
        \item the compression $r P_1A|_{P_2\mathcal{H}}: P_2{\mathcal{H}}\rightarrow P_1\mathcal{H}$, where $r:\mathcal{H} \rightarrow P_1\mathcal{H}$ is the restriction operator,
        
        \item the operator $P_1A|_{P_2\mathcal{H}}: P_2{\mathcal{H}}\rightarrow \mathcal{H}$, which is equal to the composition of the embedding $i:P_1\mathcal{H}\rightarrow \mathcal{H}$ with the compression,
        
        \item the operator $rP_1A|_{P_2\mathcal{H}} + rP_10|_{P_2^{\perp} \mathcal{H}}: \mathcal{H}\rightarrow P_1\mathcal{H}$, which is the null extension of the compression to the full space. \qedhere
    \end{itemize}
\end{rmk}

Using the projections on $\mathcal{H}$ and $\mathcal{E}$ (Definition \ref{defn:projections}), we now discuss the relation between auxiliary operators of different triples. The first observation follows directly from the definition of the triples:
\begin{align}\label{eqn:pi_op_sum}
    \Pi^{\text{stiff-int},(\uptau)} = P_{\text{stiff-int}} \Pi^{(\uptau)} \mathcal{P}_\interior, \quad
    \Pi^{\text{soft},(\uptau)} = P_{\text{soft}} \Pi^{(\uptau)}, \quad \text{and }
    \Pi^{\text{stiff-ls},(\uptau)} = P_{\text{stiff-ls}} \Pi^{(\uptau)} \mathcal{P}_\ls.
\end{align}

Secondly, by our description of the action of $S(z)$, we have
\begin{align}\label{eqn:s_op_sum}
    S_\varepsilon^{\text{stiff-int},(\uptau)}(z) &= P_{\text{stiff-int}} S_\varepsilon^{(\uptau)}(z) \mathcal{P}_\interior, \quad
    S^{\text{soft},(\uptau)}(z) = P_{\text{soft}} S_\varepsilon^{(\uptau)}(z), \quad \text{and } \nonumber \\
    S_\varepsilon^{\text{stiff-ls},(\uptau)}(z) &= P_{\text{stiff-ls}} S_\varepsilon^{(\uptau)}(z) \mathcal{P}_\ls.
\end{align}

This could be proven directly, for instance for $S^{\text{soft},(\uptau)}(z)$, by noting that
\begin{alignat*}{2}
    P_\text{soft} S_\varepsilon^{(\uptau)}(z) &= P_\text{soft} (I - z (A_0^{(\uptau)} - z)^{-1})\Pi^{(\uptau)} && \text{by definition of $S(z)$,}\\
    &= P_\text{soft} (I - z (A_0^{(\uptau)} - z)^{-1}) P_\text{soft} \Pi^{(\uptau)} && \text{$L^2(Q_{\text{soft}})$ is an invariant subspace for $A_0^{(\uptau)}$,}\\
    &= (I_{L^2(Q_\text{soft})} - z (A_0^{\text{soft},(\uptau)} - z)^{-1})\Pi^{\text{soft}, (\uptau)} \quad && \text{by construction of $A_0^{(\uptau)}$ and by (\ref{eqn:pi_op_sum}).}
\end{alignat*}

As for $M(z)$, we have

\begin{prop}\label{prop:m_op_sum}
    For $z\in\rho(A_{\varepsilon,0}^{(\uptau)})$, the following identity holds
    \begin{align}
        M_\varepsilon^{(\uptau)}(z) = M_\varepsilon^{\text{stiff-int},(\uptau)}(z) \mathcal{P}_\interior + M^{\text{soft},(\uptau)}(z) +  M_\varepsilon^{\text{stiff-ls}, (\uptau)}(z) \mathcal{P}_\ls.
    \end{align}
\end{prop}

\begin{proof}
    We will drop $\uptau$ and $\varepsilon$. Let $\phi \in H^1(\Gamma_\interior)$ and $\varphi \in H^1(\Gamma_\ls)$. We see that
    \begin{itemize}
        \item $M^{\text{stiff-int}}(z) = \Gamma_1^{\text{stiff-int}} S^{\text{stiff-int}}(z)$ takes $\phi$ to $\varepsilon^{-2}\partial_{n_\stin,\interior}^{(\uptau)} u_\phi$,
        \item $M^{\text{soft}}(z) = \Gamma_1^{\text{soft}} S^{\text{soft}}(z)$ takes $\phi + \varphi$ to $\partial^{(\uptau)}_{n_\soft,\interior}u_{\phi,\varphi} + \partial^{(\uptau)}_{n_\soft,\ls} u_{\phi,\varphi}$,
        \item $M^{\text{stiff-ls}}(z) = \Gamma_1^{\text{stiff-ls}} S^{\text{stiff-ls}}(z)$ takes $\varphi$ to $\varepsilon^{-2}\partial_{n_\stls,\ls}^{(\uptau)} u_\varphi$,
    \end{itemize}
    where $u = u_{\text{stiff-int}} + u_{\text{soft}} + u_{\text{stiff-ls}} = u_\phi + u_{\phi,\varphi} + u_\varphi$ solves the BVP
    
    \begin{equation}
        \begin{cases}
            (-\varepsilon^{-2}(\nabla + i\uptau)^2 - z) u_{\phi} = 0 & \text{ in $Q_{\text{stiff-int}}$,} \\
            (-(\nabla + i\uptau)^2 - z) u_{\phi,\varphi} = 0 & \text{ in $Q_{\text{soft}}$,} \\
            (-\varepsilon^{-2}(\nabla + i\uptau)^2 - z) u_{\varphi} = 0 & \text{ in $Q_{\text{stiff-ls}}$,} \\
            u_{\phi} = u_{\phi,\varphi} = \phi &\text{ on $\Gamma_\interior$,} \\
            u_{\varphi} = u_{\phi,\varphi} = \varphi &\text{ on $\Gamma_\ls$,} \\
            u_{\varphi} \text{ periodic} &\text{ on $\partial Q$.}
        \end{cases}
    \end{equation}
    The rule $(\phi,\varphi) \mapsto \varepsilon^{-2}\partial_{n_\stin,\interior}^{(\uptau)} u_\phi + \partial^{(\uptau)}_{n_\soft,\interior}u_{\phi,\varphi} + \partial^{(\uptau)}_{n_\soft,\ls} u_{\phi,\varphi} + \varepsilon^{-2}\partial_{n_\stls,\ls}^{(\uptau)} u_\varphi$ is precisely the action of $M(z)$.
\end{proof}

\begin{rmk}
    Note that $M^{\text{stiff-int}}(z) \neq \mathcal{P}_\interior M(z) \mathcal{P}_\interior$. The LHS is $\Gamma_1^{\text{stiff-int}} S^{\text{stiff-int}}(z)$, which takes $\phi \in H^1(\Gamma_\interior)$ to $\varepsilon^{-2} \partial^{(\uptau)}_{n_\stin,\interior} u_\phi$. The RHS takes $\phi \in H^1(\Gamma_\interior)$ to $\varepsilon^{-2} \partial^{(\uptau)}_{n_\stin,\interior} u_\phi + \partial^{(\uptau)}_{n_\soft,\interior} u_{\phi,0}$. Even though we are confronted with this asymmetry, the above proposition assures us that the additive structure of $M(z)$ remains. The additivity is exploited to great effect in \cite{eff_behavior}.
\end{rmk}

Finally, we discuss the dependence of the auxiliary operators on $\varepsilon$ and $\uptau$, and $z$. To obtain estimates that are uniform in $z$, we will restrict our choice of $z$ to the following set.

\begin{defn}
    Fix $\sigma>0$ and a compact subset of $K \subset \mathbb{C}$. Let $K_\sigma$ be the compact subset of $K$ that is $\sigma$ distance away from the real line, that is, $K_\sigma = \{ z \in \mathbb{C} : z \in K, \text{dist}(z,\mathbb{R}) \geq \sigma \}$.
\end{defn}

\begin{lem}\label{lem:soln_op_est}
    We have $S^{\text{stiff-int},(\uptau)}_\varepsilon (z) - \Pi^{\text{stiff-int},(\uptau)} = O(\varepsilon^2)$, $S^{\text{stiff-ls},(\uptau)}_\varepsilon (z) - \Pi^{\text{stiff-ls},(\uptau)} = O(\varepsilon^2)$ and $S^{\text{soft},(\uptau)}(z) - \Pi^{\text{soft},(\uptau)} = O(1)$ in operator norm, uniformly in $\uptau \in Q'$ and $z\in K_\sigma$.
\end{lem}

\begin{proof}
    This is a direct consequence of the formula $S(z)=(I + z(A_0-z)^{-1}) \Pi$.
\end{proof}

In terms of estimates that are uniform over $\varepsilon$, $\uptau$, and $z\in K_\sigma$, recall that we have already provided one for the decoupling $A_0$ in Proposition \ref{prop:decoupling_properties}, and one for the lift $\Pi$ in Proposition \ref{prop:lift_properties}.

Similarly to the solution operator $S(z)$, we may ask for a simplification of $M(z)$ up to $O(\varepsilon^2)$. Recall the notation for the unweighted decoupling $\widetilde{A}_0^{\text{stiff-int},(\uptau)}, \widetilde{A}_0^{\text{stiff-ls},(\uptau)}$ and unweighted DtN maps $\widetilde{\Lambda}^{\text{stiff-int},(\uptau)}$, $ \widetilde{\Lambda}^{\text{stiff-ls}, (\uptau)}$.

\begin{lem}\label{lem:m_op_est}
    For $\bigstar \in \{ \text{stiff-int}, \text{stiff-ls}\}$, we have 
    \begin{align}
        M_\varepsilon^{\bigstar,(\uptau)}(z) = \varepsilon^{-2} \widetilde{\Lambda}^{\bigstar, (\uptau)} + z (\Pi^{\bigstar,(\uptau)})^* \Pi^{\bigstar,(\uptau)} + O(\varepsilon^2),
    \end{align}
    in the operator norm, where the estimate is uniform over $\uptau \in Q'$ and $z \in K_\sigma$.
\end{lem}

\begin{proof}
    We omit $\bigstar$. Since $\widetilde{A}_0^{(\uptau)} = \varepsilon^2 A_{\varepsilon,0}^{(\uptau)}$, we get $\varepsilon^2 (\widetilde{A}_0^{(\uptau)})^{-1} = (A_{\varepsilon,0}^{(\uptau)})^{-1}$. Hence
    \begin{align}
        M_\varepsilon^{(\uptau)}(z) 
        &= \varepsilon^{-2} \widetilde{\Lambda}^{(\uptau)} + z (\Pi^{(\uptau)})^* \left(I-z\varepsilon^{2} \left(\widetilde{A}_0^{(\uptau)}\right)^{-1} \right)^{-1} \Pi^{(\uptau)} \nonumber \\
        &= \varepsilon^{-2} \widetilde{\Lambda}^{(\uptau)} + z (\Pi^{(\uptau)})^* \Pi^{(\uptau)} + O(\varepsilon^2).
    \end{align}
    The second equality follows from the Neumann series, which is justified by the uniform in $\uptau$ bounds for the decoupling and the lift from Propositions \ref{prop:decoupling_properties} and \ref{prop:lift_properties}, assuming $\varepsilon$ is small enough.
\end{proof}

\section{Norm-resolvent asymptotics}\label{sect:norm_resolvent_asymp}

After the long setup, we are now ready to begin the task of homogenization, by which we mean that a study of the norm-resolvent asymptotics of the main model operator $A_{\varepsilon}^{(\uptau)}$ of Section \ref{sect:mainmodel_defn}. We would like to identify an operator $\mathcal{A}_\text{hom}^{(\uptau)}$ that we will refer to as a \textit{homogenized operator}. To qualify as a ``homogenized" operator, we require that
\begin{itemize}
    \item $\mathcal{A}_\text{hom}^{(\uptau)}$ be self-adjoint on a possibly smaller subspace $L^2(Q_{\text{soft}})\oplus \widetilde{\mathcal{H}}$ of $L^2(Q)$.
    \item The dependence $\varepsilon$ is only allowed in the action of $\mathcal{A}_\text{hom}^{(\uptau)}$, on the stiff component. In particular, the subspace $\widetilde{\mathcal{H}}$, and the domain $\mathcal{D}( \mathcal{A}_\text{hom}^{(\uptau)} )$ must be independent of $\varepsilon$.
    \item $\mathcal{A}_\text{hom}^{(\uptau)}$ and $A_\eps^{(\uptau)}$ are asymptotically equivalent, as $\eps\downarrow 0$, in some specified topology.
\end{itemize}

\subsection{Decomposing the boundary space}

In this section, we will decompose the boundary space $\mathcal{E} = L^2(\Gamma_\interior) \oplus L^2(\Gamma_\ls)$ with respect to the spectral subspaces of the DtN map $\Lambda_\varepsilon^{(\uptau)}$. We give a brief explanation here on the underlying idea:

By the Krein's formula (Theorem \ref{thm:kreins_formula}), we turn our attention to the solution operator $S_\varepsilon^{(\uptau)}(z)$ and M-operator $M_\varepsilon^{(\uptau)}(z)$. Lemmas \ref{lem:soln_op_est} and \ref{lem:m_op_est} tells us that $M(z)$ is the badly behaved term of the two. Focusing on $M(z)$, problematic region of the resolvent set is located at $z=0$ and its vicinity. To see this, recall from Corollary \ref{cor:main_model_equiv} that $A_\eps^{(\tau)}$ has $\beta_0 = 0$ and $\beta_1 = I$, giving us $\beta_0 + \beta_1 M(z) = M(z)$, which the Krein's formula then assumes to be boundedly invertible. This however, becomes increasingly difficult as $\eps$ is small, because Lemma \ref{lem:m_op_est} shows that the term $\varepsilon^{-2}\widetilde{\Lambda}^{\bigstar}$ dominates when $\varepsilon$ is small. This suggests us to break the problem into two in the spectral picture: a compact neighborhood of $z=0$ and its complement. Thanks to $\Lambda$ having compact resolvent (Proposition \ref{prop:dtn_properties}), the spectral subspace of the former could be chosen to be finite dimensional, which greatly simplifies the analysis.

\begin{nts}
    \begin{itemize}
        \item ``Inversion of a `large' operator is small." This is false. We only have the lower bound $\frac{1}{\| T \|} \leq \| T^{-1} \|$. If we assume $T$ is self-adjoint, then we have $\|T^{-1}\| = \frac{1}{\text{dist}(0,\sigma(T))}$.

        \item The point $z=0$ is out of my reach as $0 \notin K_\sigma$, so I should not be bothering myself with $z=0$. This is not true. In fact, the point of the above paragraph is to explain why we should bother ourselves with $z=0$. To repeat: even if we take $z \in K_\sigma$, the identity $M(z) = \eps^{-2} \Lambda + z \Pi^* \Pi + O(\eps^2)$ will force us to look at the term $\eps^{-2}\Lambda = M(0)$. \qedhere
    \end{itemize}
\end{nts}

Recall the \textit{unweighted} DtN map in Proposition \ref{prop:dtn_properties}. We introduce the following notation:

\begin{defn}
    Let $\bigstar \in \{\text{stiff-int},\text{stiff-ls}\}$. From now on, we will only consider the first eigenvalue and eigenfunction pair. Therefore, we will drop the indices and write $\mu^{\bigstar,(\uptau)} := \mu_1^{\bigstar, (\uptau)}$ and $\psi^{\bigstar,(\uptau)} := \psi_1^{\bigstar,(\uptau)}$. Note that $\psi_1^{\stin,(\uptau)}$ and $\psi_1^{\stls,(\uptau)}$ are mutually orthogonal. Introduce the orthogonal projections $\mathcal{P}_{\bigstar}^{(\uptau)} := (\cdot, \psi^{\bigstar,(\uptau)})_{\mathcal{E}}\psi^{\bigstar,(\uptau)}$, $\mathcal{P}^{(\uptau)} := \mathcal{P}_{\text{stiff-int}}^{(\uptau)} \oplus \mathcal{P}_{\text{stiff-ls}}^{(\uptau)}$, and $\mathcal{P}_{\perp}^{(\uptau)} = I_{\mathcal{E}} - \mathcal{P}^{(\uptau)}$.
\end{defn}

\begin{rmk}
    \begin{itemize}
        \item (On notation) Note the use of calligraphic font for projections on $\mathcal{E}$. So $\mathcal{P}_{\text{stiff-int}}^{(\uptau)}$ should not be confused with $P_{\text{stiff-int}}$, which is a projection on $\mathcal{H}$.
        \item As Proposition \ref{prop:dtn_properties} does not assert the simplicity of $\mu_1^{\stls,(\tau)}$ for large $\tau$, we may for the moment pick any eigenfunction $\psi_1^{\stls,(\tau)}$. Proposition \ref{prop:cts_dependence} will show that $\psi_1^{\stls,(\tau)}$ can be chosen in a continuous manner, which will be assumed from that point on. \qedhere
    \end{itemize}
\end{rmk}

Recall that the unweighted DtN on the stiff-components, $\widetilde{\Lambda}^{\text{stiff-int},(\uptau)} \oplus \widetilde{\Lambda}^{\text{stiff-ls},(\uptau)}$ is self-adjoint with domain $H^1(\Gamma_\interior) \oplus H^1(\Gamma_\ls)$. With respect to the decomposition $\mathcal{E} = \mathcal{P}^{(\uptau)}_{\text{stiff-int}}\mathcal{E} \oplus \mathcal{P}^{(\uptau)}_{\text{stiff-ls}}\mathcal{E} \oplus \mathcal{P}^{(\uptau)}_{\perp}\mathcal{E}$, we may now write $\widetilde{\Lambda}^{\text{stiff-int},(\uptau)} \oplus \widetilde{\Lambda}^{\text{stiff-ls},(\uptau)}$ as
\begin{align}
    \widetilde{\Lambda}^{\text{stiff-int},(\uptau)} \oplus \widetilde{\Lambda}^{\text{stiff-ls},(\uptau)} = 
    \begin{pmatrix}
    \mu^{\text{stiff-int},(\uptau)} & 0 & 0\\
    0 & \mu^{\text{stiff-ls},(\uptau)} & 0\\
    0 & 0 & \mathcal{P}^{(\uptau)}_{\perp} \left(\widetilde{\Lambda}^{\text{stiff-int},(\uptau)} \oplus \widetilde{\Lambda}^{\text{stiff-ls},(\uptau)}\right) \mathcal{P}^{(\uptau)}_{\perp}
    \end{pmatrix}.
\end{align}

As for the (weighted) M-operator $M_\varepsilon^{(\uptau)}(z)$, we write its block operator representation with respect to the decompositions $\mathcal{E} = \mathcal{P}^{(\uptau)}_{\text{stiff-int}}\mathcal{E} \oplus \mathcal{P}^{(\uptau)}_{\text{stiff-ls}}\mathcal{E} \oplus \mathcal{P}^{(\uptau)}_{\perp}\mathcal{E}$ and $\mathcal{E} = \mathcal{P}^{(\uptau)}\mathcal{E} \oplus \mathcal{P}^{(\uptau)}_{\perp}\mathcal{E}$:
\begin{align}\label{eqn:m_matrix_decomp}
    M_\varepsilon^{(\uptau)}(z) = 
    \begin{matrix}
    & \mathcal{P}^{(\uptau)}_{\text{stiff-int}} \mathcal{E} & \mathcal{P}^{(\uptau)}_{\text{stiff-ls}} \mathcal{E} & \mathcal{P}^{(\uptau)}_\perp \mathcal{E} \\
    \mathcal{P}^{(\uptau)}_{\text{stiff-int}} \mathcal{E} & \tikzmark{a11}{$\mathbb{A}_{11}$} & \mathbb{A}_{12} & \tikzmark{b1}{$\mathbb{B}_{1}$}\\
    \mathcal{P}^{(\uptau)}_{\text{stiff-ls}} \mathcal{E} & \mathbb{A}_{21} & \mathbb{A}_{22} & \mathbb{B}_2\\
    \mathcal{P}^{(\uptau)}_\perp \mathcal{E} & \tikzmark{e1}{$\mathbb{E}_{1}$} & \mathbb{E}_2 & \tikzmark{d33}{$\mathbb{D}$}
    \end{matrix}
    = 
    \begin{matrix}
        & \mathcal{P}^{(\uptau)} \mathcal{E} & \mathcal{P}^{(\uptau)}_\perp \mathcal{E} \\
        \mathcal{P}^{(\uptau)} \mathcal{E}  & \tikzmark{a}{$\mathbb{A}$} & \tikzmark{b}{$\mathbb{B}$} \\
        \mathcal{P}^{(\uptau)}_\perp \mathcal{E} & \tikzmark{e}{$\mathbb{E}$} & \tikzmark{d}{$\mathbb{D}$}
    \end{matrix}.
\end{align}
\begin{tikzpicture}[remember picture,overlay]
    \draw[thick,decorate,decoration={calligraphic straight parenthesis}] ($(e1.south west)+(-.7em,0)$) -- ($(a11.north west)+(-.5em,0)$);
    \draw[thick,decorate,decoration={calligraphic straight parenthesis}] ($(b1.north east)+(+.43em,0)$) -- ($(d33.south east)+(+.6em,0)$);
\end{tikzpicture}
\begin{tikzpicture}[remember picture,overlay]
    \draw[thick,decorate,decoration={calligraphic straight parenthesis}] ($(e.south west)+(-.5em,0)$) -- ($(a.north west)+(-.5em,0)$);
    \draw[thick,decorate,decoration={calligraphic straight parenthesis}] ($(b.north east)+(+.5em,0)$) -- ($(d.south east)+(+.5em,0)$);
\end{tikzpicture}

\begin{lem}\label{lem:m_components_cts_ext}
    The components $\mathbb{A}$, $\mathbb{B}$, and $\mathbb{E}$ of $M_\varepsilon^{(\uptau)}(z)$ are all extendable to bounded operators on their respective spaces, where $z \in \rho(A^{(\uptau)}_{\varepsilon,0})$.
\end{lem}

\begin{proof}
    We will drop $\varepsilon$ and $\tau$. We modify the arguments of \cite[Section 3.2]{eff_behavior}. By Proposition \ref{prop:auxops_properties}(5), it suffices to check the claim for the case $z = 0$.

    We first check this for the operator $\mathbb{B} = \mathcal{P}\Lambda \mathcal{P}_\perp$. We claim that $\mathcal{P}_\perp \mathcal{D}(\Lambda) \subset \mathcal{D}(\Lambda)$. In other words, $\mathcal{D}(\mathbb{B})$ contains the set $\mathcal{D}(\Lambda)$ which is dense in $\mathcal{E}$. This is because if $(\phi,\varphi) \in \mathcal{D}(\Lambda)$, then $\mathcal{P}(\phi,\varphi) = (\mathcal{P}_{\text{stiff-int}}\phi, \mathcal{P}_{\text{stiff-ls}}\varphi) \in \text{span}\{\psi^{\text{stiff-int}}\} \oplus \text{span}\{\psi^{\text{stiff-ls}}\}$, and the eigenvectors are in $H^1$ on their respective spaces. Then, notice that $\mathcal{P}_\perp(\phi,\varphi)$ can be written as a linear combination of elements in $\mathcal{D}(\Lambda)$, as $\mathcal{P}_\perp(\phi,\varphi) = (\phi - \mathcal{P}_{\text{stiff-int}}\phi,\varphi - \mathcal{P}_{\text{stiff-int}}\varphi) = (\phi,\varphi) - \mathcal{P}(\phi,\varphi)$.
    
    Now suppose that $(\phi,\varphi) \in \mathcal{D}(\Lambda) = H^1(\Gamma_\interior) \oplus H^1(\Gamma_\ls) \subset \mathcal{D}(\mathbb{B})$, then its image under $\mathbb{B}$ is
    \begin{align}\label{eqn:nice_projection_exp}
        \mathcal{P}\Lambda \mathcal{P}_\perp (\phi,\varphi) &= 
        \left( \Lambda \mathcal{P}_\perp (\phi,\varphi), (\psi^{\text{stiff-int}}, \psi^{\text{stiff-ls}}) \right)_{\mathcal{E}} (\psi^{\text{stiff-int}}, \psi^{\text{stiff-ls}}) \nonumber \\
        &= \left( \mathcal{P}_\perp (\phi,\varphi), \Lambda (\psi^{\text{stiff-int}}, \psi^{\text{stiff-ls}}) \right)_{\mathcal{E}} (\psi^{\text{stiff-int}}, \psi^{\text{stiff-ls}}),
    \end{align}
    as $\Lambda$ is self-adjoint. Then using the Cauchy-Schwarz inequality, $\| \mathcal{P}_\perp \| \leq 1$, and that $\psi^{\text{stiff-int}}$ and $\psi^{\text{stiff-ls}}$ are normalized eigenfunctions, we deduce that 
    \begin{align}\label{eqn:nice_projection_bound}
        \| \mathcal{P}\Lambda \mathcal{P}_\perp (\phi,\varphi) \|_{\mathcal{E}} \leq
        \sqrt{2} \| \Lambda (\psi^{\text{stiff-int}}, \psi^{\text{stiff-ls}}) \|_{\mathcal{E}} \|(\phi,\varphi) \|_{\mathcal{E}}.
    \end{align}
    Since $H^1(\Gamma_\interior) \oplus H^1(\Gamma_\ls)$ is dense in $\mathcal{E}$, $\mathbb{B}$ admits a continuous extension to an operator $\mathcal{P}_\perp \mathcal{E} \rightarrow \mathcal{P}\mathcal{E}$. The same reasoning holds for $\mathbb{A} = \mathcal{P}\Lambda\mathcal{P}$ and $\mathbb{E}=\mathcal{P}_\perp \Lambda\mathcal{P}$.
\end{proof}

\begin{nts}
    I chose to normalize $\| \psi^{\stin} \|_{L^2(\Gamma_\interior)} = \| \psi^{\stls} \|_{L^2(\Gamma_\ls)} = 1$ instead of \newline $\| (\psi^{\stin}, \psi^{\stls}) \|_{\mathcal{E}} = 1$, because it makes the formulas for $\mathcal{A}_{\eps,\text{hom}}^{(\tau)}$ a little easier.
\end{nts}

\begin{quote}
    \textbf{We will henceforth write $\mathbb{A}$, $\mathbb{B}$, and $\mathbb{E}$ to mean its continuous extension to the full subspaces $\mathcal{P}^{(\uptau)}\mathcal{E}$ and $\mathcal{P}^{(\uptau)}_\perp \mathcal{E}$.}
\end{quote}

Note that $\mathbb{A}$, $\mathbb{B}$, and $\mathbb{E}$ depend on $\varepsilon$, $\uptau$, and $z$, since $M_\varepsilon^{(\uptau)}(z)$ does. In light of this, Lemma \ref{lem:m_components_cts_ext} is insufficient for our purposes: we would like to argue further why for $\mathbb{B}$ and $\mathbb{E}$, the RHS of (\ref{eqn:nice_projection_bound}) can be bounded by $C \|(\phi,\varphi) \|$, where $C$ is a uniform constant.

\begin{prop}\label{prop:m_components_unif_bound}
    The bound on $\|\mathbb{B}\|_{op}$ and $\|\mathbb{E}\|_{op}$ can be chosen independently of $z \in K_\sigma$, $\uptau \in Q'$, and $\varepsilon>0$. The bound on $\| \mathbb{A} \|_{op}$ can be chosen independently of $z \in K_\sigma$ and $\uptau \in Q'$.
\end{prop}

\begin{proof}
    Proposition \ref{prop:m_op_sum} permits us to address the ``soft" and ``stiff" parts individually. Proposition \ref{prop:auxops_properties}(5) allows us to split $M(z)$ into an unbounded part $\Lambda$ (which depends on $\varepsilon$ and $\uptau$), and a bounded part $z\Pi^* (I - zA_0^{-1})^{-1} \Pi$ (which depends on $\varepsilon$, $\uptau$, and $z$).
    
    We claim that the bounded part may be bounded uniformly in $\varepsilon$, $\uptau$, and $z$. Indeed, by Lemma \ref{lem:m_op_est}, it suffices to work with the ``soft" case. The claim then follows from Proposition \ref{prop:lift_properties}, and the assumption that $K_\sigma$ is compact. As the unbounded part does not depend on $z$, this proves assertion on the \textbf{independence on $z$}.
    
    Next we discuss the \textbf{independence on $\uptau$}. Without loss of generality, let us consider $\mathbb{B}=\mathcal{P}^{(\uptau)} \Lambda_\varepsilon^{(\uptau)} \mathcal{P}_\perp^{(\uptau)}$. In (\ref{eqn:nice_projection_bound}), we have shown that $\mathbb{B}$ has operator norm not exceeding
    \begin{align}\label{eqn:m_components_unif_bound_ineq}
        &\sqrt{2}\| \Lambda_\varepsilon^{(\uptau)}(\psi^{\text{stiff-int},(\uptau)}, \psi^{\text{stiff-ls},(\uptau)} )\|_{\mathcal{E}} \nonumber\\
        &\leq \sqrt{2} \left(
        \|\Lambda_\varepsilon^{\text{stiff-int},(\uptau)} \psi^{\text{stiff-int},(\uptau)} \|_{\mathcal{E}} +
        \|\Lambda_\varepsilon^{\text{stiff-ls},(\uptau)} \psi^{\text{stiff-ls},(\uptau)} \|_{\mathcal{E}} + 
        \|\Lambda^{\text{soft},(\uptau)} (\psi^{\text{stiff-int},(\uptau)}, \psi^{\text{stiff-ls},(\uptau)} ) \|_{\mathcal{E}} \right)  \nonumber\\
        &= \sqrt{2} \bigg( 
        |\varepsilon^{-2} \mu^{\text{stiff-int},(\uptau)}| \underbrace{\| \psi^{\text{stiff-int},(\uptau)} \|_{\mathcal{E}}}_{= 1} +
        |\varepsilon^{-2} \mu^{\text{stiff-ls},(\uptau)}| \underbrace{\| \psi^{\text{stiff-ls},(\uptau)} \|_{\mathcal{E}}}_{= 1}  \nonumber\\
        &\qquad\qquad +\|\Lambda^{\text{soft},(\uptau)} (\psi^{\text{stiff-int},(\uptau)}, \psi^{\text{stiff-ls},(\uptau)} ) \|_{\mathcal{E}} \bigg).
    \end{align}
    (Actually, the first two terms are absent for $\mathbb{B}$ and $\mathbb{E}$, as we see below, but we would like to include $\mathbb{A}$ for this discussion.) We apply a ``\textbf{(perturbation~+~compactness) argument}" as follows:
    
    By combining the variational characterization of $\mu^{\bigstar, (\uptau)}$ with a perturbative argument, we deduce the continuity of the mapping $\overline{Q'} \ni \uptau \mapsto \mu^{\bigstar, (\uptau)}$. Since $\overline{Q'}$ is compact, $|\varepsilon^{-2} \mu^{\bigstar,(\uptau)}| $ is bounded uniformly in $\uptau$. Next, we turn to the third term in the RHS of (\ref{eqn:m_components_unif_bound_ineq}). By \cite[Lemma 2]{friedlander2002}, we may write,
    \begin{align}\label{eqn:useful_decomp_tau_zero}
        \Lambda^{\text{soft}, (\uptau)} = \Lambda^{\text{soft},(0)} + B_{\text{soft}}, \quad \widetilde{\Lambda}^{\bigstar, (\uptau)} = \widetilde{\Lambda}^{\bigstar, (0)} + B_{\bigstar}, \quad \bigstar \in \{ \text{stiff-int},\text{stiff-ls} \},
    \end{align}
    where $B_\bigstar$, $B_\soft$ are uniformly (in $\uptau$) bounded operators. We then have
    \begin{align}\label{eqn:m_components_unif_bound_calculation}
        \| \Lambda^{\text{soft},(\uptau)} &(\psi^{\text{stiff-int},(\uptau)}, \psi^{\text{stiff-ls},(\uptau)} ) \| \leq \| \Lambda^{\text{soft},(0)} (\psi^{\text{stiff-int},(\uptau)}, \psi^{\text{stiff-ls},(\uptau)} ) \| + \| B_{\text{soft}} \|_{op} \sqrt{2} \nonumber\\
        &\qquad \leq  \| \Lambda^{\text{soft},(0)} (\psi^{\text{stiff-int},(\uptau)}, \psi^{\text{stiff-ls},(\uptau)} ) \| + C \nonumber\\
        &\qquad\leq \alpha_1 \| \widetilde{\Lambda}^{\text{stiff-int},(0)} \psi^{\text{stiff-int},(\uptau)} \| +
        \alpha_2  \| \widetilde{\Lambda}^{\text{stiff-ls},(0)} \psi^{\text{stiff-ls},(\uptau)} \| + \beta \nonumber\\
        &\qquad\leq \alpha_1 ( \| \widetilde{\Lambda}^{\text{stiff-int},(\uptau) \psi^{\text{stiff-int},(\uptau)} \|} + \| B_{\text{stiff-int}} \psi^{\text{stiff-int},(\uptau)} \| ) + \nonumber\\
        &\qquad\qquad\qquad\qquad \alpha_2 ( \| \widetilde{\Lambda}^{\text{stiff-ls},(\uptau) \psi^{\text{stiff-ls},(\uptau)} \|} + \| B_{\text{stiff-ls}} \psi^{\text{stiff-ls},(\uptau)} \| ) + \beta \nonumber\\
        &\qquad\leq C_1 |\mu^{\text{stiff-int},(\uptau)}|+ C_2 |\mu^{\text{stiff-ls},(\uptau)}| + C_3,
    \end{align}
    where the constants are all independent of $\tau$. The second and fourth inequality follows from (\ref{eqn:useful_decomp_tau_zero}). The third inequality follows by noting that the domains $\mathcal{D}(\Lambda^{\text{soft},(\uptau)})$, $\mathcal{D}( \widetilde{\Lambda}^{\bigstar,(\uptau)})$ are independent of $\uptau$, and then using the observation that $\Lambda^{\text{soft},(0)}$ is relatively $\Lambda^{\bigstar,(0)}$-bounded by \cite[Lemma 8.4]{konrad_book}. As noted above, $|\mu^{\bigstar,(\uptau)}|$ is bounded uniformly in $\uptau$.
    
    Finally, for \textbf{independence on $\varepsilon$}, we notice further that
    \begin{align}
        \mathbb{B} = \mathcal{P}^{(\uptau)}( \Lambda^{\text{soft},(\uptau)} + \Lambda_\varepsilon^{\text{stiff-int},(\uptau)} \mathcal{P}_\ls + \Lambda_\varepsilon^{\text{stiff-int},(\uptau)} \mathcal{P}_\interior) \mathcal{P}^{(\uptau)}_\perp = \mathcal{P}^{(\uptau)} \Lambda^{\text{soft},(\uptau)} \mathcal{P}^{(\uptau)}_\perp,
    \end{align}
    since $\mathcal{P}_\perp^{(\uptau)}\mathcal{E}$ is an invariant subspace for the stiff DtN maps. (We have a diagonal block matrix.) A similar statement holds for $\mathbb{E}$ (but not for $\mathbb{A}$).
\end{proof}

\begin{rmk}
    \begin{itemize}
        \item While \cite{friedlander2002} does not study the case of annuluar domains, the arguments of \cite[Lemma 2]{friedlander2002} still apply to give \eqref{eqn:useful_decomp_tau_zero}, since $Q_\soft$ is connected with smooth boundary $\Gamma_\interior \cup \Gamma_\ls$.
        \item A variant of the (perturbation + compactness) argument will be used again in Theorem \ref{thm:m_inverse_est} (for the term ``$\mathbb{S}$"). \qedhere
    \end{itemize}
\end{rmk}

\subsection{Inverting the M-operator}
Corollary \ref{cor:main_model_equiv} suggests that our study of norm-resolvent asymptotics of the main model operator $A_{\varepsilon}^{(\uptau)}$ requires us to estimate $(M_\varepsilon^{(\uptau)}(z))^{-1}$ in the operator norm. The goal of this section is to prove

\begin{thm}\label{thm:m_inverse_est}
    We have the following estimate in the operator norm
    \begin{align}
        \left(M_\varepsilon^{(\uptau)}(z)\right)^{-1} = 
        \begin{pmatrix}
            \mathbb{A}^{-1} & 0\\
            0 & 0
        \end{pmatrix}
        + O(\varepsilon^2),
    \end{align}
    relative to the decomposition $\mathcal{E} = \mathcal{P}^{(\uptau)}\mathcal{E} \oplus \mathcal{P}_{\perp}^{(\uptau)}\mathcal{E}$. $\| \mathbb{A}^{-1} \|_{op}$ is bounded uniformly in $\varepsilon>0$, $\uptau \in Q'$ and $z \in K_\sigma$. This estimate is uniform in $\uptau \in Q'$ and $z \in K_\sigma$.
\end{thm}

\begin{nts}
    \begin{itemize}
        \item (On the term $\mathbb{A}$) It suffices to show that $\text{Im}(M_\varepsilon^{(\uptau)}(z))$ is bounded below, because $|(Tv,v)| = |(\text{Re}Tv,v)| + |(\text{Im}Tv,v)|$.
        \item (On the term $\mathbb{A}$) Claim 1 does not simplify by throwing out the soft parts. Claim 2 does simplify, and we do throw it out in the second equation.
        \item (On the term $\mathbb{A}$) Yes you can simplify things by invoking continuity in $\tau$. But I refuse to use this until Proposition \ref{prop:cts_dependence} (cts dependence). I also refuse to use \cite{friedlander2002} (Friedlander), with the exception of Lemma \ref{lem:m_components_cts_ext} (as I don't see an alternative). This is because \cite{friedlander2002} is extremely sloppy, and the reader trying to dig through the details will need to close numerous gaps: 
        \begin{enumerate}
            \item \cite[Lemma 2]{friedlander2002} do not specify the domains its operators. If you read between the lines, you find that his DtN is a map from $H^{1/2}$ to $H^{-1/2}$, unlike us. So how do we know that the eigenfunctions and eigenvalues of the DtN in \cite{friedlander2002} is the same for us?
            \item After \cite[Lemma 2]{friedlander2002} ($\tau \mapsto \Lambda^{\stls,(\tau)}$ is differentiable) he remarks that $\tau \mapsto \Lambda^{\stls,(\tau)}$ is in fact smooth as $D_\tau \Lambda^{\stls,(\tau)}$ is smooth. The latter apparently ``obvious" from the \cite[formula (3.4)]{friedlander2002}.
            \item In the proof of \cite[Lemma 7]{friedlander2002}, it was claimed without justification that the smoothness of $\tau \mapsto \Lambda^{\stls,(\tau)}$ imply the the smoothness of $\tau \mapsto \psi_1^{\stls,(\tau)}$ and $\tau \mapsto \mu_1^{\stls,(\tau)}$. (I refuse to accept the reply ``by Kato" until I can work out the full details.) \qedhere
        \end{enumerate}

        \item (On the term $\mathbb{S}$) Reply to Step 2, on a simpler mthd to show that $\mu_2^{(\tau)}$ is bounded away from zero: ``Assume it's not, and arrive at a contradiction, because you know where the $\mu_1^{(\tau)}$ is" ... This ``simpler argument" needs more ingredients, as the band fcts $\mu_1^{(\tau)}$ and $\mu_2^{(\tau)}$ may overlap.

        \item (On the term $\mathbb{S}$) Reply to why am I not using the fact that $|\text{Im}z|>0$: I did use this. But the dependence on $z$ falls only on $\mathbb{D}_{\eps,b}^{(\tau)}(z)$, so it is limited to small portions of the argument. 

        \item (On the term $\mathbb{S}$) On a suggestion to prove a lower bound for $\mathbb{D}_\soft + \mathbb{D}_\st$, to show that the sum is SA: I don't see this, why does positivity implies self-adjointness?
        
        Also, the lower bounds for $\mathbb{D}_\st$ and $\mathbb{D}_\soft$ are discussed in Step 2 and Step 3 respectively.

        \item (On the term $\mathbb{S}$) Why does the soft component seem so complicated as compared to \cite{eff_behavior} (effective behavior): The argument in \cite{eff_behavior} was short, because it did not explain why the constants in \eqref{eqn:relative_bounded} can be chosen independently of $\tau$. Also, how does one deal with $\mathcal{P}^{(\tau)}$?

        \item (On the term $\mathbb{S}$) On Step 4, ``is this the best way to do it?": What does Kirill have in mind? I thought obtaining an expression for $\mathbb{D}^{-1}$ is already quite straightforward, since it helps to get the bound $\| \mathbb{D}^{-1} \| \leq C\eps^2$.

        \item On the remark after the proof, that I am contradicting myself by saying that ``We do not show that $\mathbb{D}_\soft$ is self-adjoint": There is no contradiction. I used $\mathbb{D}_\soft$ is symmetric and $\mathbb{D}_\st$ is self-adjoint to obtain $\mathbb{D}_\soft + \mathbb{D}_\st$ is self-adjoint.

        \item P.S. cty of $\tau \mapsto \Lambda^{(\tau)}u$ is used in Step 3.
    \end{itemize}
\end{nts}

\begin{proof}
    Since $M_\varepsilon^{(\uptau)}(z)$ is closed by Proposition \ref{prop:auxops_properties}(6), \cite[Theorem 2.3.3(i)]{tretter} implies that its inverse can be written in block operator form as:
    \begin{align}\label{eqn:block_op_inversion}
        \left(M_\varepsilon^{(\uptau)}(z)\right)^{-1} = \overline{
        \begin{pmatrix}
            \mathbb{A} & \mathbb{B}\\
            \mathbb{E} & \mathbb{D}
        \end{pmatrix}
        }^{-1} = 
        \begin{pmatrix}
            \mathbb{A}^{-1} + \overline{\mathbb{A}^{-1}\mathbb{B}}(\overline{\mathbb{S}})^{-1} \mathbb{E}\mathbb{A}^{-1}  & -\overline{\mathbb{A}^{-1}\mathbb{B}}(\overline{\mathbb{S}})^{-1}\\
            -(\overline{\mathbb{S}})^{-1} \mathbb{E}\mathbb{A}^{-1} & (\overline{\mathbb{S}})^{-1}
        \end{pmatrix},
    \end{align}
    where the Schur-Frobenius complement $\mathbb{S}$ is given by $\mathbb{S} := \mathbb{D} - \mathbb{E} \mathbb{A}^{-1} \mathbb{B}$. In writing down the above formula, it suffices to check that: (a) $\mathbb{A}$ is boundedly invertible, (b) $\mathbb{B}$ is bounded, (c) $\mathbb{D}$ is closed, and (d) $\mathbb{S}$ is closable, with $\overline{\mathbb{S}}$ being boundedly invertible.

    (b) is immediate, since $\mathbb{B}$ has finite range. In the remainder of the proof, we verify (a), (c), and (d), and provide bounds on $\mathbb{A}^{-1}$ and $(\overline{\mathbb{S}})^{-1}$, with dependence on $\varepsilon$, $\uptau$, and $z$ shown explicitly.
    %
    %
    %
    %

    \subsubsection*{The term \texorpdfstring{$\mathbb{A} = \mathcal{P}^{(\uptau)} M_\varepsilon^{(\uptau)}(z) \mathcal{P}^{(\uptau)}$}{A}}
    
    As $\mathbb{A}$ is an operator on the finite-dimensional space $\mathcal{P}^{(\uptau)}\mathcal{E}$, $\mathbb{A}$ will be boundedly invertible (uniformly in $\varepsilon$, $\uptau$, and $z$) if we can show that $\mathbb{A}$ is bounded below (uniformly in $\varepsilon$, $\uptau$, and $z$), as this implies injectivity. To do so, recall from Proposition \ref{prop:auxops_properties}(7) that $\text{Im}(M_\varepsilon^{(\uptau)}(z))= \text{Im}(M_\varepsilon^{(\uptau)}(z) - M_\varepsilon^{(\uptau)}(0))$.
    Now define the real part of $M_\varepsilon^{(\uptau)}(z)$ by
    \begin{align}
        \text{Re}(M_\varepsilon^{(\uptau)}(z)) &:= M_\varepsilon^{(\uptau)}(0) + \text{Re}\left( M_\varepsilon^{(\uptau)}(z) - M_\varepsilon^{(\uptau)}(0)\right)
        = \Lambda_\varepsilon^{(\uptau)} + \text{Re}\left(z (\Pi^{(\uptau)})^* (I-z (A_{\varepsilon,0}^{(\uptau)})^{-1})^{-1} \Pi^{(\uptau)} \right).
    \end{align}
    Then, by the symmetry of $\Lambda_\varepsilon^{(\uptau)}$,
    \begin{align}
        (M_\varepsilon^{(\uptau)}(z)v,v)_{\mathcal{E}} = (\text{Re}(M_\varepsilon^{(\uptau)}(z))v,v)_{\mathcal{E}} + i (\text{Im}(M_\varepsilon^{(\uptau)}(z))v,v)_{\mathcal{E}}, \quad v \in \mathcal{D}(M_\varepsilon^{(\uptau)}(z)) = \mathcal{D}(\Lambda_\varepsilon^{(\uptau)}).
    \end{align}
    
    Therefore, it suffices to show that $\text{Im}(M_\varepsilon^{(\uptau)}(z))$ is bounded below on $\mathcal{P}^{(\tau)} \mathcal{E}$. To show this, we recall from Proposition \ref{prop:auxops_properties}(7) that $\text{Im}( M_\varepsilon^{(\uptau)}(z) ) = (\text{Im } z) S_\varepsilon^{(\uptau)}(\overline{z})^* S_\varepsilon^{(\uptau)}(\overline{z})$. Since $z \in K_\sigma$, we may ignore the term $\text{Im }z$. Then, for $v\in \mathcal{E}$, Proposition \ref{prop:auxops_properties}(4) implies that
    \begin{align}\label{eqn:main_estimate}
        (\mathcal{P}^{(\uptau)} S_\eps^{(\uptau)}(\overline{z})^* S_\eps^{(\uptau)}(\overline{z}) \mathcal{P}^{(\uptau)} v,v)_{\mathcal{E}}
        = \| (I - \overline{z} (A_{\eps,0}^{(\uptau)})^{-1} )^{-1} \Pi^{(\uptau)} \mathcal{P}^{(\uptau)} v \|_{\mathcal{H}}^2.
    \end{align}

    \textbf{Claim 1: $(I - \overline{z} (A_{\varepsilon,0}^{(\uptau)})^{-1} )^{-1}$ is bounded below in the operator norm, uniformly in $\varepsilon$, $\uptau$, and $z$.} $A_{\varepsilon,0}^{(\uptau)}$ has compact resolvent (Proposition \ref{prop:decoupling_properties}), and hence admits an eigenfunction expansion
    \begin{align}
        A_{\varepsilon,0}^{(\uptau)} = \sum_{j=1}^{\infty} \left( \cdot, w_{\varepsilon,j}^{(\uptau)} \right)_{\mathcal{H}} \lambda_{\varepsilon,j}^{(\uptau)} w_{\varepsilon,j}^{(\uptau)},
    \end{align}
    where the eigenvalues $\lambda_{\varepsilon,j}^{(\uptau)}$ are real, due to the self-adjointness of $A_{\varepsilon,0}^{(\uptau)}$. We now split the operator in two, in the spectral picture: Since $K_\sigma$ is compact, there is some $R = R(K_\sigma) > 0$ such that $B(0,R)$ contains $K_\sigma$. We then choose ($\varepsilon$ and $\uptau$ dependent) spectral projections $P_1 = P_{\varepsilon,1}^{(\uptau)}$ and $P_2 = P_{\varepsilon,2}^{(\uptau)}$ on $\mathcal{H} = L^2(Q)$ such that $P_1 = I_\mathcal{H} - P_2$ and
    \begin{align}
        P_{\varepsilon,2}^{(\uptau)}\mathcal{H} = \text{span} \left\{ w_{\varepsilon,j}^{(\uptau)} :  j \text{ satisfies } |\lambda_{\varepsilon,j}^{(\uptau)}| > 3R (> R \geq |\overline{z}|) \right\}.
    \end{align}

    Next we observe that for $f\in\mathcal{H}$,
    \begin{alignat}{2}\label{eqn:a0_efct_exp}
        &\| (I - \overline{z} (A_{\varepsilon,0}^{(\uptau)})^{-1} )^{-1} f \|_{\mathcal{H}}^2 = \| (I - \overline{z} (A_{\varepsilon,0}^{(\uptau)})^{-1} )^{-1} (P_1 f + P_2 f) \|_{\mathcal{H}}^2 && \nonumber\\
        &= \| P_1 (I - \overline{z} (A_{\varepsilon,0}^{(\uptau)})^{-1} )^{-1} f + P_2 (I - \overline{z} (A_{\varepsilon,0}^{(\uptau)})^{-1} )^{-1} f) \|_{\mathcal{H}}^2 && \text{ $P_1$, $P_2$ are spectral projections.} \nonumber\\
        &= \| P_1 (I - \overline{z} (A_{\varepsilon,0}^{(\uptau)})^{-1} )^{-1} f \|_{\mathcal{H}}^2 + \|P_2 (I - \overline{z} (A_{\varepsilon,0}^{(\uptau)})^{-1} )^{-1} f) \|_{\mathcal{H}}^2 && \text{ Pythagoras theorem.} \nonumber\\
        &= \| (I - \overline{z} (A_{\varepsilon,0}^{(\uptau)})^{-1} )^{-1} P_1 f \|_{\mathcal{H}}^2 + \| (I - \overline{z} (A_{\varepsilon,0}^{(\uptau)})^{-1} )^{-1} P_2 f) \|_{\mathcal{H}}^2.
    \end{alignat}
    
    If we denote by $J = J_{\varepsilon}^{(\uptau)} \in \mathbb{N}$ the smallest integer that satisfies the condition of $P_2= P_{\varepsilon,2}^{(\uptau)}$, then we may write $P_1 f$ (and similarly for $P_2 f$) as
    \begin{align}
        P_{\varepsilon,1}^{(\uptau)} f 
        = \sum_{j=1}^{J_{\varepsilon}^{(\uptau)} - 1} \left( f, w_{\varepsilon,j}^{(\uptau)} \right)_{\mathcal{H}} w_{\varepsilon,j}^{(\uptau)} 
        = \sum_{j=1}^{J_{\varepsilon}^{(\uptau)} - 1} c_{\varepsilon,j}^{(\uptau)} w_{\varepsilon,j}^{(\uptau)}.
    \end{align}

    With this notation, the first term on the RHS of (\ref{eqn:a0_efct_exp}) can be estimated below as:
    \begin{alignat}{2}
        \| (I - \overline{z} &(A_{\varepsilon,0}^{(\uptau)})^{-1} )^{-1} P_{\varepsilon,1}^{(\uptau)} f \|_{\mathcal{H}}^2
        = \| A_{\varepsilon,0}^{(\uptau)} (A_{\varepsilon,0}^{(\uptau)} - \overline{z})^{-1} P_{\varepsilon,1}^{(\uptau)} f\|_{\mathcal{H}}^2 \nonumber\\
        &\geq c_1 \| (A_{\varepsilon,0}^{(\uptau)} - \overline{z})^{-1} P_{\varepsilon,1}^{(\uptau)} f\|_{\mathcal{H}}^2 &&\text{ By Proposition \ref{prop:decoupling_properties}.} \nonumber\\
        &= c_1 \left\| \sum_{j=1}^{J_{\varepsilon}^{(\uptau)} - 1} \frac{c_{\varepsilon,j}^{(\uptau)}} {\lambda_{\varepsilon,j}^{(\uptau)} - \overline{z} } w_{\varepsilon,j}^{(\uptau)} \right\|_{\mathcal{H}}^2 &&\text{ By functional calculus.} \nonumber\\
        &= c_1 \sum_{j=1}^{J_\varepsilon^{(\uptau)}-1} \frac{|c_{\varepsilon,j}^{(\uptau)}|^2}{|\lambda_{\varepsilon,j}^{(\uptau)} - \overline{z}|^2}  &&\text{ By Parseval's identity.} \nonumber\\
        &\geq c_1 \left( \min_{1\leq j \leq J_\varepsilon^{(\uptau)}-1} \left\{ \frac{1}{ |\lambda_{\varepsilon,j}^{(\uptau)} - \overline{z}|^2}\right\} \right) \sum_{j=1}^{J_{\varepsilon}^{(\uptau)} - 1} |c_{\varepsilon,j}^{(\uptau)}|^2 \nonumber\\
        &= c_1  c_2  \|P_{\varepsilon,1}^{(\uptau)} f \|_{\mathcal{H}}^2 &&\text{ By Parseval's identity,}
    \end{alignat}
    where $c_1>0$ and $c_2:= 1/(2R)^2$ are constants independent of $\varepsilon$, $\uptau$, and $z$. Observe that although $P_1$ depends on $\varepsilon$ and $\uptau$, the constant $c_2$ does not -- $c_2$ only depends on $K_\sigma$ through $B(0,R)$.
    
    For the term $(I - \overline{z} (A_{\varepsilon,0}^{(\uptau)})^{-1} )^{-1} P_{\varepsilon,2}^{(\uptau)} f$, we observe that since $P_{\varepsilon,2}^{(\uptau)}$ is a spectral projection for $A_{\varepsilon,0}^{(\uptau)}$, the second term equals $(I - \overline{z} (A_{\varepsilon,0}^{(\uptau)} P_{\varepsilon,2}^{(\uptau)})^{-1} )^{-1} f$. Next, recall that $P_2$ is chosen such that
    \begin{align}
        \left\| \overline{z} (A_{\varepsilon,0}^{(\uptau)} P_{\varepsilon,2}^{(\uptau)})^{-1} \right\|_{\mathcal{H}\rightarrow\mathcal{H}}
        &= |\overline{z}| \left\| (A_{\varepsilon,0}^{(\uptau)} P_{\varepsilon,2}^{(\uptau)})^{-1} \right\|_{\mathcal{H}\rightarrow\mathcal{H}}
        = |\overline{z}| \frac{1}{\text{dist}\left(0,\lambda_{\varepsilon, J_{\varepsilon}^{(\uptau)}}^{(\uptau)} \right)} \nonumber \\
        & = |\overline{z}| \frac{1}{\left|\lambda_{\varepsilon, J_{\varepsilon}^{(\uptau)}}^{(\uptau)} \right|} 
        < |\overline{z}| \frac{1}{3R} < \frac{1}{3}.
    \end{align}
    (Once again, $P_2$ depends on $\varepsilon$ and $\uptau$, while this estimate does not.) As a result, we may apply the Neumann series expansion:
    \begin{align}
        (I - \overline{z} (A_{\varepsilon,0}^{(\uptau)} P_{\varepsilon,2}^{(\uptau)})^{-1} )^{-1} = 
        I + \overline{z} (A_{\varepsilon,0}^{(\uptau)} P_{\varepsilon,2}^{(\uptau)})^{-1} + \cdots
    \end{align}
    
    The terms after $I$ have norm not exceeding $\sum_{n=1}^{\infty} (1/3)^n = 1/2$. Therefore the reverse triangle inequality implies that $(I - \overline{z} (A_{\varepsilon,0}^{(\uptau)})^{-1} )^{-1} P_{\varepsilon,2}^{(\uptau)}$ is bounded below (independently of $\varepsilon$, $\uptau$, and $z$). This proves Claim 1.

    Applying Claim 1 to (\ref{eqn:main_estimate}), we now have some $\tilde{c}>0$ independent of $\varepsilon$, $\uptau$, and $z$, such that
    \begin{align*}
        \| (I - \overline{z} (A_{\varepsilon,0}^{(\uptau)})^{-1} )^{-1} \Pi^{(\uptau)} \mathcal{P}^{(\uptau)} v \|_{\mathcal{H}}^2
        \geq \tilde{c} \| \Pi^{(\uptau)} \mathcal{P}^{(\uptau)} v \|_\mathcal{H}^2.
    \end{align*}
    
    \textbf{Claim 2: There is some $c>0$, independent of $\varepsilon$, $\uptau$, and $z$, such that 
    \begin{align}
        \| \Pi^{(\uptau)} \mathcal{P}^{(\uptau)} v \|_{\mathcal{H}}^2 \geq c \| \mathcal{P}^{(\uptau)} v \|_{\mathcal{E}}^2, \quad \text{where } v \in \mathcal{E}.
    \end{align}}
    
    Write $v = \phi + \varphi$, where $\phi \in L^2(\Gamma_\interior)$ and $\varphi \in L^2(\Gamma_\ls)$. Since $\mathcal{P}^{(\uptau)}$ is the projection onto $\text{span}\{ \psi^{\text{stiff-int},(\uptau)} \} \oplus \text{span}\{ \psi^{\text{stiff-ls},(\uptau)} \}$, we can write $\mathcal{P}^{(\uptau)} v = c_1 \psi^{\text{stiff-int},(\uptau)} + c_2 \psi^{\text{stiff-ls},(\uptau)}$ where $c_1,c_2 \in \mathbb{C}$. Now suppose that $v \neq 0$, then we must have either $c_1 \neq 0$ or $c_2 \neq 0$. Then
    \begin{align}
        \begin{cases}
            \left\| \Pi^{(\uptau)} (\mathcal{P}^{(\uptau)} v) \right\|_{\mathcal{H}}^2  \geq \left\| \Pi^{\text{stiff-int},(\uptau)} c_1 \psi^{\text{stiff-int},(\uptau)}  \right\|_{L^2(Q_{\text{stiff-int}})}^2 \quad &\text{if $c_1\neq 0$,} \\
            \left\| \Pi^{(\uptau)} (\mathcal{P}^{(\uptau)} v) \right\|_{\mathcal{H}}^2  \geq \left\| \Pi^{\text{stiff-ls},(\uptau)} c_2 \psi^{\text{stiff-ls},(\uptau)}  \right\|_{L^2(Q_{\text{stiff-ls}})}^2 &\text{if $c_2\neq 0$.}
        \end{cases}
    \end{align}
    The inequality follows by the Pythagoras theorem, as we recall that $\Pi^{(\uptau)}(\phi + \varphi) = \Pi^{\text{stiff-int}, (\uptau)}\phi + \Pi^{\text{soft}, (\uptau)}(\phi + \varphi) + \Pi^{\text{stiff-ls}, (\uptau)}\varphi$, and the lifts into the individual components $\Pi^{\text{stiff-int}, (\uptau)}\phi$, $\Pi^{\text{soft}, (\uptau)}(\phi + \varphi)$, and $\Pi^{\text{stiff-ls}, (\uptau)}\varphi$ are orthogonal. Therefore, by the linearity of $\Pi$ and the homogeneity of norms, it suffices to find some $c>0$ independent of $\varepsilon$, $\uptau$, and $z$ such that
    \begin{align*}
        \| \Pi^{\bigstar,(\uptau)} \psi^{\bigstar,(\uptau)} \|_{\mathcal{H}} \geq c \| \psi^{\bigstar,(\uptau)} \|_{\mathcal{E}} \stackrel{\psi \text{ is normalized}}{=} c, \quad \bigstar \in \{ \text{stiff-int}, \text{stiff-ls} \}.
    \end{align*}

    The proof of this inequality follows from two facts (the closure of $Q'$ is of importance here):
    \begin{enumerate}[label=(\roman*)]
        \item For each $\uptau \in \overline{Q'}$, $\| \Pi^{\bigstar,(\uptau)} \psi^{\bigstar,(\uptau)} \|$ is strictly positive, as or else this means the $\uptau-$harmonic lift of a non-zero function $\psi$ is zero, which contradicts the injectivity of $\Pi^{(\uptau)}$.
        
        \item The mapping $\overline{Q'} \ni \uptau \mapsto \| \Pi^{\bigstar,(\uptau)} \psi^{\bigstar,(\uptau)} \| \in \mathbb{R}_{\geq 0}$ is continuous.
    \end{enumerate}
    
    It remains to prove (ii) and hence complete Claim 2. The proof of the fact will be postponed to Proposition \ref{prop:cts_dependence}, and this concludes the discussion on $\mathbb{A}$.

    \subsubsection*{The term \texorpdfstring{$\mathbb{S} = \mathbb{D}-\mathbb{E} \mathbb{A}^{-1}\mathbb{B}$}{S}}
    
    We will proceed in four steps. \textbf{Step 1.} First, we introduce the notation
    \begin{align}\label{eqn:d_decompose}
        \mathbb{D}_\varepsilon^{(\uptau)}(z) = \mathbb{D}_{\text{soft}}^{(\uptau)} + \mathbb{D}_{\varepsilon, \text{stiff}}^{(\uptau)} + \mathbb{D}_{\varepsilon ,b}^{(\uptau)}(z)
    \end{align}
    where $\mathbb{D}_{\text{soft}} = \mathcal{P}_\perp^{(\uptau)} \Lambda^{\text{soft},(\uptau)} \mathcal{P}_\perp^{(\uptau)}$, $\mathbb{D}_{\text{stiff}} = \mathcal{P}_\perp^{(\uptau)} ( \Lambda_\varepsilon^{\text{stiff-int},(\uptau)} \oplus \Lambda_\varepsilon^{\text{stiff-ls},(\uptau)} ) \mathcal{P}_\perp^{(\uptau)}$, and $\mathbb{D}_b$ as what remains of $\mathbb{D}_\varepsilon^{(\uptau)}(z) = \mathcal{P}_\perp^{(\uptau)}M_\varepsilon^{(\uptau)}(z) \mathcal{P}_\perp^{(\uptau)}$. In this way, $\mathbb{D}_b$ is a bounded operator on $\mathcal{P}_\perp^{(\uptau)}\mathcal{E}$, with operator norm bounded uniformly in $\varepsilon$, $\uptau$, and $z$, by Proposition \ref{prop:auxops_properties}(5). 
    
    Furthermore, we claim that $\mathbb{D}_{\text{soft}} + \mathbb{D}_{\text{stiff}}$ is self-adjoint on $\mathcal{P}_\perp^{(\uptau)}\mathcal{E}$ with domain $\mathcal{D}(\mathbb{D}_\text{soft} + \mathbb{D}_\text{stiff}) = \mathcal{D}(\mathbb{D}_\text{soft}) = \mathcal{D}(\mathbb{D}_\text{stiff})$: The claim on the domain follows simply by construction. $\mathbb{D}_{\text{stiff}}$ is self-adjoint since $\mathcal{P}^{(\uptau)}$ is a spectral projection. $\mathbb{D}_\text{soft}$ is symmetric, and is relatively $\mathbb{D}_{\text{stiff}}$-bounded with relative bound strictly less than one, as pointed out in the proof of Lemma \ref{lem:dtn_selfadjoint}. Therefore the claim follows by the Kato-Rellich theorem \cite[Theorem~8.5]{konrad_book}.
    
    Being a sum of a closed $\mathbb{D}_\text{soft} + \mathbb{D}_\text{stiff}$ and a bounded $\mathbb{D}_b$ operator, it follows that $\mathbb{D}$ is closed, and so $\mathbb{S}$ is closed by the boundedness of $\mathbb{E}\mathbb{A}^{-1}\mathbb{B}$. We therefore drop the closures for $\mathbb{S}$ in (\ref{eqn:block_op_inversion}).

    \textbf{Step 2.} Next we discuss estimates for $\mathbb{D}$. As mentioned in Step 1, $\mathbb{D}_{\varepsilon ,b}^{(\uptau)}(z)$ is uniformly bounded in $\eps$, $\tau$ and $z$. As for $\mathbb{D}_{\varepsilon, \text{stiff}}^{(\uptau)}$ we claim that $\mathbb{D}_{\varepsilon, \text{stiff}}^{(\uptau)}$ is invertible with the following estimate
    \begin{align}\label{eqn:d_inverse_est}
        \left\| \left( \mathbb{D}_{\varepsilon, \text{stiff}}^{(\uptau)} \right)^{-1} \right\|_{\mathcal{P}_\perp^{(\uptau)} \mathcal{E} \rightarrow \mathcal{P}_\perp^{(\uptau)} \mathcal{E} } \leq C \varepsilon^2, \quad \text{$C>0$ is independent of $\varepsilon,\uptau$ and $z$.}
    \end{align}

    The independence on $z$ is immediate. Invertibility follows from Proposition \ref{prop:dtn_properties} and the fact that we have removed the lowest eigenspace using $\mathcal{P}^{(\uptau)}_\perp$. Since $\mathcal{P}^{(\uptau)}$ is the projection with respect to the unweighted DtN operator, we can separate out $\varepsilon$ and obtain the bound $C\varepsilon^2$, with $C$ independent of $\varepsilon$. It remains to justify the independence of $C$ on $\uptau$. For this, it suffices to show that the second eigenvalues $\mu_2^{\stin, (\uptau)}$ and $\mu_2^{\stls, (\uptau)}$ can be bounded away from zero, uniformly in $\uptau$. This is certainly true for each $\uptau$ (by Proposition \ref{prop:dtn_properties}), and can be extended to a neighbourhood $B(\uptau,\delta)$ of $\uptau$, as the mapping $\uptau \mapsto \mu_2^{(\uptau)}$ is continuous. 
    
    Now consider a dense set $\{ \uptau_n \} \subset \overline{Q'}$. With $B(\uptau_n,\delta_n)$ obtained as above, $\{ B(\uptau_n,\delta_n) \}_n$ is now an open cover of $\overline{Q'}$. By compactness of $\overline{Q'}$, we may extract a finite subcover $\{ B(\uptau_{n_k},\delta_{n_k}) \}_{k=1}^{K}$. Since $\mu_2^{(\uptau)}$ is bounded away from zero on each $B_k \equiv B(\uptau_{n_k},\delta_{n_k})$, we deduce that $\mu_2^{(\uptau)}$ is bounded \textit{above} by $\max_k \{ \mu_2^{(\uptau)} : \uptau \in B_k \}$, the latter being strictly negative (note our convention of the DtN map), and independent of $\uptau$. This concludes the justification of (\ref{eqn:d_inverse_est}).
    
    \textbf{Step 3.} Now consider the unweighted stiff DtN operator, denoted by $\widetilde{\mathbb{D}}^{(\uptau)}_\text{stiff} = \varepsilon^2 \mathbb{D}_{\varepsilon, \text{stiff}}^{(\uptau)}$. We claim that there exists constants $\alpha,\beta>0$, independent of $\uptau$ such that
    \begin{align}\label{eqn:relative_bounded}
        \| \mathbb{D}_{\text{soft}}^{(\uptau)} u \| \leq \alpha \| \widetilde{\mathbb{D}}^{(\uptau)}_\text{stiff} u \| + \beta \| u \|, \quad \text{ $\forall u\in \mathcal{D}(\mathbb{D}_{\text{soft}}^{(\uptau)}) = \mathcal{D}(\widetilde{\mathbb{D}}^{(\uptau)}_\text{stiff})$.}
    \end{align}
    That is, $\mathbb{D}_\text{soft}^{(\uptau)}$ is relatively $\widetilde{\mathbb{D}}^{(\uptau)}_\text{stiff}-$bounded, with uniform constants $\alpha$, $\beta$. To prove this claim, we first show this without the projections $\mathcal{P}_\perp^{(\uptau)}$, that is, for $\Lambda^{\text{soft},(\uptau)}$ and $\widetilde{\Lambda}^{\text{stiff-int},(\uptau)} \oplus \widetilde{\Lambda}^{\text{stiff-ls},(\uptau)}$. This is done by using \cite[Lemma 8.4]{konrad_book} to show relative boundedness for each $\uptau$, then applying perturbation theory to the soft and stiff DtN maps, then using the compactness of $\overline{Q'}$, similarly to what was done for (\ref{eqn:d_inverse_est}). We omit the details for brevity.
    
    We then proceed to add back the projections. Pre-composing with $\mathcal{P}^{(\uptau)}_\perp$ is trivial. Since $\mathcal{P}^{(\uptau)}$ is a spectral projection for the stiff DtN map, post-composing with $\mathcal{P}^{(\uptau)}_\perp$ is immediate, giving us the RHS of the inequality. As for $\Lambda^{\text{soft},(\uptau)}$, we write 
    \begin{align}
        \Lambda^{\text{soft},(\uptau)}\mathcal{P}^{(\uptau)}_\perp = \mathcal{P}^{(\uptau)} \Lambda^{\text{soft},(\uptau)}\mathcal{P}^{(\uptau)}_\perp + \mathcal{P}^{(\uptau)}_\perp \Lambda^{\text{soft},(\uptau)}\mathcal{P}^{(\uptau)}_\perp = \mathcal{P}^{(\uptau)} \Lambda^{\text{soft},(\uptau)}\mathcal{P}^{(\uptau)}_\perp + \mathbb{D}_\text{soft}^{(\uptau)}.
    \end{align}
    The first term is bounded uniformly in $\uptau$ thanks to Proposition \ref{prop:m_components_unif_bound}. Hence it can be absorbed into the RHS by picking a bigger $\beta$. This shows the claim for (\ref{eqn:relative_bounded}).

    \textbf{Step 4.} We omit the short argument combining (\ref{eqn:d_inverse_est}) and (\ref{eqn:relative_bounded}) to arrive at
    \begin{align}
        \|\mathbb{D}_\text{soft}^{(\uptau)} ( \mathbb{D}_{\varepsilon, \text{stiff}}^{(\uptau)} )^{-1} \|_{\mathcal{P}_\perp^{(\uptau)} \mathcal{E} \rightarrow \mathcal{P}_\perp^{(\uptau)} \mathcal{E} } \leq C \varepsilon^2, \quad \text{ where $C>0$ is independent of $\varepsilon$, $\uptau$ and $z$.}
    \end{align}
    (See \cite[Section 3.2]{eff_behavior} for details.) As a result, we have found the inverse for $\mathbb{D}$, namely
    \begin{align}
        \mathbb{D}^{-1} = \mathbb{D}_{\text{stiff}}^{-1} \left( I_{\mathcal{P}_\perp^{(\uptau)}\mathcal{E}} + \mathbb{D}_{\text{soft}} \mathbb{D}_{\text{stiff}}^{-1} + \mathbb{D}_b \mathbb{D}_{\text{stiff}}^{-1} \right)^{-1}.
    \end{align}

    Furthermore, thanks to our estimates on $\mathbb{D}_{\text{soft}} \mathbb{D}_{\text{stiff}}^{-1}$ and $\mathbb{D}_{\text{stiff}}^{-1}$ obtained above, we know that the terms after $I$ are of order $O(\varepsilon^2)$. Therefore the Neumann series expansion applies, giving the overall estimate of $\|\mathbb{D}^{-1} \| \leq C \varepsilon^2$. Meanwhile, Proposition \ref{prop:m_components_unif_bound} implies that $\| \mathbb{E} \mathbb{A}^{-1} \mathbb{B} \| \leq C$, where $C$ is an independent constant. Therefore, the formula $\mathbb{S}^{-1} = (I - \mathbb{D}^{-1} \mathbb{E} \mathbb{A}^{-1} \mathbb{B})^{-1} \mathbb{D}^{-1}$ implies that $\| \mathbb{S}^{-1} \| \leq C \varepsilon^2$. That is, $\mathbb{S}$ is boundedly invertible with the mentioned bound, where $C>0$ is independent of $\varepsilon$, $\uptau$, and $z$. This concludes the discussion on the term $\mathbb{S}$.
    
    We have shown that $\| \mathbb{A}^{-1} \| \leq C$ and $\| \mathbb{S}^{-1} \| \leq C \eps^2$. Together with $\| \mathbb{B} \| \leq C$, $\| \mathbb{E} \| \leq C$ (Proposition \ref{prop:m_components_unif_bound}), and (\ref{eqn:block_op_inversion}), this concludes the proof of the theorem.
\end{proof}


We conclude the proof of Theorem \ref{thm:m_inverse_est} with the following result:

\begin{prop}[Continuous dependence]\label{prop:cts_dependence}
    The mapping $f:\overline{Q'} \rightarrow \mathbb{R}_{\geq 0}$ given by
    \begin{align*}
        f(\uptau) = \| \Pi^{\bigstar, (\uptau)} \psi^{\bigstar, (\uptau)} \|_{L^2(Q)}
    \end{align*}
    is continuous, where $(\bullet,\bigstar) \in \{ \text{(int, stiff-int), (ls, stiff-ls)} \}$. 
\end{prop}

\begin{proof}
    See \cite[Proposition~5.5]{ys_thesis}. This boils down to proving the continuity of $\uptau \mapsto \Pi^{\bigstar,(\uptau)}$ and $\uptau \mapsto \psi_1^{\bigstar,(\uptau)}$. The bulk of the proof is devoted to the existence of a choice $\psi_1^{\stls,(\uptau)}$ that makes the mapping $\overline{Q'} \ni \uptau \mapsto \psi_1^{\stls,(\uptau)} \in L^2(\Gamma_\ls)$ continuous, and it is done through the method of asymptotic expansions. Note that this proves the remaining assertions of Proposition \ref{prop:dtn_properties}.
\end{proof}

\section{Identifying a suitable homogenized operator}\label{sect:identifying_homo_op}

The task now is to identify an operator $\widehat{A}_{\beta_0,\beta_1}$ that is $O(\varepsilon^2)$ close to $A_\varepsilon^{(\uptau)} = \widehat{A}_{0,I}$ in the norm-resolvent sense, by using Theorem \ref{thm:m_inverse_est}. To ensure that $\widehat{A}_{\beta_0,\beta_1}$ is well defined, we need to check that (i) $\beta_0$ and $\beta_1$ satisfies the domain considerations, and (ii) $\beta_0 + \beta_1 M_\varepsilon^{(\uptau)}(z)$ is boundedly invertible. Here we record a useful observation that is used for checking (ii):

\begin{lem}\label{lem:useful_b0b1_check}
    For $z\in K_\sigma$, we have
    \begin{align}\label{eqn:useful_b0b1_check}
        - \left( \overline{\mathcal{P}_{\perp}^{(\uptau)} + \mathcal{P}^{(\uptau)} M_\varepsilon^{(\uptau)}(z) } \right)^{-1} \mathcal{P}^{(\uptau)} = - \mathcal{P}^{(\uptau)} \left( \mathcal{P}^{(\uptau)} M_\varepsilon^{(\uptau)}(z) \mathcal{P}^{(\uptau)}  \right)^{-1}  \mathcal{P}^{(\uptau)}
    \end{align}
\end{lem}

\begin{proof}
    Note that $\mathcal{P}_\perp^{(\uptau)}$ and $\mathcal{P}^{(\uptau)}M_\varepsilon^{(\uptau)}(z)$ are bounded operators, hence the sum is closed.~That $\mathcal{P}_{\perp}^{(\uptau)} + \mathcal{P}^{(\uptau)} M_\varepsilon^{(\uptau)}(z)$ is boundedly invertible follows from the second equality of (\ref{eqn:block_op_inversion}), as
    \begin{align}
         -\left( \mathcal{P}_{\perp}^{(\uptau)} + \mathcal{P}^{(\uptau)} M_\varepsilon^{(\uptau)}(z) \right)^{-1} = -
        \begin{pmatrix}
            \mathcal{P}^{(\uptau)} M_\varepsilon^{(\uptau)} \mathcal{P}^{(\uptau)} & \mathcal{P}^{(\uptau)} M_\varepsilon^{(\uptau)} \mathcal{P}_\perp^{(\uptau)} \\
            0 & I
        \end{pmatrix}^{-1} = -
        \begin{pmatrix}
            \mathbb{A}^{-1} & -\mathbb{A}^{-1} \mathbb{B} \\
            0 & I
        \end{pmatrix},
    \end{align}
    which is bounded by Theorem \ref{thm:m_inverse_est}. Now applying $\mathcal{P}^{(\uptau)}$ on the right, we obtain
    \begin{align}
        -\left( \mathcal{P}_{\perp}^{(\uptau)} + \mathcal{P}^{(\uptau)} M_\varepsilon^{(\uptau)}(z) \right)^{-1} \mathcal{P}^{(\uptau)} = -
        \begin{pmatrix}
            \mathbb{A}^{-1} & 0 \\
            0 & 0
        \end{pmatrix}.
    \end{align}
    This is precisely the RHS of (\ref{eqn:useful_b0b1_check}), completing the proof.
\end{proof}

\begin{rmk}
    We have abused notation when writing $\mathcal{P}^{(\uptau)}$ in (\ref{eqn:useful_b0b1_check}). To be precise,
    \begin{align*}
        - \big( \underbrace{\mathcal{P}_{\perp}^{(\uptau)}}_{\mathcal{E} \rightarrow \mathcal{E}} + \underbrace{\mathcal{P}^{(\uptau)} M_\varepsilon^{(\uptau)}(z)}_{\mathcal{E} \rightarrow \mathcal{E}} \big)^{-1} \underbrace{\mathcal{P}^{(\uptau)}}_{\mathcal{E} \rightarrow \mathcal{E}}
        = - \underbrace{\mathcal{P}^{(\uptau)}}_{\mathcal{P}^{(\uptau)} \mathcal{E} \rightarrow \mathcal{E}}
        \big( \underbrace{\mathcal{P}^{(\uptau)} M_\varepsilon^{(\uptau)}(z) \mathcal{P}^{(\uptau)}}_{\mathcal{P}^{(\uptau)} \mathcal{E} \rightarrow \mathcal{P}^{(\uptau)} \mathcal{E}  } \big)^{-1}
        \underbrace{\mathcal{P}^{(\uptau)}}_{\mathcal{E} \rightarrow \mathcal{P}^{(\uptau)} \mathcal{E}}. \tag*{\qedhere}
    \end{align*}
\end{rmk}

Our first attempt on identifying a suitable homogenized operator is

\begin{thm}\label{thm:first_asymp_result}
    There exist $C>0$, independent of $\varepsilon>0$ (assumed to be small enough), $z\in K_\sigma$, and $\uptau \in Q'$, such that
    \begin{align}
        \left\| (A_\varepsilon^{(\uptau)} - z)^{-1} - \left( \widehat{A}^{(\uptau)}_{\varepsilon, \mathcal{P}_\perp^{(\uptau)}, \mathcal{P}^{(\uptau)} } - z \right)^{-1} \right\|_{\mathcal{H} \rightarrow \mathcal{H}} \leq C \varepsilon^2.
    \end{align}
    The operator $\widehat{A}^{(\uptau)}_{\varepsilon, \mathcal{P}_\perp^{(\uptau)}, \mathcal{P}^{(\uptau)} }$ is constructed relative to the triple $(A_{\varepsilon,0}^{(\uptau)}, \Lambda_\varepsilon^{(\uptau)}, \Pi^{(\uptau)})$ with $\mathcal{H} = L^2(Q)$ and boundary space $\mathcal{E} = L^2(\Gamma_\interior) \oplus L^2(\Gamma_\ls)$. Furthermore, $\widehat{A}^{(\uptau)}_{\varepsilon, \mathcal{P}_\perp^{(\uptau)}, \mathcal{P}^{(\uptau)} }$ is self-adjoint.
\end{thm}

\begin{proof}
    The inequality follows by Krein's formula (Theorem \ref{thm:kreins_formula}), the estimate on $M_\varepsilon^{(\uptau)}(z)$ (Theorem \ref{thm:m_inverse_est}), and the identity (\ref{eqn:useful_b0b1_check}). Self-adjointness of $\widehat{A}^{(\uptau)}_{\varepsilon, \mathcal{P}_\perp^{(\uptau)}, \mathcal{P}^{(\uptau)} }$ follows from \cite[Corollary 5.8]{ryzhov2009}.
\end{proof}

\begin{rmk}
    Theorem \ref{thm:first_asymp_result} extends the result of \cite[Theorem 3.1]{eff_behavior} to the stiff-soft-stiff setup, with $\mathcal{P}^{(\uptau)}$ now being two-dimensional, corresponding to the two stiff components.
\end{rmk}

While the operator $\widehat{A}^{(\uptau)}_{\varepsilon, \mathcal{P}_\perp^{(\uptau)}, \mathcal{P}^{(\uptau)} }$ satisfies the first criterion (self-adjointness) of a homogenized operator, it is unclear what the action of $\widehat{A}^{(\uptau)}_{\varepsilon, \mathcal{P}_\perp^{(\uptau)}, \mathcal{P}^{(\uptau)} }$ is, since it requires us to convert the boundary condition $\mathcal{P}_\perp^{(\uptau)} \Gamma_0^{(\uptau)} + \mathcal{P}^{(\uptau)} \Gamma_{\varepsilon,1}^{(\uptau)} = 0$ into the action of $\widehat{A}^{(\uptau)}_{\varepsilon, \mathcal{P}_\perp^{(\uptau)}, \mathcal{P}^{(\uptau)} }$. The goal now is to build on our result and identify other $O(\varepsilon^2)$-close operators whose actions can be more easily written down.

\subsection{An ``observer" in the soft component}

We begin this section by making the following definitions

\begin{defn}
    $M_\varepsilon^{\text{stiff},(\uptau)}(z) := M_\varepsilon^{\text{stiff-int},(\uptau)}(z) \mathcal{P}_\interior + M_\varepsilon^{\text{stiff-ls},(\uptau)}(z) \mathcal{P}_\ls$.
\end{defn}

\begin{defn}
    For $z\in \rho(A_\varepsilon^{(\uptau)})$, set $R_\varepsilon^{(\uptau)}(z) := P_{\text{soft}} (A_\varepsilon^{(\uptau)} - z)^{-1} P_{\text{soft}}$.
\end{defn}

We will refer to $R_\varepsilon^{(\uptau)}(z)$ as the \textit{generalized resolvent} of $A_\varepsilon^{(\uptau)}$ at $z$, with respect to $L^2(Q_{\text{soft}})$. The term ``generalized resolvent" refers to the fact that it is the resolvent of some operator on a larger space. This is not to be confused with pseudoresolvents in Theorem \ref{thm:kreins_formula}.

In this section, we demonstrate that we can draw conclusions on the full system on $L^2(Q)$, using the partial information on $L^2(Q_\soft)$ provided by $R_\varepsilon^{(\uptau)}(z)$, as the missing pieces can be attributed to ``error". Let us begin with an easy but important computation, which says that $R_\varepsilon^{(\uptau)}$ is itself a solution operator for some abstract BVP on $L^2(Q_{\text{soft}})$:

\begin{prop}\label{prop:r_eps_resolvent_alt}
    We have,
    \begin{align}
        R_\varepsilon^{(\uptau)}(z) = \left( \widehat{A}_{M_\varepsilon^{\text{stiff},(\uptau)}(z), I}^{\text{soft},(\uptau)} - z \right)^{-1},
    \end{align}
    where $\widehat{A}_{M_\varepsilon^{\text{stiff},(\uptau)}(z), I}^{\text{soft},(\uptau)}$ is constructed from the triple $(A_0^{\text{soft},(\uptau)}, \Lambda^{\text{soft},(\uptau)}, \Pi^{\text{soft}, (\uptau)} )$ with $L^2(Q_{\text{soft}})$ and boundary space $L^2(\Gamma_\interior) \oplus L^2(\Gamma_\ls)$. In other words $R_\varepsilon^{(\uptau)}(z)$ is the solution operator of the BVP:
    \begin{align}
        \begin{cases}
            \left( -(\nabla + i\uptau)^2 - z \right)u = f \qquad &\text{in $Q_\text{soft}$,}\\
            \partial_{n_{\text{soft}}}^{(\uptau)} u = - M_\varepsilon^{\stin,(\uptau)}(z) u &\text{on $\Gamma_\interior$,} \\
            \partial_{n_{\text{soft}}}^{(\uptau)} u = - M_\varepsilon^{\stls,(\uptau)}(z) u &\text{on $\Gamma_\ls$,}
        \end{cases}
    \end{align}
    which is to be rigorously interpreted in terms of the following system 
    \begin{align}
        \begin{cases}
            (\widehat{A}^{\text{soft},(\uptau)} - z )u = f, \\
            \Gamma_1^{\text{soft},(\uptau)}u = - M_\varepsilon^{\text{stiff},(\uptau)}(z) \Gamma_0^{\text{soft},(\uptau)}u.
        \end{cases}
    \end{align}
    Here, $f\in L^2(Q_\text{soft})$. (Note the $z$-dependent boundary conditions.)
\end{prop}

\begin{proof}
    We have
    \begin{align}\label{eqn:r_eps_resolvent}
        R_\varepsilon^{(\uptau)}(z) 
        &= P_\text{soft} (A_{\varepsilon,0}^{(\uptau)} - z)^{-1} P_\text{soft} 
        - P_\text{soft}  S_\varepsilon^{(\uptau)}(z) \left( M_\varepsilon^{(\uptau)}(z) \right)^{-1}  \left( S_\varepsilon^{(\uptau)}(\bar{z}) \right)^*   P_\text{soft}  \nonumber\\
        &= (A_0^{\text{soft},(\uptau)} - z)^{-1} - S^{\text{soft},(\uptau)}(z) \left( M_\varepsilon^{(\uptau)}(z) \right)^{-1} \left( S^{\text{soft},(\uptau)}(\bar{z}) \right)^* \nonumber\\
        &= (A_0^{\text{soft},(\uptau)} - z)^{-1} - S^{\text{soft},(\uptau)}(z) \left( 
        M_\varepsilon^{\text{stiff},(\uptau)}(z)  + M^{\text{soft},(\uptau)}(z) \right)^{-1} \left( S^{\text{soft},(\uptau)}(\bar{z}) \right)^*.
    \end{align}
    The first equality follows by Corollary \ref{cor:main_model_equiv}. For the second equality, $P_\text{soft} (A_{\varepsilon,0}^{(\uptau)} - z)^{-1} P_\text{soft} = (A_0^{\text{soft},(\uptau)} - z)^{-1}$ follows directly by construction, and $P_\text{soft} S_{\varepsilon}^{(\uptau)}(z) = S^{\text{soft}, (\uptau)}(z)$ by (\ref{eqn:s_op_sum}). The final equality follows by Proposition \ref{prop:m_op_sum}. The assertion on the solution operator then follows by Theorem \ref{thm:kreins_formula}.
\end{proof}

\begin{rmk}
    By passing from $L^2(Q)$ to $L^2(Q_\text{soft})$, the BCs has changed from $(\beta_0$, $\beta_1) = (0,I)$ to $(\beta_0$, $\beta_1) = (M^{\text{stiff}}(z), I)$. We describe this informally by saying that an observer living in $L^2(Q_\text{soft})$ is able to feel the effect of the ``stiff" part of the system through the ($z$-dependent) BCs.
\end{rmk}

Recall the computations in Lemma \ref{lem:useful_b0b1_check}: when converting $-(\mathcal{P}_\perp^{(\uptau)} + \mathcal{P}^{(\uptau)} M_\varepsilon^{(\uptau)}(z) )^{-1} \mathcal{P}^{(\uptau)}$ to \\ $-\mathcal{P}^{(\uptau)} ( \mathcal{P}^{(\uptau)} M_\varepsilon^{(\uptau)}(z) \mathcal{P}^{(\uptau)} )^{-1} \mathcal{P}^{(\uptau)}$, we are only interested in the left column of the block matrix. This suggests that we could modify the top right entry to our desire. (The bottom right entry should be kept as $I$ to ensure invertibility of the matrix.) In particular,
\begin{align}
    &\begin{pmatrix}
        \mathcal{P}^{(\uptau)} M_\varepsilon^{\text{stiff}, (\uptau)}(z) \mathcal{P}^{(\uptau)} + \mathcal{P}^{(\uptau)} M^{\text{soft}, (\uptau)}(z) \mathcal{P}^{(\uptau)} & \mathcal{P}^{(\uptau)} M_\varepsilon^{(\uptau)}(z) \mathcal{P}_\perp^{(\uptau)} \\
        0 & I
    \end{pmatrix}^{-1} \mathcal{P}^{(\uptau)} \nonumber \\
    = 
    &\begin{pmatrix}
        \mathcal{P}^{(\uptau)} M_\varepsilon^{\text{stiff}, (\uptau)}(z) \mathcal{P}^{(\uptau)} + \mathcal{P}^{(\uptau)} M^{\text{soft}, (\uptau)}(z) \mathcal{P}^{(\uptau)} & \mathcal{P}^{(\uptau)} M^{\text{soft}, (\uptau)}(z) \mathcal{P}_\perp^{(\uptau)} \\
        0 & I
    \end{pmatrix}^{-1} \mathcal{P}^{(\uptau)}
    =
    \begin{pmatrix}
        \mathbb{A}^{-1} & 0 \\
        0 & 0
    \end{pmatrix}
\end{align}

This suggests us to make the following definition:

\begin{defn}
    $R_{\varepsilon,\text{eff}}^{(\uptau)}(z) := \left( \widehat{A}^{\text{soft},(\uptau)}_{\mathcal{P}_\perp^{(\uptau)} + \mathcal{P}^{(\uptau)} M_\varepsilon^{\text{stiff},(\uptau)}(z) \mathcal{P}^{(\uptau)}, \mathcal{P}^{(\uptau)} }  - z\right)^{-1}$.
\end{defn}

To check that the choice $\beta_0$ and $\beta_1$ for $R_{\varepsilon,\text{eff}}^{(\uptau)}(z)$ is valid, we note that (i) $\beta_0 + \beta_1 M^{\text{soft}(\uptau)}(z)$ is a sum of bounded operators with maximal domain hence it is closed, (ii) the invertibility of the matrix follows because it is upper triangular, (iii) boundedness of the inverse follows from the estimate for $\mathbb{A}^{-1}$ in Theorem \ref{thm:m_inverse_est}. The observation on the equality of matrices is used in the following result:

\begin{prop}\label{prop:r_eff_estimate_soft}
    We have the following operator norm estimate, uniform in $z \in K_\sigma$ and $\uptau \in Q'$:
    \begin{align*}
        R_\varepsilon^{(\uptau)}(z) - R_{\varepsilon,\text{eff}}^{(\uptau)}(z) = O(\varepsilon^2).
    \end{align*}
\end{prop}

\begin{proof}
    See Appendix \ref{appendix:krein_calc}.
\end{proof}

We now turn our attention to discuss dilations of $R_{\varepsilon,\text{eff}}^{(\uptau)}(z)$. We would like to guess an operator $\mathcal{R}_{\varepsilon,\text{eff}}^{(\uptau)}(z)$ (note the use of calligraphic font) on the full space $L^2(Q)$ that is $O(\varepsilon^2)$ close to $(A_\varepsilon^{(\uptau)}-z)^{-1}$. The hope is that $\mathcal{R}_{\varepsilon,\text{eff}}^{(\uptau)}(z)$ is the resolvent of a self-adjoint operator whose action depends on $\varepsilon$ in a clear way. One necessary condition is $\mathcal{R}_{\varepsilon, \text{eff}}^{(\uptau)}(z)^* = \mathcal{R}_{\varepsilon, \text{eff}}^{(\uptau)}(\bar{z})$. The guess is as follows:

\begin{defn}
    Let $\mathcal{R}_{\varepsilon,\text{eff}}^{(\uptau)}(z)$ be the operator on $L^2(Q)$ defined by the following formula with respect to the decomposition $\mathcal{H} = L^2(Q_\text{soft}) \oplus L^2(Q_\text{stiff-int}) \oplus L^2(Q_\text{stiff-ls})$:
    \begin{align}\label{eqn:r_eff_fullspace}
       \mathcal{R}_{\varepsilon,\text{eff}}^{(\uptau)}(z) = 
       \begin{pmatrix}
        R_{\varepsilon,\text{eff}}^{(\uptau)}(z) & a_{12} & a_{13} \\
        a_{21} & a_{22} & a_{23} \\
        a_{31} & a_{32} & a_{33}
    \end{pmatrix}
    \end{align}
    where
    \begin{align*}
        a_{21} &= \Pi^{\text{stiff-int}, (\uptau)} k^{(\uptau)}(z)
        \left[ R_{\varepsilon,\text{eff}}^{(\uptau)}(z) - (A_0^{\text{soft},(\uptau)} - z)^{-1} \right] \\
        a_{31} &= \Pi^{\text{stiff-ls}, (\uptau)} k^{(\uptau)}(z)
        \left[ R_{\varepsilon,\text{eff}}^{(\uptau)}(z) - (A_0^{\text{soft},(\uptau)} - z)^{-1} \right] \\
        &\qquad a_{12} = \left( k^{(\uptau)}(\bar{z}) \left[ R_{\varepsilon,\text{eff}}^{(\uptau)}(\bar{z}) - (A_0^{\text{soft},(\uptau)} - \bar{z})^{-1} \right]   \right)^* \left( \Pi^{\text{stiff-int},(\uptau)} \right)^* \\
        &\qquad a_{22} = \Pi^{\text{stiff-int}, (\uptau)} k^{(\uptau)}(z)
        \left( k^{(\uptau)}(\bar{z}) \left[ R_{\varepsilon,\text{eff}}^{(\uptau)}(\bar{z}) - (A_0^{\text{soft},(\uptau)} - \bar{z})^{-1} \right]   \right)^* \left( \Pi^{\text{stiff-int},(\uptau)} \right)^*  \\
        &\qquad a_{32} = \Pi^{\text{stiff-ls}, (\uptau)} k^{(\uptau)}(z)
        \left( k^{(\uptau)}(\bar{z}) \left[ R_{\varepsilon,\text{eff}}^{(\uptau)}(\bar{z}) - (A_0^{\text{soft},(\uptau)} - \bar{z})^{-1} \right]   \right)^* \left( \Pi^{\text{stiff-int},(\uptau)} \right)^*  \\
        &\qquad\qquad\qquad a_{13} = \left( k^{(\uptau)}(\bar{z}) \left[ R_{\varepsilon,\text{eff}}^{(\uptau)}(\bar{z}) - (A_0^{\text{soft},(\uptau)} - \bar{z})^{-1} \right]   \right)^* \left( \Pi^{\text{stiff-ls},(\uptau)} \right)^* \\
        &\qquad\qquad\qquad a_{23} = \Pi^{\text{stiff-int}, (\uptau)} k^{(\uptau)}(z)
        \left( k^{(\uptau)}(\bar{z}) \left[ R_{\varepsilon,\text{eff}}^{(\uptau)}(\bar{z}) - (A_0^{\text{soft},(\uptau)} - \bar{z})^{-1} \right]   \right)^* \left( \Pi^{\text{stiff-ls},(\uptau)} \right)^*  \\
        &\qquad\qquad\qquad a_{33} = \Pi^{\text{stiff-ls}, (\uptau)} k^{(\uptau)}(z)
        \left( k^{(\uptau)}(\bar{z}) \left[ R_{\varepsilon,\text{eff}}^{(\uptau)}(\bar{z}) - (A_0^{\text{soft},(\uptau)} - \bar{z})^{-1} \right]   \right)^* \left( \Pi^{\text{stiff-ls},(\uptau)} \right)^*
    \end{align*}
    where $k^{(\uptau)}(z) := \Gamma_0^{\text{soft},(\uptau)}|_{\mathcal{D}(A_0^{\text{soft},(\uptau)}) \dot{+} \text{ran}(\Pi^{\soft,(\tau)}\mathcal{P}^{(\tau)})}$.
\end{defn}

\begin{prop}\label{prop:r_eff_fullspace}
    We have the following operator norm estimate, uniform in $z\in K_\sigma$ and $\uptau \in Q'$:
    $$
        (A_\varepsilon^{(\uptau)} - z )^{-1} - \mathcal{R}_{\varepsilon,\text{eff}}^{(\uptau)}(z) = O(\varepsilon^2).
    $$
\end{prop}

\begin{proof}
    This is verified entry-wise, and only requires minimal modifications to the proof in \cite[Theorem 3.9]{eff_behavior}. For completeness we include its proof in Appendix \ref{appendix:krein_calc}.
\end{proof}

Recall that for $R_{\varepsilon,\text{eff}}^{(\uptau)}(z)$ we have $\beta_0 = \mathcal{P}_\perp^{(\uptau)} + \mathcal{P}^{(\uptau)} M_\varepsilon^{\text{stiff},(\uptau)}(z) \mathcal{P}^{(\uptau)}$. We can further simplify the term $\mathcal{P}^{(\uptau)} M_\varepsilon^{\text{stiff},(\uptau)}(z) \mathcal{P}^{(\uptau)}$ by using Lemma \ref{lem:m_op_est} and (\ref{eqn:pi_op_sum}):
\begin{align}
    \mathcal{P}^{(\uptau)} &M_\varepsilon^{\text{stiff},(\uptau)}(z) \mathcal{P}^{(\uptau)} = \mathcal{P}^{(\uptau)} M_\varepsilon^{\text{stiff-int},(\uptau)}(z) \mathcal{P}_\interior \mathcal{P}^{(\uptau)} + \mathcal{P}^{(\uptau)} M_\varepsilon^{\text{stiff-ls},(\uptau)}(z) \mathcal{P}_\ls \mathcal{P}^{(\uptau)} \nonumber\\
    &= \mathcal{P}_{\text{stiff-int}}^{(\uptau)} M_\varepsilon^{\text{stiff-int},(\uptau)}(z) \mathcal{P}_{\text{stiff-int}}^{(\uptau)} \oplus \mathcal{P}_{\text{stiff-ls}}^{(\uptau)} M_\varepsilon^{\text{stiff-ls},(\uptau)}(z) \mathcal{P}_{\text{stiff-ls}}^{(\uptau)} \nonumber\\
    &= \left( \mathcal{P}_{\text{stiff-int}}^{(\uptau)} \Lambda_\varepsilon^{\text{stiff-int},(\uptau)} \mathcal{P}_{\text{stiff-int}}^{(\uptau)} + z \mathcal{P}^{(\uptau)}_{\text{stiff-int}} (\Pi^{\text{stiff-int},(\uptau)})^* \Pi^{\text{stiff-int},(\uptau)} \mathcal{P}^{(\uptau)}_{\text{stiff-int}} \right) \nonumber\\
    &\qquad \oplus \left( \mathcal{P}_{\text{stiff-ls}}^{(\uptau)} \Lambda_\varepsilon^{\text{stiff-ls},(\uptau)} \mathcal{P}_{\text{stiff-ls}}^{(\uptau)} + z \mathcal{P}^{(\uptau)}_{\text{stiff-ls}} (\Pi^{\text{stiff-ls},(\uptau)})^* \Pi^{\text{stiff-ls},(\uptau)} \mathcal{P}^{(\uptau)}_{\text{stiff-ls}} \right) + O(\varepsilon^2) \nonumber\\
    &= \mathcal{P}^{(\uptau)} \left( \Lambda_\varepsilon^{\text{stiff-int},(\uptau)} \oplus \Lambda_\varepsilon^{\text{stiff-ls},(\uptau)} \right) \mathcal{P}^{(\uptau)} \nonumber\\
    &\qquad + z\mathcal{P}^{(\uptau)} \left( (\Pi^{\text{stiff-int},(\uptau)})^* \Pi^{\text{stiff-int},(\uptau)} \oplus (\Pi^{\text{stiff-ls},(\uptau)})^* \Pi^{\text{stiff-ls},(\uptau)} \right) \mathcal{P}^{(\uptau)} + O(\varepsilon^2) \nonumber\\
    &= \mathcal{P}^{(\uptau)} \left( \Lambda_\varepsilon^{\text{stiff-int},(\uptau)} \oplus \Lambda_\varepsilon^{\text{stiff-ls},(\uptau)} \right) \mathcal{P}^{(\uptau)} \nonumber\\
    &\qquad + z \mathcal{P}^{(\uptau)} \left( (\Pi^{\text{stiff-int},(\uptau)} \oplus \Pi^{\text{stiff-ls},(\uptau)})^* (\Pi^{\text{stiff-int},(\uptau)} \oplus \Pi^{\text{stiff-ls},(\uptau)}) \right) \mathcal{P}^{(\uptau)} + O(\varepsilon^2).
\end{align}

This is helpful as it separates the term that depends on $\varepsilon^{-2}$ (the stiff DtN maps), from the terms that are uniformly bounded (the stiff harmonic lifts). We therefore define:

\begin{defn}\label{defn:r_hom_soft}
    We define $R_{\varepsilon,\text{hom}}^{(\uptau)}(z)$ as the following operator on $L^2(Q_\text{soft})$:
    $$\left( \widehat{A}^{\text{soft},(\uptau)}_{\mathcal{P}_\perp^{(\uptau)} + \mathcal{P}^{(\uptau)} \left[ \left(\Lambda_\varepsilon^{\text{stiff-int},(\uptau)} \oplus \Lambda_\varepsilon^{\text{stiff-ls},(\uptau)} \right) + z  (\Pi^{\text{stiff-int},(\uptau)} \oplus \Pi^{\text{stiff-ls},(\uptau)})^* (\Pi^{\text{stiff-int},(\uptau)} \oplus \Pi^{\text{stiff-ls},(\uptau)}) \right] \mathcal{P}^{(\uptau)} ,\mathcal{P}^{(\uptau)}  } - z \right)^{-1},$$
    and set $\mathcal{R}_{\varepsilon,\text{hom}}^{(\uptau)}(z)$ as the operator on $L^2(Q)$ defined by (\ref{eqn:r_eff_fullspace}), but with all the terms involving ``$R_{\varepsilon,\text{eff}}^{(\uptau)}$" to be replaced by $R_{\varepsilon,\text{hom}}^{(\uptau)}$.
\end{defn}

For the validity of the choice $(\beta_0,\beta_1)$, we use: the validity of $(\beta_0,\beta_1)$ for $R_{\varepsilon,\text{eff}}^{(\uptau)}(z)$, and the observation that $\mathbb{A}$ is bounded below uniformly in $\varepsilon$, $\uptau$, and $z$. (Details are provided in the proof of Theorem \ref{thm:self_adjoint} later (``top left entry").)

We conclude this section with the following result:

\begin{thm}
    We have the following operator norm estimate, uniform in $z\in K_\sigma$ and $\uptau \in Q'$:
    \begin{align*}
        (A_\varepsilon^{(\uptau)} - z)^{-1} - \mathcal{R}_{\varepsilon,\text{hom}}^{(\uptau)}(z) = O(\varepsilon^2).
    \end{align*}
\end{thm}

\begin{proof}
    This follows from $R_{\varepsilon,\text{eff}}^{(\uptau)}(z) - R_{\varepsilon,\text{hom}}^{(\uptau)}(z) = O(\varepsilon^2)$, which can be checked, for instance by the resolvent identity applied to $-(\overline{\beta_0 + \beta_1 M(z)})^{-1}$ (this is boundedly invertible by construction).
\end{proof}

\subsection{Self-adjointness of \texorpdfstring{$\mathcal{A}_{\varepsilon,\text{hom}}^{(\uptau)}$}{A eps hom tau} (Preliminaries)}

In the previous section, we have identified a candidate operator $\mathcal{R}_{\varepsilon,\text{hom}}^{(\uptau)}(z)$ on $L^2(Q)$ which could serve as the resolvent of some self-adjoint operator which will be denoted later by $\mathcal{A}_{\varepsilon,\text{hom}}^{(\uptau)}$. However, the term $k^{(\uptau)}(z)$ has finite-range, since $\mathcal{P}^{(\uptau)}\mathcal{E}$ is finite dimensional. This implies that self-adjointness on $L^2(Q)$ is impossible, as we will be left with non-zero defect indices. We may still pursue the question of self-adjointness, but on some subspace of $L^2(Q)$. This motivates us to define:

\begin{defn}\label{defn:truncated_dtn_lift}
    Write $\breve{\mathcal{E}}^{(\uptau)} := \mathcal{P}^{(\uptau)}\mathcal{E}$ for the truncated boundary space. Now introduce the following truncated operators:
    \begin{alignat*}{2}
        &\breve{\Pi}^{\text{soft}, (\uptau)} := \Pi^{\text{soft}, (\uptau)} |_{\breve{\mathcal{E}}^{(\uptau)}},
        \qquad &&\breve{\Lambda}^{\text{soft},(\uptau)} := \mathcal{P}^{(\uptau)} \Lambda^{\text{soft},(\uptau)} |_{\breve{\mathcal{E}}^{(\uptau)}}, \\
        &\breve{\Pi}^{\text{stiff-int}, (\uptau)} := \Pi^{\text{stiff-int}, (\uptau)} |_{\breve{\mathcal{E}}^{(\uptau)}}, 
        \qquad &&\breve{\Lambda}^{\text{stiff-int},(\uptau)}_\varepsilon := \mathcal{P}^{(\uptau)} \Lambda^{\text{stiff-int},(\uptau)}_\varepsilon |_{\breve{\mathcal{E}}^{(\uptau)}}, \\
        &\breve{\Pi}^{\text{stiff-ls}, (\uptau)} := \Pi^{\text{stiff-ls}, (\uptau)} |_{\breve{\mathcal{E}}^{(\uptau)}}, 
        \qquad &&\breve{\Lambda}^{\text{stiff-ls},(\uptau)}_\varepsilon := \mathcal{P}^{(\uptau)} \Lambda^{\text{stiff-ls},(\uptau)}_\varepsilon |_{\breve{\mathcal{E}}^{(\uptau)}}.
    \end{alignat*}
    Set $\breve{\Pi}^{\text{stiff}, (\uptau)} := \breve{\Pi}^{\text{stiff-int},(\uptau)} \oplus \breve{\Pi}^{\text{stiff-ls},(\uptau)}$ and $\breve{\Lambda}_{\varepsilon}^{\text{stiff}, (\uptau)} := \breve{\Lambda}_{\varepsilon}^{\text{stiff-int}, (\uptau)} \oplus \breve{\Lambda}_{\varepsilon}^{\text{stiff-ls}, (\uptau)}$. By the truncated DtN maps $\breve{\Lambda}$, we mean its continuous extension to the full subspace $\breve{\mathcal{E}}$. (Recall Lemma \ref{lem:m_components_cts_ext} and the comment thereafter.)
\end{defn}

\begin{rmk}
    The goal of this section is to prove self-adjointness for each $\varepsilon$ and $\uptau$. We will hence drop the dependence on $\varepsilon$ and $\uptau$ where convenient.
\end{rmk}

As $\mathcal{P}^{(\uptau)}$ is a spectral projection with respect to the stiff DtN maps, we immediately see that $\breve{\Lambda}^{\text{stiff-int}}$ and $\breve{\Lambda}^{\text{stiff-ls}}$ are self-adjoint. In fact, $\breve{\Lambda}^{\text{soft}}$ is self-adjoint too, as it is symmetric on the finite dimensional space $\breve{\mathcal{E}}$.

The lifts $\breve{\Pi}^{\text{soft}}$, $\breve{\Pi}^{\text{stiff-int}}$, and $\breve{\Pi}^{\text{stiff-ls}}$ are injective and bounded since they are restrictions of operators that are so. We can turn it into a surjective map by restricting its codomain to:

\begin{defn} Introduce the following subspaces of $\mathcal{H} = L^2(Q)$:
    \begin{align*}
        \breve{\mathcal{H}}^{\text{soft},(\uptau)} := \text{ran}(\breve{\Pi}^{\text{soft},(\uptau)}), \quad
        \breve{\mathcal{H}}^{\text{stiff-int},(\uptau)} := \text{ran}(\breve{\Pi}^{\text{stiff-int},(\uptau)}), \quad
        \breve{\mathcal{H}}^{\text{stiff-ls},(\uptau)} := \text{ran}(\breve{\Pi}^{\text{stiff-ls},(\uptau)}),
    \end{align*}
    and set $\breve{\mathcal{H}}^{\text{stiff},(\uptau)} = \breve{\mathcal{H}}^{\text{stiff-int},(\uptau)} \oplus \breve{\mathcal{H}}^{\text{stiff-ls},(\uptau)}$. (The orthogonality is a consequence of our setup.)
\end{defn}

We may now define the following triple together with its auxiliary operators:

\begin{defn}\label{defn:trunc_triple}
    Consider the $(A_0^{\text{soft},(\uptau)}, \breve{\Lambda}^{\text{soft},(\uptau)}, \breve{\Pi}^{\text{soft},(\uptau)})$ on $L^2(Q_\text{soft})$ and boundary space $\breve{\mathcal{E}}^{(\uptau)}$. Construct the following operators in accordance with Definition \ref{defn:aux_operators}:
    \begin{align*}
        \Breve{\widehat{A}^{\text{soft},(\uptau)}} &: \mathcal{D}(A_0^{\text{soft},(\uptau)}) \dot{+} \breve{\mathcal{H}}^{\text{soft},(\uptau)} \rightarrow L^2(Q_\text{soft}), \\
        \breve{\Gamma}_0^{\text{soft},(\uptau)} &: \mathcal{D}(A_0^{\text{soft},(\uptau)}) \dot{+} \breve{\mathcal{H}}^{\text{soft},(\uptau)} \rightarrow \breve{\mathcal{E}}^{(\uptau)}, \\
        \breve{\Gamma}_1^{\text{soft},(\uptau)} &: \mathcal{D}(A_0^{\text{soft},(\uptau)}) \dot{+} \breve{\Pi}^{\text{soft},(\uptau)} \mathcal{D}(\breve{\Lambda}^{\text{soft},(\uptau)}) \rightarrow \breve{\mathcal{E}}^{(\uptau)}, \\
        \breve{S}^{\text{soft},(\uptau)}(z) &: \breve{\mathcal{E}}^{(\uptau)} \rightarrow L^2(Q_\text{soft}), \\
        \breve{M}^{\text{soft},(\uptau)}(z) &: \mathcal{D}(\breve{\Lambda}^{\text{soft},(\uptau)}) \rightarrow \breve{\mathcal{E}}^{(\uptau)}.
    \end{align*}
\end{defn}

\begin{rmk}
    As a consequence of Definition \ref{defn:truncated_dtn_lift}, $\mathcal{D}( \breve{\Lambda}^{\text{soft},(\uptau)}) = \breve{\mathcal{E}}^{(\uptau)}$.
\end{rmk}

We record some properties of the truncated triple in relation to its original counterpart.

\begin{prop}\label{prop:trunc_aux_properties}
    ~
    \begin{enumerate}
        \item $\breve{\Pi}^{\text{soft},(\uptau)} : \breve{\mathcal{E}}^{(\uptau)} \rightarrow \breve{\mathcal{H}}^{\text{soft},(\uptau)}$ and $(\breve{\Pi}^{\text{stiff-int},(\uptau)} \oplus \breve{\Pi}^{\text{stiff-ls},(\uptau)}): \breve{\mathcal{E}}^{(\uptau)} \rightarrow \breve{\mathcal{H}}^{\text{stiff},(\uptau)}$ are both bounded and boundedly invertible.
        
        \item $\Breve{\widehat{A}^{\text{soft},(\uptau)}}$ is densely defined and closed.
        
        \item $\breve{S}^{\text{soft},(\uptau)}(z) = S^{\text{soft},(\uptau)}(z)|_{\breve{\mathcal{E}}}$.
        
        \item $\breve{M}^{\text{soft},(\uptau)}(z) = \mathcal{P}^{(\uptau)} M^{\text{soft},(\uptau)}(z)|_{\breve{\mathcal{E}}}$, that is, $\breve{M}^{\text{soft},(\uptau)}(z)$ is the compression of its original operator.
        
        \item $\breve{\Gamma}_0^{\text{soft},(\uptau)}$ and $\breve{\Gamma}_1^{\text{soft},(\uptau)}$ are surjective mappings from $\mathcal{D}(\Breve{\widehat{A}^{\text{soft},(\uptau)}})$ to $\breve{\mathcal{E}}^{(\uptau)}$. Furthermore, their restrictions to $\mathcal{D}(A_0^{\text{soft},(\uptau)})$ are also surjective.
        
        \item $\breve{\Gamma}_0^{\text{soft},(\uptau)} = \Gamma_0^{\text{soft},(\uptau)} |_{\mathcal{D}(A_0^{\text{soft},(\uptau)}) \dot{+} \breve{\mathcal{H}}^{\text{soft},(\uptau)}}$ and $\breve{\Gamma}_1^{\text{soft},(\uptau)} = \mathcal{P}^{(\uptau)} \Gamma_1^{\text{soft},(\uptau)} |_{\mathcal{D}(A_0^{\text{soft},(\uptau)}) \dot{+} \breve{\mathcal{H}}^{\text{soft},(\uptau)}}$.
    \end{enumerate}
\end{prop}

\begin{proof}
    (1) Boundedness and invertibility of $\breve{\Pi}^{\text{soft}}$ have been addressed above. Boundedness of $(\breve{\Pi}^{\text{soft}})^{-1}$ now follows as $\breve{\mathcal{E}}$ is finite dimensional. The same argument holds for $(\breve{\Pi}^{\text{stiff-int},(\uptau)} \oplus \breve{\Pi}^{\text{stiff-ls},(\uptau)})$. (2) Density follows from the assumption that $A_0^{\text{soft}}$ is densely defined. Closedness follows from the observation that the graph of $\Breve{\widehat{A}^{\text{soft},(\uptau)}}$ is the union of the graph of $A_0^{\text{soft}}$ with $\breve{\mathcal{H}}^{\text{soft},(\uptau)} \times \{ 0 \}$, both of which are closed. (3) follows from the formula for $S(z)$ in Proposition \ref{prop:auxops_properties}(4) and the definition $\breve{\Pi}^{\text{soft}} = \Pi^{\text{soft}} |_{\breve{\mathcal{E}}}$. Similarly, (4) follows from Proposition \ref{prop:auxops_properties}(5) and the definitions $\breve{\Pi}^{\text{soft}} = \Pi^{\text{soft}} |_{\breve{\mathcal{E}}}$ and $\breve{\Lambda}^{\text{soft}} = \mathcal{P}^{(\uptau)} \Lambda^{\text{soft}} |_{\breve{\mathcal{E}}}$. (5) Surjectivity of $\breve{\Gamma}_0^{\text{soft}}$ follows from by noting that $\breve{\Gamma}_0^{\text{soft}}$ is defined as the null extension of $(\breve{\Pi}^{\text{soft}})^{-1}$ (left inverse.) But $(\breve{\Pi}^{\text{soft}})^{-1}$ is in fact a two-sided inverse thanks to (1). Surjectivity of $\breve{\Gamma}_1^{\text{soft}}$ follows from: If $f \in L^2(Q_{\text{soft}})$, $\phi \in \breve{\mathcal{E}}$, then
    \begin{align*}
        \breve{\Gamma}_1^{\text{soft}}( (A_0^{\text{soft}})^{-1} f + \breve{\Pi}^{\text{soft}} \phi) 
        = (\breve{\Pi}^{\text{soft}})^* f + \breve{\Lambda}^{\text{soft}} \phi.
    \end{align*}
    Surjectivity of $\breve{\Gamma}_1^{\text{soft}}$ is hence a consequence of surjectivity of $(\breve{\Pi}^{\text{soft}})^*$, which was established in (1). For (6), the claim on  $\breve{\Gamma}_0^{\text{soft}}$ is immediate from the definitions. As for $\breve{\Gamma}_1^{\text{soft}}$, we can continue the computation above, to see that
    \begin{align*}
        \breve{\Gamma}_1^{\text{soft}}( (A_0^{\text{soft}})^{-1} f + \breve{\Pi}^{\text{soft}} \phi) 
        &= \mathcal{P}^{(\uptau)} (\Pi^{\text{soft}})^* f + \mathcal{P}^{(\uptau)} \Lambda^{\text{soft}} \mathcal{P}^{(\uptau)} \phi = \mathcal{P}^{(\uptau)} \left[ (\Pi^{\text{soft}})^* f + \Lambda^{\text{soft}} \mathcal{P}^{(\uptau)} \phi \right].
    \end{align*}
    The latter is precisely the action of $\mathcal{P} \Gamma_1^{\text{soft}}$ on $\mathcal{D}(A_0^{\text{soft}}) \dot{+} \breve{\mathcal{H}}^{\text{soft}}$. This completes the proof.
\end{proof}

\begin{rmk}
    The closedness of $\Breve{\widehat{A}^{\text{soft},(\uptau)}}$ relies crucially on the fact that $\breve{\mathcal{H}}^{\text{soft},(\uptau)} = \text{ran}(\breve{\Pi}^{\text{soft},(\uptau)})$ is finite dimensional. For general triples, $\widehat{A}$ is not necessarily closed nor closable.
\end{rmk}

To conclude the section, let us write down $\mathcal{R}_{\varepsilon,\text{hom}}^{(\uptau)}(z)$ with respect to the truncated objects. We will do this with respect to the decomposition $\mathcal{H} = L^2(Q_\text{soft}) \oplus ( L^2(Q_{\text{stiff-int}}) \oplus L^2(Q_{\text{stiff-ls}}) )$:

\begin{align}\label{defn:r_hom_fullspace_truncated}
    \mathcal{R}_{\varepsilon,\text{hom}}^{(\uptau)}(z) &= \left(
    \begin{matrix}
        R_{\varepsilon,\text{hom}}^{(\uptau)}(z) \\
        \breve{\Pi}^{\text{stiff}, (\uptau)} k^{(\uptau)}(z)
        \left[ R_{\varepsilon,\text{hom}}^{(\uptau)}(z) - (A_0^{\text{soft},(\uptau)} - z)^{-1} \right]
    \end{matrix} \right. \nonumber \\
    &\qquad\qquad\qquad \left. \begin{matrix}
        \left( k^{(\uptau)}(\bar{z}) \left[ R_{\varepsilon,\text{hom}}^{(\uptau)}(\bar{z}) - (A_0^{\text{soft},(\uptau)} - \bar{z})^{-1} \right] \right)^* \breve{\Pi}^{\text{stiff}, (\uptau)*} \\
        \breve{\Pi}^{\text{stiff}, (\uptau)} k^{(\uptau)}(z) \left( k^{(\uptau)}(\bar{z}) \left[ R_{\varepsilon,\text{hom}}^{(\uptau)}(\bar{z}) - (A_0^{\text{soft},(\uptau)} - \bar{z})^{-1} \right] \right)^* \breve{\Pi}^{\text{stiff}, (\uptau)*}
    \end{matrix} \right)
\end{align}
where we recall, $R_{\varepsilon,\text{hom}}^{(\uptau)}(z)$ is defined in Definition \ref{defn:r_hom_soft}, $k^{(\uptau)}(z) = \Gamma_0^{\text{soft},(\uptau)}|_{\mathcal{D}(A_0^{\text{soft},(\uptau)}) \dot{+} \breve{\mathcal{H}}^{\text{soft},(\uptau)}}$ and $\breve{\Pi}^{\text{stiff}, (\uptau)} = \breve{\Pi}^{\text{stiff-int}, (\uptau)} \oplus \breve{\Pi}^{\text{stiff-ls}, (\uptau)}$. With this in hand, we may view $\mathcal{R}_{\varepsilon,\text{hom}}^{(\uptau)}(z)$ as an operator on $L^2(Q_\text{soft}) \oplus \breve{\mathcal{H}}^{\text{stiff},(\uptau)} = L^2(Q_\text{soft}) \oplus \breve{\Pi}^{\text{stiff},(\uptau)} \mathcal{E}$.

\begin{rmk}
    Recall also that $\breve{\Lambda}_{\varepsilon}^{\text{stiff}, (\uptau)} = \breve{\Lambda}_{\varepsilon}^{\text{stiff-int}, (\uptau)} \oplus \breve{\Lambda}_{\varepsilon}^{\text{stiff-ls}, (\uptau)}$. Then by the Krein's formula,
    \begin{align}\label{defn:r_hom_soft_truncated}
        R_{\varepsilon,\text{hom}}^{(\uptau)}(z) &= (A_0^{\text{soft},(\uptau)} - z)^{-1} \nonumber \\
        &\quad - \breve{S}^{\text{soft},(\uptau)}(z) \bigg[ \breve{\Lambda}_\varepsilon^{\text{stiff},(\uptau)} + z \breve{\Pi}^{\text{stiff},(\uptau)*} \breve{\Pi}^{\text{stiff},(\uptau)} + \breve{M}^{\text{soft},(\uptau)}(z) \bigg]^{-1} \left( \breve{S}^{\text{soft},(\uptau)}(\bar{z}) \right)^*.
    \end{align}
    Therefore, $R_{\varepsilon,\text{hom}}^{(\uptau)}(z) = ( \breve{\widehat{A}_{\beta_0,\beta_1}^{\text{soft},(\uptau)}} - z )^{-1}$, where $\beta_1 = I$ and
    \begin{align}\label{eqn:trunc_rhom_beta0}
        \beta_{\varepsilon,0}^{(\uptau)}(z) = \breve{\Lambda}_\varepsilon^{\text{stiff},(\uptau)} + z  \breve{\Pi}^{\text{stiff},(\uptau)*} \breve{\Pi}^{\text{stiff},(\uptau)}.
    \end{align}
    For the validity of the choice $(\beta_0,\beta_1)$, we refer to the proof of Theorem \ref{thm:self_adjoint} below (``top left entry".) Compare this with Definition \ref{defn:r_hom_soft}. We see that we have two different parameterizations of the boundary conditions $(\beta_0,\beta_1)$, arising from two different choices of boundary triples. Formulas (\ref{defn:r_hom_fullspace_truncated}) and (\ref{defn:r_hom_soft_truncated}) will serve as quick reference for the subsequent sections.
\end{rmk}

\subsection{Self-adjointness of \texorpdfstring{$\mathcal{A}_{\varepsilon,\text{hom}}^{(\uptau)}$}{A eps hom tau}}

In this section, will use \cite[Section 4.1]{eff_behavior} and in the process supply further details to the arguments provided. Recall the notations for $\breve{\Pi}^{\text{stiff}, (\uptau)}$ and $\breve{\Lambda}_\varepsilon^{\text{stiff},(\uptau)}$. It will be convenient to set

\begin{defn}
    $\mathcal{B}_\varepsilon^{(\uptau)} := - (\breve{\Pi}^{\text{stiff},(\uptau)*})^{-1} \breve{\Lambda}_\varepsilon^{\text{stiff},(\uptau)} (\breve{\Pi}^{\text{stiff},(\uptau)})^{-1}$.
\end{defn}
Using the truncated ``soft" triple $(A_0^{\text{soft},(\uptau)}, \breve{\Lambda}^{\text{soft},(\uptau)}, \breve{\Pi}^{\text{soft},(\uptau)})$, we define

\begin{defn}\label{defn:ahom}
    Let $\mathcal{A}^{(\uptau)}_{\varepsilon,\text{hom}}$ be the operator on $L^2(Q_\text{soft}) \oplus \breve{\mathcal{H}}^{\text{stiff},(\uptau)}$ defined by
    \begin{align}
        \mathcal{D}(\mathcal{A}^{(\uptau)}_{\varepsilon,\text{hom}}) &:= \Bigg\{  \begin{pmatrix} u \\ \widehat{u} \end{pmatrix} \in L^2(Q_\text{soft}) \oplus \breve{\mathcal{H}}^{\text{stiff},(\uptau)} \text{ $:$ }
        u \in \mathcal{D}(\breve{\widehat{A}^{\text{soft}, (\uptau)}}), \text{ } \widehat{u} = \breve{\Pi}^{\text{stiff},(\uptau)} \breve{\Gamma}_0^{\text{soft},(\uptau)} u
        \Bigg\}, \\
        \mathcal{A}^{(\uptau)}_{\varepsilon,\text{hom}} \begin{pmatrix} u \\ \widehat{u} \end{pmatrix} &:= \begin{pmatrix} \breve{\widehat{A}^{\text{soft}, (\uptau)}} u \\ -  (\breve{\Pi}^{\text{stiff},(\uptau)*})^{-1} \breve{\Gamma}_1^{\text{soft},(\uptau)} u + \mathcal{B}_\varepsilon^{(\uptau)} \widehat{u} \end{pmatrix}.
    \end{align}
    Linearity of the subspace $\mathcal{D}(\mathcal{A}^{(\uptau)}_{\varepsilon,\text{hom}})$ and the operator $\mathcal{A}_{\text{hom}}$ follows from the linearity of all the operators involved. 
\end{defn}

Let us discuss some basic properties of $\mathcal{A}_{\varepsilon,\text{hom}}^{(\uptau)}$ and $\mathcal{B}_\varepsilon^{(\uptau)}$, \textit{for each fixed $\varepsilon$ and $\uptau$} (we will therefore drop $\varepsilon$ and $\uptau$ where convenient).

\begin{lem}
    $\mathcal{A}_{\text{hom}}$ is densely defined.
\end{lem}

\begin{proof}
    $u \in \mathcal{D}(\breve{\widehat{A}^{\text{soft}}})$ can be expressed as $u = (A_0^{\text{soft}})^{-1} f + \breve{\Pi}^{\text{soft}} \phi$, for some $f \in L^2(Q_\text{soft})$ and $\phi \in \breve{\mathcal{E}}$. But recall that $\mathcal{D}(\breve{\widehat{A}^{\text{soft}, (\uptau)}}) = \mathcal{D}(A_0^{\text{soft},(\uptau)}) \dot{+} \breve{\mathcal{H}}^{\text{soft},(\uptau)}$ is a (vector space) direct sum, so we may vary $(A_0^{\text{soft}})^{-1} f$ independently of $\breve{\Pi}^{\text{soft}} \phi$. Since $\mathcal{D}(A_0^{\text{soft},(\uptau)})$ is dense in $L^2(Q_\text{soft})$, and 
    \begin{align*}
        \widehat{u} = \breve{\Pi}^{\text{stiff}} \breve{\Gamma}_0^{\text{soft}} u = \breve{\Pi}^{\text{stiff}} \phi,
    \end{align*}
    ranging through $\phi \in \breve{\mathcal{E}}$ implies that the second component of $\mathcal{D}(\mathcal{A}_{\text{hom}})$ equals (!)~$\breve{\mathcal{H}}^{\text{stiff}}$.
\end{proof}

\begin{lem}\label{lem:ahom_symmetric}
    $\mathcal{A}_{\text{hom}}$ is symmetric if and only if $\mathcal{B}$ is self-adjoint.
\end{lem}

\begin{proof}
    This follows from the Green's identity (Theorem \ref{thm:greens_id}). See \cite[Lemma~4.3]{eff_behavior} for a proof.
\end{proof}

By Proposition \ref{prop:trunc_aux_properties}, $\mathcal{B}$ is a bounded self-adjoint operator on $\breve{\mathcal{H}}^{\text{stiff}}$. As explained in \cite[Section 4.1]{eff_behavior}, the point of singling out the operator $\mathcal{B}$ is because the self-adjointness of $\mathcal{B}$ implies the \textit{self-adjointness} of $\mathcal{A}_{\text{hom}}$:

\begin{thm}\label{thm:self_adjoint}
    Fix $\varepsilon>0$ and $\uptau \in Q'$. Suppose that $\mathcal{B}$ is self-adjoint. Then $\mathcal{A}_{\text{hom}}$ is self-adjoint. Furthermore its resolvent $(\mathcal{A}_{\text{hom}} - z)^{-1}$ is defined for all $z \in \mathbb{C}\setminus\mathbb{R}$ by the following block matrix decomposition with respect to $L^2(Q_\text{soft}) \oplus \breve{\mathcal{H}}^{\text{stiff}}$:
    \begin{align}\label{eqn:ahom_block_matrix}
        &(\mathcal{A}_{\text{hom}} - z)^{-1} \nonumber\\ &= \begin{pmatrix}
            R(z) & 
            \left( k(\bar{z}) \left[ R(\bar{z}) - (A_0^{\text{soft}} - \bar{z})^{-1} \right] \right)^* ( \breve{\Pi}^{\text{stiff}} )^* \\
            \breve{\Pi}^{\text{stiff}} k(z) \left[ R(z) - (A_0^{\text{soft}} - z)^{-1} \right] & 
            \breve{\Pi}^{\text{stiff}} k(z) \left( k(\bar{z}) \left[ R(\bar{z}) - (A_0^{\text{soft}} - \bar{z})^{-1} \right] \right)^* ( \breve{\Pi}^{\text{stiff}} )^*
        \end{pmatrix},
    \end{align}
    where we define $k(z) := \breve{\Gamma}_0^{\text{soft}} \stackrel{\text{Prop \ref{prop:trunc_aux_properties}(6)}}{=} \Gamma_0^{\text{soft}} |_{\mathcal{D}(A_0^{\text{soft},(\uptau)}) \dot{+} \breve{\mathcal{H}}^{\text{soft},(\uptau)}}$ and
    \begin{align}
        R(z) := \left( \breve{\widehat{A}^{\text{soft},(\uptau)}}_{(\breve{\Pi}^{\text{stiff}})^* (\mathcal{B} - z ) \breve{\Pi}^{\text{stiff}} , I} - z \right)^{-1}.
    \end{align}
\end{thm}

\begin{rmk}
    By Proposition \ref{prop:trunc_aux_properties}(6), $k(z)$ as defined in this theorem coincides with the one in (\ref{eqn:r_eff_fullspace}). Thus (\ref{eqn:ahom_block_matrix}) is precisely $\mathcal{R}_{\varepsilon,\text{hom}}^{(\uptau)}(z)$.
\end{rmk}

\begin{proof}[Proof of Theorem \ref{thm:self_adjoint}.]
    See Appendix \ref{appendix:krein_calc}.
\end{proof}

\section{Homogenization result}\label{sect:homo_result}

This section summarizes the results thus far into a fibre-wise (for each $\uptau$) homogenization result. To begin, we collect the key ingredients of Sections \ref{sect:norm_resolvent_asymp} and \ref{sect:identifying_homo_op} required for stating the result. We have the following spaces
\begin{align}
    \breve{\mathcal{E}}^{(\uptau)} = \mathcal{P}^{(\uptau)}\mathcal{E} = \mathcal{P}_{\text{stiff-int}}^{(\uptau)}\mathcal{E} \oplus \mathcal{P}_{\text{stiff-ls}}^{(\uptau)}\mathcal{E} &= \text{span} \{ \psi_1^{\text{stiff-int},(\uptau)} \} \oplus \text{span} \{ \psi_1^{\text{stiff-ls},(\uptau)} \}, \\
    \breve{\mathcal{H}}^{\text{stiff},(\uptau)} = \breve{\mathcal{H}}^{\text{stiff-int},(\uptau)} \oplus \breve{\mathcal{H}}^{\text{stiff-ls},(\uptau)} &= \text{ran}( \Pi^{\text{stiff-int}, (\uptau)} |_{\breve{\mathcal{E}}^{(\uptau)}} ) \oplus \text{ran}( \Pi^{\text{stiff-ls}, (\uptau)} |_{\breve{\mathcal{E}}^{(\uptau)}} ).
\end{align}
We denote by $\Psi_1$ the lifts of $\psi_1$ into their respective stiff spaces. That is,
\begin{align}
    \Psi_1^{\bigstar,(\uptau)} := \breve{\Pi}^{\bigstar,(\uptau)} \psi_1^{\bigstar,(\uptau)} = \Pi^{\bigstar,(\uptau)} \psi_1^{\bigstar,(\uptau)}, \quad (\bullet, \bigstar) \in \{ (\text{int}, \text{stiff-int}), (\text{ls}, \text{stiff-ls}) \}.
\end{align}
The homogenized operator $\mathcal{A}_{\varepsilon, \text{hom}}^{(\uptau)}$ is defined to have domain
\begin{align}\label{eqn:conclusion_domain}
    \mathcal{D}(\mathcal{A}^{(\uptau)}_{\varepsilon,\text{hom}}) = \Bigg\{  \begin{pmatrix} u \\ \widehat{u} \end{pmatrix} &\in L^2(Q_\text{soft}) \oplus \breve{\mathcal{H}}^{\text{stiff},(\uptau)} \text{ $:$ } \nonumber\\
    &u \in \mathcal{D}(A_0^{\text{soft},(\uptau)}) \text{ }\dot{+}\text{ } \text{ran}(\Pi^{\text{soft},(\uptau)} |_{\breve{\mathcal{E}}^{(\uptau)} }  ),
    \quad \widehat{u} = \breve{\Pi}^{\text{stiff},(\uptau)} \breve{\Gamma}_0^{\text{soft},(\uptau)} u
    \Bigg\}.
\end{align}

\begin{rmk}
    $\mathcal{D}(A_0^{\text{soft},(\uptau)}) = H^2(Q_\text{soft}) \cap H^1_0(Q_\text{soft})$, which is independent of $\uptau$.
\end{rmk}

\begin{defn}\label{defn:amax}
    We write $-(\nabla + i \uptau)^2$ to mean the operator of $\breve{\widehat{A}^{\text{soft},(\uptau)}}$, that is, $-(\nabla + i \uptau)^2$ is the magnetic Laplacian on $Q_\text{soft}$ with (zero) Dirichlet BCs, extended by zero on $\breve{\mathcal{H}}^{\text{stiff},(\uptau)}$.
\end{defn}

For its action, we first note that a typical $u \in \mathcal{D}(-(\nabla + i \uptau)^2)$ may be written as
\begin{align}
    u = (A_0^{\text{soft},(\uptau)})^{-1} f + \Pi^{\text{soft},(\uptau)} (a \psi_1^{\text{stiff-int},(\uptau)} + b \psi_1^{\text{stiff-ls},(\uptau)} ), \quad f \in L^2(Q_\text{soft}), \quad a,b \in \mathbb{C}.
\end{align}
If we further expand $\widehat{u} \in \breve{\mathcal{H}}^{\text{stiff},(\uptau)}$ into $(\widehat{u}_{\text{stiff-int}}, \widehat{u}_{\text{stiff-ls}}) \in \breve{\mathcal{H}}^{\text{stiff-int},(\uptau)} \oplus \breve{\mathcal{H}}^{\text{stiff-ls},(\uptau)}$, then by the definition of $\Gamma_0$, the condition on $\widehat{u}$ in (\ref{eqn:conclusion_domain}) may be written as 
\begin{align*}
    \widehat{u} = \begin{pmatrix}
        \widehat{u}_{\text{stiff-int}} \\
        \widehat{u}_{\text{stiff-ls}}
    \end{pmatrix} = \begin{pmatrix}
        a \Pi^{\text{stiff-int},(\uptau)} \psi_1^{\text{stiff-int},(\uptau)} \\
        b \Pi^{\text{stiff-ls},(\uptau)} \psi_1^{\text{stiff-ls},(\uptau)}
    \end{pmatrix} = \begin{pmatrix}
        a \Psi_1^{\text{stiff-int},(\uptau)} \\
        b \Psi_1^{\text{stiff-ls},(\uptau)}
    \end{pmatrix}.
\end{align*}
Therefore, the action of $\mathcal{A}_{\varepsilon,\text{hom}}^{(\uptau)}$ may be written in with respect to the decomposition $L^2(Q_\text{soft}) \oplus \breve{\mathcal{H}}^{\text{stiff-int},(\uptau)} \oplus \breve{\mathcal{H}}^{\text{stiff-ls},(\uptau)}$ as
\begin{align}\label{eqn:ahom_action}
    \mathcal{A}_{\varepsilon,\text{hom}}^{(\uptau)} \begin{pmatrix}
        u \\
        \widehat{u}_{\text{stiff-int}} \\
        \widehat{u}_{\text{stiff-ls}}
    \end{pmatrix} 
    &= \begin{pmatrix}
        -(\nabla + i \uptau)^2 u \\
        -(\breve{\Pi}^{\text{stiff-int},(\uptau)*})^{-1} \mathcal{P}_{\text{stiff-int}}^{(\uptau)} \left[ \Gamma_1^{\text{soft},(\uptau)} u + \Gamma_{\varepsilon,1}^{\text{stiff-int},(\uptau)} \left( a \Psi_1^{\text{stiff-int},(\uptau)}  \right) \right] \\
        -(\breve{\Pi}^{\text{stiff-ls},(\uptau)*})^{-1} \mathcal{P}_{\text{stiff-ls}}^{(\uptau)} \left[ \Gamma_1^{\text{soft},(\uptau)} u + \Gamma_{\varepsilon,1}^{\text{stiff-ls},(\uptau)} \left( b \Psi_1^{\text{stiff-ls},(\uptau)}  \right) \right]
    \end{pmatrix} \\
    &= \begin{pmatrix}
        -(\nabla + i \uptau)^2 u \\
        -(\breve{\Pi}^{\text{stiff-int},(\uptau)*})^{-1} \mathcal{P}_{\text{stiff-int}}^{(\uptau)} \Gamma_1^{\text{soft},(\uptau)} u \\
        -(\breve{\Pi}^{\text{stiff-ls},(\uptau)*})^{-1} \mathcal{P}_{\text{stiff-ls}}^{(\uptau)} \left[ \Gamma_1^{\text{soft},(\uptau)} u + {\color{blue}\varepsilon^{-2}} \mu_1^{\text{stiff-ls},(\uptau)} \left( b \psi_1^{\text{stiff-ls},(\uptau)}  \right) \right]
    \end{pmatrix}.
\end{align}

To deduce this from Definition \ref{defn:ahom}, we have used: In the first equality, that $\Lambda = \Gamma_1 \Pi$ and $\Pi$ is boundedly invertible. In the second equality, that $\breve{\Lambda}_\varepsilon^{\text{stiff-ls}}$ acts as multiplication by $\varepsilon^{-2} \mu_1^{\text{stiff-ls}, (\uptau)}$, and similarly for $\breve{\Lambda}_\varepsilon^{\text{stiff-int}}$. Recall also that $\mu_1^{\text{stiff-int},(\uptau)}\equiv 0$ (Proposition \ref{prop:dtn_properties}).

\begin{lem}\label{lem:pi_inverse_adjoint}
    The action of $(\breve{\Pi}^{\bigstar,(\uptau)*})^{-1} : \text{span} \{ \psi_1^{\bigstar, (\uptau)} \} \rightarrow \text{ran}(\Pi^{\bigstar,(\uptau)}|_{\text{span} \{ \psi_1^{\bigstar, (\uptau)} \}})$ is given by \begin{align*}
        \psi_1^{\bigstar,(\uptau)} \mapsto \| \Psi_1^{\bigstar,(\uptau)} \|^{-2} \Psi_1^{\bigstar,(\uptau)}.
    \end{align*}
\end{lem}
\begin{proof}
    We will drop $\bullet$, $\bigstar$, and $\uptau$ where convenient. It suffices to figure out its action on $\psi_1$. Since $(\breve{\Pi}^*)^{-1} = (\breve{\Pi}^{-1})^*$,
    \begin{align*}
        &\left( (\breve{\Pi}^{-1})^* \psi_1, \Psi_1 \right)_{L^2(Q_{\bigstar})}
        = \left( \psi_1, (\breve{\Pi})^{-1} \Psi_1 \right)_{L^2(Q_{\bigstar})} = \left( \psi_1, \psi_1 \right)_{L^2(\Gamma_{\bullet})} = 1.
    \end{align*}
    But $(\breve{\Pi}^{-1})^* \psi_1 \in \text{ran}(\Pi|_{\breve{\mathcal{E}}})$ is a multiple of $\Psi_1$, say $(\breve{\Pi}^{-1})^* \psi_1 = c \Psi_1$. By the above calculation, we must have $c = 1/\| \Psi_1 \|^2 $.
\end{proof}

We are now in the position to state the homogenization result.

\begin{thm}[Fibre-wise homogenization result] \label{thm:fibre_homo_result}
    With the homogenized operator $\mathcal{A}_{\varepsilon, \text{hom}}^{(\uptau)}$ as defined above, we have that:
    \begin{itemize}
        \item $\mathcal{A}_{\varepsilon, \text{hom}}^{(\uptau)}$ is asymptotically close to our main model operator $A_\varepsilon^{(\uptau)}$ in the norm-resolvent sense, with an $O(\varepsilon^2)$ estimate. This estimate is uniform in $z \in K_\sigma$ and $\uptau \in Q'$.
        \item The resolvent $(\mathcal{A}_{\varepsilon, \text{hom}}^{(\uptau)} - z)^{-1}$ is given by  $\mathcal{R}_{\varepsilon,\text{hom}}^{(\uptau)}(z)$ (see Definition \ref{defn:r_hom_soft} or (\ref{defn:r_hom_fullspace_truncated}) or (\ref{eqn:ahom_block_matrix})).
        \item $\mathcal{A}_{\varepsilon, \text{hom}}^{(\uptau)}$ is self-adjoint on $L^2(Q_\text{soft}) \oplus \breve{\mathcal{H}}^{\text{stiff},(\uptau)}$, and its null-extension to the full space $L^2(Q) = L^2(Q_{\text{soft}}) \oplus L^2(Q_{\text{stiff-int}}) \oplus L^2(Q_{\text{stiff-ls}})$, which we will still denote by $\mathcal{A}_{\varepsilon, \text{hom}}^{(\uptau)}$, is symmetric.
        \item $\mathcal{A}_{\varepsilon, \text{hom}}^{(\uptau)}$ depends on $\varepsilon$ only through its action on the third component $\breve{\mathcal{H}}^{\text{stiff-ls},(\uptau)}$.
    \end{itemize}
    
\end{thm}

\begin{proof}
    The only point that remains to be shown is why the operator $\mathcal{R}_{\varepsilon,\text{hom}}^{(\uptau)}(z)$, while initially defined for $z \in \mathbb{C}\setminus\mathbb{R}$ (Theorem \ref{thm:self_adjoint}), can be extended to the whole resolvent set $\rho(\mathcal{A}_{\varepsilon, \text{hom}}^{(\uptau)})$. This is due to the analyticity of the resolvent $(\mathcal{A}_{\text{hom}} - z)^{-1}$: Given $z_0 \in \rho(\mathcal{A}_{\varepsilon, \text{hom}}^{(\uptau)}) \cap \mathbb{R}$,  we can always find an open ball $B(z,\eps_z)$, with $z \in \mathbb{C}\setminus\mathbb{R}$, and $z_0 \in B(z,\eps_z)$, such that the formula (\ref{eqn:ahom_block_matrix}) holds.
\end{proof}

\begin{rmk}
    Explicit expressions for $\breve{\Pi}$ are available for the case $Q_{\text{stiff-int}}$ (for all $\uptau$) and for $Q_{\text{stiff-ls}}$ (for $\uptau=0$). See Proposition \ref{prop:dtn_properties} for the formula for the eigenfunction $\psi_1$.
\end{rmk}

\section{A closer look at the homogenized operator}\label{sect:closer_look}

In this section, we study the bottom right entry of $(\mathcal{A}_{\text{hom}} - z )^{-1}$, i.e.
\begin{align}
    P_{\breve{\mathcal{H}}^{\text{stiff},(\uptau)}} (\mathcal{A}_{\text{hom}} - z )^{-1} P_{\breve{\mathcal{H}}^{\text{stiff},(\uptau)}} \stackrel{\text{(\ref{eqn:r_eff_fullspace})}}{=}
    \begin{matrix}
        & \breve{\mathcal{H}}^{\text{stiff-int},(\uptau)} & \breve{\mathcal{H}}^{\text{stiff-ls},(\uptau)} \\
        \breve{\mathcal{H}}^{\text{stiff-int},(\uptau)}  & \tikzmark{a}{$a_{22}$} & \tikzmark{b}{$a_{23}$} \\
        \breve{\mathcal{H}}^{\text{stiff-ls},(\uptau)} & \tikzmark{e}{$a_{32}$} & \tikzmark{d}{$a_{33}$}
    \end{matrix},
\end{align}
\begin{tikzpicture}[remember picture,overlay]
    \draw[thick,decorate,decoration={calligraphic straight parenthesis}] ($(e.south west)+(-.9em,-.5em)$) -- ($(a.north west)+(-.9em,+.5em)$);
    \draw[thick,decorate,decoration={calligraphic straight parenthesis}] ($(b.north east)+(+.9em,+.5em)$) -- ($(d.south east)+(+.9em,-.5em)$);
\end{tikzpicture}

\noindent
with a particular focus on the diagonal entries $a_{22}$ and $a_{33}$. To begin, we first apply an isomorphism $\breve{\mathcal{H}}^{\text{stiff-int},(\uptau)} \oplus \breve{\mathcal{H}}^{\text{stiff-ls},(\uptau)} \cong \mathbb{C}^2$ so that we do not have to deal with a varying space. Let us define:

\begin{defn}
    Set $j_{\text{stiff-int}}^{(\uptau)}: \text{ran}(\Pi^{\text{stiff-int},(\uptau)} |_{\mathcal{P}_{\text{stiff-int}}^{(\uptau)} \mathcal{E} } ) \rightarrow \mathbb{C}$ to be the unitary mapping
    \begin{align*}
        \Psi_1^{\text{stiff-int},(\uptau)} \mapsto \| \Psi_1^{\text{stiff-int},(\uptau)} \|_{L^2(Q_{\text{stiff-int}})}.
    \end{align*}
    (Note: $\breve{\mathcal{H}}^{\text{stiff-int},(\uptau)} = \text{ran}(\Pi^{\text{stiff-int},(\uptau)} |_{\mathcal{P}_{\text{stiff-int}}^{(\uptau)} \mathcal{E} } ) $) And similarly for $j_{\text{stiff-ls}}^{(\uptau)}$. Set $j^{(\uptau)} = j_{\text{stiff-int}}^{(\uptau)} \oplus j_{\text{stiff-ls}}^{(\uptau)}$.
\end{defn}

In this case, the operator $j_{\text{stiff-int}}^{(\uptau)} \Pi^{\text{stiff-int}} |_{\breve{\mathcal{E}}^{(\uptau)}}: \breve{\mathcal{E}} \mapsto \mathbb{C}$ is a mapping taking $\psi_1^{\text{stiff-int},(\uptau)}$ to ($\Psi_1^{\text{stiff-int},(\uptau)}$, and then to) $\|\Psi_1^{\text{stiff-int},(\uptau)} \|$. For the reader's convenience we compute the inverse of its adjoint:

\begin{lem}
    For $\bigstar \in \{ \text{stiff-int}, \text{stiff-ls} \}$, $((j_{\bigstar}^{(\uptau)} \breve{\Pi}^{\bigstar,(\uptau)})^*)^{-1} : \text{span}\{ \psi_1^{\bigstar,(\uptau)} \} \rightarrow \mathbb{C} $ is given by
    \begin{align*}
        \psi_1^{\bigstar,(\uptau)} \mapsto \| \Psi_1^{\bigstar,(\uptau)} \|^{-1}.
    \end{align*}
\end{lem}

\begin{proof}
    We drop $\bigstar$ and $\uptau$. We simply have to note that $((j \breve{\Pi})^*)^{-1} = (j^*)^{-1} (\breve{\Pi}^*)^{-1}$. Since $j$ is unitary, $(j^*)^{-1} = j$. The result now follows from Lemma \ref{lem:pi_inverse_adjoint}.
\end{proof}

Under this identification, we may view our homogenized operator as an operator on $L^2(Q_\text{soft}) \oplus \mathbb{C}^2$, which we will still denote by $\mathcal{A}_{\varepsilon,\text{hom}}^{(\uptau)}$. Let us write $\mathbb{C}_{\text{stiff-int}} \oplus \mathbb{C}_{\text{stiff-ls}}$ to distinguish between the copies of $\mathbb{C}$. In that case, our homogenized operator may be written as
\begin{align}
    \mathcal{D}(\mathcal{A}_{\varepsilon,\text{hom}}^{(\uptau)}) &:= \{ (u, \beta_{\text{stiff-int}}, \beta_{\text{stiff-ls}}) \in L^2(Q_{\text{soft}}) \oplus \mathbb{C}_{\text{stiff-int}} \oplus \mathbb{C}_{\text{stiff-ls}} \text{ $:$ } \nonumber\\
    &\qquad u \in \mathcal{D}(A_0^{\text{soft}}) \text{ $\dot{+}$ } \text{ran}(\Pi^{\text{soft},(\uptau)}\mathcal{P}^{(\uptau)} ), \nonumber\\
    &\qquad \beta_{\text{stiff-int}} = j^{(\uptau)}_{\text{stiff-int}} \Pi^{\text{stiff-int},(\uptau)} \Gamma_0^{\text{soft},(\uptau)} u, \quad
    \beta_{\text{stiff-ls}} = j^{(\uptau)}_{\text{stiff-ls}} \Pi^{\text{stiff-ls},(\uptau)} \Gamma_0^{\text{soft},(\uptau)} u \},
\end{align}
\begin{align}
    \mathcal{A}^{(\uptau)}_{\varepsilon,\text{hom}} \begin{pmatrix}
        u \\
        \beta_{\text{stiff-int}} \\
        \beta_{\text{stiff-ls}}
    \end{pmatrix} &= \begin{pmatrix}
        - (\nabla + i \uptau)^2 u \\
        - ((j_{\text{stiff-int}}^{(\uptau)} \breve{\Pi}^{\text{stiff-int},(\uptau)})^*)^{-1} \mathcal{P}^{(\uptau)}_{\text{stiff-int}} \Gamma_1^{\text{soft},(\uptau)} u \\
        - ((j_{\text{stiff-ls}}^{(\uptau)} \breve{\Pi}^{\text{stiff-ls},(\uptau)})^*)^{-1} \mathcal{P}^{(\uptau)}_{\text{stiff-ls}} \left[ \Gamma_1^{\text{soft},(\uptau)} u + \varepsilon^{-2} \mu_1^{\text{stiff-ls},(\uptau)} (j_{\text{stiff-ls}}^{(\uptau)} \breve{\Pi}^{\text{stiff-ls},(\uptau)})^{-1} \beta_{\text{stiff-ls}} \right] 
    \end{pmatrix} \nonumber\\
    &=: \begin{pmatrix}
        - (\nabla + i \uptau)^2 u \\
        T_{\varepsilon,\text{stiff-int}}^{(\uptau)} (u, \beta_{\text{stiff-int}}, \beta_{\text{stiff-ls}})^\top \\
        T_{\varepsilon,\text{stiff-ls}}^{(\uptau)} (u, \beta_{\text{stiff-int}}, \beta_{\text{stiff-ls}})^\top
    \end{pmatrix}.
\end{align}

\subsection{Stiff-interior to stiff-interior}

Let us now figure out the action of $P_{\mathbb{C}_{\text{stiff-int}}} (\mathcal{A}_{\text{hom}} -z )^{-1} P_{\mathbb{C}_{\text{stiff-int}}}$, which is a multiplication by a constant. We will drop $\varepsilon$ and $\uptau$ where convenient. We will also assume that $z\in K_\sigma$.

The operator in question takes $\delta \in \mathbb{C}$, solves the system
\begin{align}\label{eqn:system_dispersion_relation}
    \begin{cases}
        -(\nabla + i \uptau)^2 u - zu = 0, \\
        T_{\text{stiff-int}} (u, \beta_{\text{stiff-int}}, \beta_{\text{stiff-ls}})^\top -z \beta_{\text{stiff-int}} = \delta, \\
        T_{\text{stiff-ls}} (u, \beta_{\text{stiff-int}}, \beta_{\text{stiff-ls}})^\top -z \beta_{\text{stiff-ls}} = 0,
    \end{cases}
\end{align}
and then outputs $\beta_{\text{stiff-int}}$. (Recall Definition \ref{defn:amax} for the notation $-(\nabla + i\uptau)^2$.) Our goal is to write $T_{\text{stiff-int}} (u, \beta_{\text{stiff-int}}, \beta_{\text{stiff-ls}})^T$ as a constant times $\beta_{\text{stiff-int}}$. To assist us, we define

\begin{defn}
    $v := \Pi^{\text{soft},(\uptau)}(\psi_1^{\text{stiff-int},(\uptau)}, 0)$, and $w := \Pi^{\text{soft} ,(\uptau)} (0, \psi_1^{\text{stiff-ls},(\uptau)})$.
\end{defn}

Observe that if $u = (A_0^{\text{soft},(\uptau)})^{-1} f + \Pi^{\text{soft},(\uptau)}(a \psi_1^{\text{stiff-int},(\uptau)} + b \psi_1^{\text{stiff-ls},(\uptau)})$ for some $f \in L^2(Q_{\text{soft}})$ and $a$, $b\in\mathbb{C}$, then $u = (A_0^{\text{soft}})^{-1} f + av + bw$. That means $\widetilde{u} := u - av - bw \in \mathcal{D}(A_0^{\text{soft}})$. So $(-(\nabla + i \uptau)^2 - z) \widetilde{u} = (A_0^{\text{soft}} - z) \widetilde{u}$. In fact,
\begin{align*}
    (-(\nabla + i \uptau)^2 - z) \widetilde{u} = \underbrace{\cancel{(-(\nabla + i\uptau)^2 - z) u}}_{\text{By (\ref{eqn:system_dispersion_relation})}} - (\underbrace{\cancel{-(\nabla + i\uptau)^2}}_{\text{Since $\widehat{A}(\Pi \phi) = 0$} } -z ) av - (\underbrace{\cancel{-(\nabla + i\uptau)^2}}_{\text{Since $\widehat{A}(\Pi \phi) = 0$} } -z ) bw = zav + zbw.
\end{align*}
This implies that 
\begin{align*}
    \widetilde{u} = za (A_0^{\text{soft}} - z)^{-1} v + zb (A_0^{\text{soft}} - z)^{-1} w.
\end{align*}
The key is that $a$ and $b$ are related to $\beta_{\text{stiff-int}}$ and $\beta_{\text{stiff-ls}}$ respectively by
\begin{align*}
    \beta_{\text{stiff-int}} = a \| \Psi_1^{\text{stiff-int}} \|, \quad \beta_{\text{stiff-ls}} = b \| \Psi_1^{\text{stiff-ls}} \|.
\end{align*}
This is a consequence of the computation of $\widehat{u}$ in the previous section, and the definition of the isomorphism $j$. This allows us to write
\begin{align}\label{eqn:dispersion_relation_v1}
    &T_{\text{stiff-int}} \begin{pmatrix}
        u \\
        \beta_{\text{stiff-int}} \\
        \beta_{\text{stiff-ls}}
    \end{pmatrix} = T_{\text{stiff-int}} \begin{pmatrix}
        a z (A_0^{\text{soft}} - z)^{-1} v + b z (A_0^{\text{soft}} - z)^{-1} w + av + bw \\
        \beta_{\text{stiff-int}} \\
        \beta_{\text{stiff-ls}}
    \end{pmatrix} \nonumber\\
    &= T_{\text{stiff-int}} \begin{pmatrix}
        \dfrac{\beta_{\text{stiff-int}}}{\| \Psi_1^{\text{stiff-int}} \|} z (A_0^{\text{soft}} - z)^{-1} v 
        + \dfrac{\beta_{\text{stiff-ls}}}{\| \Psi_1^{\text{stiff-ls}} \|} z (A_0^{\text{soft}} - z)^{-1} w 
        + \dfrac{\beta_{\text{stiff-int}}}{\| \Psi_1^{\text{stiff-int}} \|}v 
        + \dfrac{\beta_{\text{stiff-ls}}}{\| \Psi_1^{\text{stiff-ls}} \|}w \\
        \beta_{\text{stiff-int}} \\
        \beta_{\text{stiff-ls}}
    \end{pmatrix} \nonumber\\
    &= T_{\text{stiff-int}} \begin{pmatrix}
        \dfrac{\beta_{\text{stiff-int}}}{\| \Psi_1^{\text{stiff-int}} \|} z (A_0^{\text{soft}} - z)^{-1} v + \dfrac{\beta_{\text{stiff-int}}}{\| \Psi_1^{\text{stiff-int}} \|}v \\
        \beta_{\text{stiff-int}} \\
        0
    \end{pmatrix} + T_{\text{stiff-int}} \begin{pmatrix}
        \dfrac{\beta_{\text{stiff-ls}}}{\| \Psi_1^{\text{stiff-ls}} \|} z (A_0^{\text{soft}} - z)^{-1} w + \dfrac{\beta_{\text{stiff-ls}}}{\| \Psi_1^{\text{stiff-ls}} \|}w \\
        0 \\
        \beta_{\text{stiff-ls}}
    \end{pmatrix} \nonumber\\
    &= \frac{\beta_{\text{stiff-int}}}{\| \Psi_1^{\text{stiff-int}} \|} T_{\text{stiff-int}} \begin{pmatrix}
         z (A_0^{\text{soft}} - z)^{-1} v + v \\
        \| \Psi_1^{\text{stiff-int}} \| \\
        0
    \end{pmatrix} + \frac{\beta_{\text{stiff-ls}}}{\| \Psi_1^{\text{stiff-ls}} \|} T_{\text{stiff-int}} \begin{pmatrix}
         z (A_0^{\text{soft}} - z)^{-1} w + w \\
        0 \\
        \| \Psi_1^{\text{stiff-ls}} \|
        \end{pmatrix}.
\end{align}
We would like to write $\beta_{\text{stiff-ls}}$ in terms of $\beta_{\text{stiff-int}}$. To do this, we use the third equation of the system (\ref{eqn:system_dispersion_relation}), and (\ref{eqn:dispersion_relation_v1}) with $T_{\text{stiff-int}}$ replaced by $T_{\text{stiff-ls}}$, we obtain
\begin{align}
    &\beta_{\text{stiff-ls}} = \frac{1}{z} T_{\text{stiff-ls}} \begin{pmatrix}
        u \\
        \beta_{\text{stiff-int}} \\
        \beta_{\text{stiff-ls}}
    \end{pmatrix} \nonumber\\
    &= \frac{\beta_{\text{stiff-int}}}{z \| \Psi_1^{\text{stiff-int}} \|} T_{\text{stiff-ls}} \begin{pmatrix}
        z (A_0^{\text{soft}} - z )^{-1} v + v \\
        \| \Psi_1^{\text{stiff-int}} \| \\
        0
    \end{pmatrix} + \frac{\beta_{\text{stiff-ls}}}{z \| \Psi_1^{\text{stiff-ls}} \|} T_{\text{stiff-ls}} \begin{pmatrix}
        z (A_0^{\text{soft}} - z )^{-1} w + w \\
        0 \\
        \| \Psi_1^{\text{stiff-ls}} \|
    \end{pmatrix}.
\end{align}
%
%
Rearranging and plugging this back into (\ref{eqn:dispersion_relation_v1}),
\begin{align}
    &T_{\text{stiff-int}} \begin{pmatrix}
        u \\
        \beta_{\text{stiff-int}} \\
        \beta_{\text{stiff-ls}}
    \end{pmatrix} = \beta_{\text{stiff-int}} \bigg\{ \frac{1}{\| \Psi_1^{\text{stiff-int}} \|} T_{\text{stiff-int}} \begin{pmatrix}
        z (A_0^{\text{soft}} -z )^{-1}v + v \\
        \| \Psi_1^{\text{stiff-int}} \| \\
        0
    \end{pmatrix} \nonumber \\
    &+ \frac{1}{z \| \Psi_1^{\text{stiff-int}} \| \| \Psi_1^{\text{stiff-ls}} \|} T_{\text{stiff-ls}} \begin{pmatrix}
        z (A_0^{\text{soft}} -z )^{-1}v + v \\
        \| \Psi_1^{\text{stiff-int}} \| \\
        0
    \end{pmatrix} T_{\text{stiff-int}} \begin{pmatrix}
        z (A_0^{\text{soft}} -z )^{-1}w + w \\
        0 \\
        \| \Psi_1^{\text{stiff-ls}} \| 
    \end{pmatrix} \times \nonumber\\
    & \qquad\qquad\qquad \times \left( 1 - \frac{1}{z \| \Psi_1^{\text{stiff-ls}}\|} T_{\text{stiff-ls}} \begin{pmatrix}
        z (A_0^{\text{soft}} -z )^{-1}w + w \\
        0 \\
        \| \Psi_1^{\text{stiff-ls}} \| 
    \end{pmatrix} \right)^{-1} \bigg\} \label{eqn:dispersion_relation_divisionstep}\\
    &=: \beta_{\text{stiff-int}} \left\{ K_{a,\text{stiff-int}}(\uptau, z) + K_{b,\text{stiff-int}}(\uptau, z) \right\} \\
    &=: \beta_{\text{stiff-int}} K_{\text{stiff-int}}(\uptau, z). \label{eqn:dispersion_stiffint}
\end{align}

The derivation above suggests the following:

\begin{thm}\label{thm:stiff_int_dispersion_reln}
    For $\eps>0$ small enough, independently of $z \in K_\sigma$ and $\tau \in Q'$, \newline $P_{\C_\stin} (\mathcal{A}_{\eps,\text{hom}}^{(\tau)} -z )^{-1} P_{\C_\stin}$ is the operator on $\mathbb{C}_{\text{stiff-int}}$ of multiplication by the number \newline $(K_\text{stiff-int}(\uptau,z) - z)^{-1}$. In the notation of Section \ref{sect:preliminaries_main}, this means that
    \begin{align}
        P_{\mathbb{C}_{\text{stiff-int}}} (\mathcal{A}^{(\uptau)}_{\eps, \text{hom}} -z )^{-1} P_{\mathbb{C}_{\text{stiff-int}}} = M_{(K_\text{stiff-int}(\uptau,z) - z)^{-1}}.
    \end{align}
\end{thm}

\begin{proof}
    To ensure that $K_\text{stiff-int}(\uptau,z)$ is well-defined, we need to show that the denominator of the second term in (\ref{eqn:dispersion_relation_divisionstep}) is non-zero. We will do this by showing that it has a non-zero imaginary component. This requires us to uncover the action of $T_{\stls}$. First, we observe that
    \begin{align}
        z(A_0^{\text{soft},(\uptau)} - z)^{-1}w + w 
        &= (I + z(A_0^{\text{soft},(\uptau)} -z)^{-1}) w \nonumber\\
        &= (I + z(A_0^{\text{soft},(\uptau)} -z)^{-1}) \Pi^{\text{soft},(\uptau)} (0 + \psi_1^{\text{stiff-ls},(\uptau)}) \nonumber\\
        &= S^{\text{soft},(\uptau)}(z) (0 + \psi_1^{\text{stiff-ls},(\uptau)}).
    \end{align}
    In the action of $T_{\text{stiff-ls}}$, we need to apply to the above, $\Gamma_1^{\text{soft},(\uptau)}$, then $\mathcal{P}^{(\uptau)}_{\text{stiff-ls}}$, and then $((j \breve{\Pi})^*)^{-1}$:
    \begin{align}
        &((j \breve{\Pi})^*)^{-1} \mathcal{P}^{(\uptau)}_{\text{stiff-ls}} \Gamma_1^{\text{soft},(\uptau)} S^{\text{soft},(\uptau)}(z) (0 + \psi_1^{\text{stiff-ls},(\uptau)}) \nonumber\\
        &\qquad = ((j \breve{\Pi})^*)^{-1} \mathcal{P}^{(\uptau)}_{\text{stiff-ls}} M^{\text{soft},(\uptau)}(z) (0 + \psi_1^{\text{stiff-ls},(\uptau)}) \nonumber\\
        &\qquad = ((j \breve{\Pi})^*)^{-1} \left\langle M^{\text{soft},(\uptau)}(z) (0 + \psi_1^{\text{stiff-ls},(\uptau)}), (0 + \psi_1^{\text{stiff-ls},(\uptau)}) \right\rangle_{L^2(\Gamma_\interior) \oplus L^2(\Gamma_\ls)} \psi_1^{\text{stiff-ls},(\uptau)} \nonumber\\
        &\qquad = \left\langle M^{\text{soft},(\uptau)}(z) (0 + \psi_1^{\text{stiff-ls},(\uptau)}), (0 + \psi_1^{\text{stiff-ls},(\uptau)}) \right\rangle_{\mathcal{E}} \frac{1}{\| \Psi_1^{\text{stiff-ls},(\uptau)} \|}.
    \end{align}
    Also we need to apply to $\| \Psi_1^{\text{stiff-ls},(\uptau)} \|$, the operator $(j \breve{\Pi})^{-1}$, then a multiplication by $\varepsilon^{-2} \mu_1^{\text{stiff-ls}, (\uptau)}$, then $\mathcal{P}^{(\uptau)}_{\text{stiff-ls}}$, and then $((j \breve{\Pi})^*)^{-1}$:
    \begin{align}
        &((j \breve{\Pi})^*)^{-1} \mathcal{P}^{(\uptau)}_{\text{stiff-ls}} \varepsilon^{-2} \mu_1^{\text{stiff-ls}, (\uptau)} (j \breve{\Pi})^{-1} \| \Psi_1^{\text{stiff-ls},(\uptau)} \| \nonumber\\
        &\qquad = \varepsilon^{-2} \mu_1^{\text{stiff-ls}, (\uptau)} ((j \breve{\Pi})^*)^{-1} \mathcal{P}^{(\uptau)}_{\text{stiff-ls}} \psi_1^{\text{stiff-ls},(\uptau)} \nonumber\\
        &\qquad = \varepsilon^{-2} \mu_1^{\text{stiff-ls}, (\uptau)} ((j \breve{\Pi})^*)^{-1} \left\langle (0 + \psi_1^{\text{stiff-ls},(\uptau)}), (0 + \psi_1^{\text{stiff-ls},(\uptau)}) \right\rangle_{L^2(\Gamma_\interior) \oplus L^2(\Gamma_\ls)} \psi_1^{\text{stiff-ls},(\uptau)} \nonumber\\
        &\qquad = \varepsilon^{-2} \mu_1^{\text{stiff-ls}, (\uptau)} ((j \breve{\Pi})^*)^{-1} \psi_1^{\text{stiff-ls},(\uptau)} \nonumber\\
        &\qquad = \varepsilon^{-2} \mu_1^{\text{stiff-ls}, (\uptau)} \frac{1}{\| \Psi_1^{\text{stiff-ls},(\uptau)} \|}.
    \end{align}
    Using these two computations, we observe that the denominator of the second term in (\ref{eqn:dispersion_relation_divisionstep}) is
    \begin{align}
        & 1 - \frac{1}{z \| \Psi_1^{\text{stiff-ls},(\uptau)} \|} T_{\varepsilon, \text{stiff-ls}}^{(\uptau)} \begin{pmatrix}
            z (A_0^{\text{soft}} -z )^{-1}w + w \\
            0 \\
            \| \Psi_1^{\text{stiff-ls}} \| 
        \end{pmatrix} \nonumber\\
        &= 1 + \frac{1}{z \| \Psi_1^{\text{stiff-ls},(\uptau)} \|^2} \left[ \varepsilon^{-2} \mu_1^{\text{stiff-ls},(\uptau)} + \left\langle M^{\text{soft},(\uptau)}(z) (0 + \psi_1^{\text{stiff-ls},(\uptau)}), (0 + \psi_1^{\text{stiff-ls},(\uptau)}) \right\rangle \right] \label{eqn:denominiator_a}\\
        &= 1 + \frac{(\text{Re }z) - i (\text{Im }z)}{|z|^2 \| \Psi_1^{\text{stiff-ls},(\uptau)} \|^2} \left[ \varepsilon^{-2} \mu_1^{\text{stiff-ls},(\uptau)} + \left\langle M^{\text{soft},(\uptau)}(z) (0 + \psi_1^{\text{stiff-ls},(\uptau)}), (0 + \psi_1^{\text{stiff-ls},(\uptau)}) \right\rangle \right]. \label{eqn:denominiator_b}
    \end{align}
    Focusing on the imaginary part of (\ref{eqn:denominiator_b}), it suffices to show that the following expression is non-zero for every $\uptau$ (Note that $|\text{Im }z|>0$, since $z \in K_\sigma$.):
    \begin{align}\label{eqn:denominiator_im}
        -i (\text{Im }z) \bigg[ \varepsilon^{-2} \mu_1^{\text{stiff-ls},(\uptau)} 
        &+ \left\langle \text{Re }M^{\text{soft},(\uptau)}(z) (0 + \psi_1^{\text{stiff-ls},(\uptau)}), (0 + \psi_1^{\text{stiff-ls},(\uptau)}) \right\rangle \nonumber\\
        &- (\text{Re }z) \left\langle S^{\text{soft},(\uptau)} (\overline{z})^* S^{\text{soft},(\uptau)}(\overline{z})  (0 + \psi_1^{\text{stiff-ls},(\uptau)}), (0 + \psi_1^{\text{stiff-ls},(\uptau)}) \right\rangle \bigg] \nonumber \\
        = -i (\text{Im }z) \left[ A + B + C \right]
    \end{align}
    where we have used the identity $\text{Im } M(z) = (\text{Im }z) S(\overline{z})^* S(\overline{z})$ (Proposition \ref{prop:auxops_properties}(7)). Recall that $\text{Re }M^{\text{soft},(\uptau)}(z)$ was defined in the beginning of the proof of Theorem \ref{thm:m_inverse_est}. The terms $B$ and $C$ are real, and independent of $\eps$. In fact, they can be bounded uniformly in $\tau$: 
    
    \begin{itemize}
        \item For $B$, use the identity
        $$M(z) = \Lambda + z \Pi^* (I - zA_0^{-1})^{-1} \Pi$$
        (Proposition \ref{prop:auxops_properties}(5)). Now apply Proposition \ref{prop:lift_properties} to $\Pi^{\soft,(\uptau)}$, and apply the arguments of Proposition \ref{prop:m_components_unif_bound} (see (\ref{eqn:m_components_unif_bound_calculation})) to $\langle \Lambda^{\text{soft},(\uptau)}  (0, \psi_1^{\text{stiff-ls}, (\uptau) }), (0, \psi_1^{\text{stiff-ls}, (\uptau)}) \rangle$.
        
        \item For $C$, apply Lemma \ref{lem:soln_op_est} to $S^{\soft,(\tau)}(z)$, followed by Proposition \ref{prop:lift_properties}.
    \end{itemize}
    If $\tau \neq 0$, Proposition \ref{prop:dtn_properties} says that $A = \eps^{-2} \mu_1^{\stls,(\tau)}$ is a negative real number. If $\tau = 0$, we have $A=0$ (Proposition \ref{prop:dtn_properties}). Nonetheless, the expression (\ref{eqn:denominiator_a}) non-zero. Indeed, we first compute
    \begin{align}
        \left\langle M^{\text{soft},(0)}(z) (0 + \psi_1^{\text{stiff-ls},(0)}), (0 + \psi_1^{\text{stiff-ls},(0)}) \right\rangle.
    \end{align}
    To do this, we first write down the BVP that $u := S^{\text{soft},(0)}(z)(0 + \psi_1^{\text{stiff-ls},(0)})$ solves:
    \begin{align}
        \begin{cases}
            -\Delta u = zu \quad &\text{in $Q_\text{soft}$}, \\
            u = 0 &\text{on $\Gamma_{\text{int}}$}, \\
            u = \psi_1^{\text{stiff-ls},(0)} = |\Gamma_\ls|^{-\frac{1}{2}}\mathbf{1}_{\Gamma_\ls} &\text{on $\Gamma_{\text{ls}}$}.
        \end{cases}
    \end{align}
    Then, we compute
    \begin{align}\label{eqn:soft_stiff_test}
        \langle M^{\text{soft},(0)}(z) &(0 + \psi_1^{\text{stiff-ls},(0)}) , (0 + \psi_1^{\text{stiff-ls},(0)}) \rangle \nonumber \\ 
        & = - \underbrace{\int_{\Gamma_{\text{int}}}...}_{=0} ~ - \int_{\Gamma_{\text{ls}}} \frac{\partial u}{ \partial n} \psi_1^{\text{stiff-ls},(0)}
        = \int_{Q_\text{soft}} z |u|^2 - \int_{Q_\text{soft}} |\nabla u|^2.
    \end{align}
    We note that $\| \nabla u \| \neq 0$, or else $u$ will be a constant function on the connected set $Q_\soft$, contradicting the fact that $u$ has different traces on $\Gamma_\interior$ and $\Gamma_\ls$. Since $\text{Im } z \neq 0$ and $\| \nabla u \| \neq 0$, (\ref{eqn:soft_stiff_test}) implies that (\ref{eqn:denominiator_a}) non-zero.
    
    Thus far, we have shown that \textit{for each $\uptau$}, we may pick $\eps$ small enough such that (\ref{eqn:denominiator_a}) is non-zero. This is not enough, as we would like to pick $\eps$ small enough independently of $\tau$. To achieve this, we will have to enhance the above argument argument as follows: Since (\ref{eqn:denominiator_a}) is continuous in $\tau$ and non-zero at $\tau=0$, it must be bounded away from zero in a neighbourhood of $\tau=0$. Furthermore, the expression (\ref{eqn:denominiator_a}) allows us to pick this neighbourhood independently of $\eps$. Now combine these facts with the arguments of the $\tau \neq 0$ case, which says that $\eps>0$ can be chosen small enough, independently of $\tau \in Q'$ outside this neighborhood, such that (\ref{eqn:denominiator_a}) is non-zero.
    
    We have therefore shown that for $\eps>0$ small enough, the mapping $\tau \mapsto K_\stin(\tau,z)$ is well-defined. This concludes the proof.
\end{proof}

\begin{defn}
    We call $K_\text{stiff-int}(\uptau,z)$ the dispersion function with respect to $Q_{\text{stiff-int}}$.
\end{defn}

\begin{rmk}
    To justify $K_\stin$ which describe wave propagation on the stiff-interior region, we have relied crucially on the properties of the stiff-landscape region, namely the eigenvalue $\mu_1^{\stls,(\tau)}$.
\end{rmk}

We conclude the section with a few important observations.

\begin{enumerate}
    \item The function $K_\stin(\tau,z)$ consists of two terms, $K_{a,\text{stiff-int}}(\uptau, z)$ and $K_{b,\text{stiff-int}}(\uptau, z)$.
    
    The first term, $K_{a,\text{stiff-int}}(\uptau, z)$, is what we would have if there were only one stiff component (See \cite[Section 5.3]{eff_behavior}, for Model II). In our case with two stiff components, we have compensated by adding the ``correction" term $K_{b,\text{stiff-int}}(\uptau, z)$.
    
    \item The dependence of $K_\stin(\tau,z)$ on $\varepsilon$ falls solely on the $T_{\text{stiff-ls}}$ terms with a non-zero third component. In particular, we observe that $\eps$ appears in the term `$A$' of (\ref{eqn:denominiator_im}), and nowhere else. So the correction term becomes small as $\eps \rightarrow 0$. To be precise, we have
    
    \begin{cor}
        If $\tau \neq 0$, then $K_{b,\text{stiff-int}}(\uptau, z) = O(\eps^2)$, uniformly in $z \in K_\sigma$. If we assume further that $\tau$ is uniformly bounded away from $0$, then the estimate is also uniform in $\tau$.
    \end{cor}
\end{enumerate}

\subsection{Stiff-landscape to stiff-landscape}

We now turn our attention to $P_{\mathbb{C}_{\text{stiff-ls}}} (\mathcal{A}_{\text{hom}} -z )^{-1} P_{\mathbb{C}_{\text{stiff-ls}}}$. We will use the functions $v$ and $w$ as defined in the previous section. We omit the analogous derivation of $K_\stls(\uptau,z)$, and define:
\begin{defn}
    The dispersion function $K_\stls(\uptau,z)$ with respect to $Q_\stls$ is given by
    \begin{alignat}{2}\label{eqn:dispersion_relation2_divisionstep}
        &K_\stls (\uptau,z) &&:= K_{a,\text{stiff-ls}}(\uptau, z) + K_{b,\text{stiff-ls}}(\uptau, z) \\
        & &&:= \frac{1}{\| \Psi_1^{\text{stiff-ls}} \|} T_{\text{stiff-ls}} \begin{pmatrix}
        z (A_0^{\text{soft}} -z )^{-1}w + w \\
        0 \\
        \| \Psi_1^{\text{stiff-ls}} \|
    \end{pmatrix} \nonumber \\
    &\qquad\qquad + &&\frac{1}{z \| \Psi_1^{\text{stiff-int}} \| \| \Psi_1^{\text{stiff-ls}} \|} T_{\text{stiff-int}} \begin{pmatrix}
        z (A_0^{\text{soft}} -z )^{-1}w + w \\
        0 \\
        \| \Psi_1^{\text{stiff-ls}} \|
    \end{pmatrix} T_{\text{stiff-ls}} \begin{pmatrix}
        z (A_0^{\text{soft}} -z )^{-1}v + v \\
        \| \Psi_1^{\text{stiff-int}} \| \\
        0
    \end{pmatrix} \times \nonumber\\
    & &&\qquad\qquad\qquad \times \left( 1 - \frac{1}{z \| \Psi_1^{\text{stiff-int}}\|} T_{\text{stiff-int}} \begin{pmatrix}
        z (A_0^{\text{soft}} -z )^{-1}v + v \\
        \| \Psi_1^{\text{stiff-int}} \| \\
        0
    \end{pmatrix} \right)^{-1}.
    \end{alignat}
\end{defn}

\begin{thm}\label{thm:stiff_ls_dispersion_reln}
    For $\eps>0$ small enough, independently of $z \in K_\sigma$ and $\tau \in Q'$, \newline $P_{\C_\stls} (\mathcal{A}_{\text{hom}} -z )^{-1} P_{\C_\stls}$ is the operator on $\mathbb{C}_{\text{stiff-ls}}$ of multiplication by the number \newline ${(K_\text{stiff-ls}(\uptau,z) - z)^{-1}}$. i.e.
    \begin{align}
        P_{\mathbb{C}_{\text{stiff-ls}}} (\mathcal{A}^{(\uptau)}_{\eps, \text{hom}} -z )^{-1} P_{\mathbb{C}_{\text{stiff-ls}}} = M_{(K_\text{stiff-ls}(\uptau,z) - z)^{-1}}.
    \end{align}
\end{thm}

\begin{proof}
    The first part of the proof proceeds analogously to the proof of Theorem \ref{thm:stiff_int_dispersion_reln}, so we omit this. In place of (\ref{eqn:denominiator_a}), we now have
    \begin{align}\label{eqn:denominiator_ls}
        1 + \frac{1}{z \| \Psi_1^{\stin,(\tau)} \|^2 } \left\langle M^{\text{soft},(\uptau)}(z) (\psi_1^{\text{stiff-int},(\uptau)} + 0), (\psi_1^{\text{stiff-int},(\uptau)} + 0) \right\rangle,
    \end{align}
    and now we would like to show that the following expression
    \begin{align}\label{eqn:denominiator_im_ls}
        \left\langle M^{\text{soft},(\uptau)}(z) (\psi_1^{\text{stiff-int},(\uptau)} + 0), (\psi_1^{\text{stiff-int},(\uptau)} + 0) \right\rangle
    \end{align}
    is a non-zero constant that does not depend on $\uptau$. The argument is a generalization of the case $\uptau = 0$ in Theorem \ref{thm:stiff_int_dispersion_reln}. To begin, we write down the BVP that $u := S^{\text{soft},(\uptau)}(z)(\psi_1^{\text{stiff-int},(\uptau)} + 0)$ solves (ignoring the normalization constant as it will not affect the arguments):
    \begin{align}
        \begin{cases}
            -(\nabla + i\uptau)^2 u = zu \quad &\text{in $Q_\text{soft}$,} \\
            u = \psi_1^{\text{stiff-int},(\uptau)} = e^{-i \uptau \cdot x} &\text{on $\Gamma_{\text{int}}$,}\\
            u = 0 &\text{on $\Gamma_{\text{ls}}$.}
        \end{cases}
    \end{align}
    Now define $w(x) = e^{i\uptau \cdot x} u(x)$. Then in $Q_{\text{soft}}$, we have that
    \begin{align}
        &-(\nabla + i\uptau)^2 u 
        = -e^{-i\uptau\cdot x} \text{div} \left( e^{i\uptau\cdot x} (\nabla + i\uptau) u \right)
        = -e^{-i\uptau\cdot x} \text{div} \left( e^{i\uptau\cdot x} (\nabla + i\uptau) (e^{-i\uptau\cdot x} w)    \right) \nonumber\\
        &\qquad = -e^{-i\uptau\cdot x} \text{div} \left( e^{i\uptau\cdot x} \left[ e^{-i\uptau\cdot x} \nabla w + \cancel{w(-i\uptau)e^{-i\uptau\cdot x}} + \cancel{(i\uptau) e^{-i\uptau\cdot x} w}   \right] \right)
        = -e^{-i\uptau\cdot x} \Delta w.
    \end{align}
    Since $e^{-i\uptau\cdot x}$ cannot be zero, we deduce that $w$ solves the BVP:
    \begin{align}\label{eqn:BVP_w_stls}
        \begin{cases}
            -\Delta w = z w \quad &\text{in $Q_\text{soft}$,} \\
            w = \mathbf{1}_{\Gamma_{\text{int}}} &\text{on $\Gamma_{\text{int}}$,}\\
            w = 0 &\text{on $\Gamma_{\text{ls}}$.}
        \end{cases}
    \end{align}
    We may now compute
    \begin{align}
        &\langle M^{\text{soft},(\uptau)}(z) (\psi_1^{\text{stiff-int},(\uptau)} + 0) , (\psi_1^{\text{stiff-int},(\uptau)} + 0) \rangle \nonumber \\
        &\qquad = - \int_{\Gamma_{\text{int}}} \left[ \frac{\partial u}{ \partial n} + i(\uptau \cdot n) u \right] \psi_1^{\text{stiff-int},(\uptau)} - \underbrace{\int_{\Gamma_{\text{ls}}}...}_{=0} 
        = \int_{Q_\text{soft}} z |u|^2 - \int_{Q_\text{soft}} |(\nabla + i\uptau) u|^2 \nonumber \\
        &\qquad = \int_{Q_\text{soft}} z | e^{-i\uptau \cdot x} w|^2 - \int_{Q_\text{soft}} | e^{-i\uptau \cdot x} \nabla w|^2
        = \int_{Q_\text{soft}} z |w|^2 - \int_{Q_\text{soft}} |\nabla w|^2. \label{eqn:soft_stiff_test2}
    \end{align}
    Since \eqref{eqn:BVP_w_stls} does not depend on $\tau$, the same is true for $w$, and thus for \eqref{eqn:soft_stiff_test2}. Since $\text{Im } z \neq 0$ and $\|\nabla w \| \neq 0$, (\ref{eqn:soft_stiff_test2}) implies that (\ref{eqn:denominiator_im_ls}) is a non-zero constant.
\end{proof}

Similarly to $K_\stin(\tau,z)$, we make a few important observations for $K_\stls(\tau,z)$.

\begin{enumerate}
    \item Again, $K_\stls(\tau,z)$ consists of two terms, $K_{a,\text{stiff-ls}}(\uptau, z)$ and $K_{b,\text{stiff-ls}}(\uptau, z)$.
    
    The first term, $K_{a,\text{stiff-ls}}(\uptau, z)$, corresponds to the dispersion function for Model I of \cite[Section 5.3]{eff_behavior} (one stiff component). In our case with two stiff components, we have the ``correction" term $K_{b,\text{stiff-ls}}(\uptau, z)$.
    
    \item $K_{a,\text{stiff-ls}}(\uptau, z)$ depends on $\varepsilon$ while $K_{b,\text{stiff-ls}}(\uptau, z)$ does not.
\end{enumerate}

\section{Acknowledgements}
YSL is supported by a scholarship from the EPSRC Centre for Doctoral Training in Statistical Applied Mathematics at Bath (SAMBa), under the project EP/S022945/1.~YSL would like to thank to Prof.~Kirill D.~Cherednichenko for his patience, guidance, and the many helpful discussions throughout this project. YSL would like to thank Dr.~Alexander V.~Kiselev, Prof.~Euan Spence, and Prof.~David Krejčiřík for their careful reading of the manuscript and helpful comments.

\appendix
\section{Preparatory results for estimating the first Steklov eigenvalue with respect to the stiff landscape}\label{appendix:aux_steklov}

The following result is a slight modification of \cite[Proposition A.7]{simplified_method}:

\begin{prop}\label{prop:coercive_est}
    There exists a constant $C>0$ such that for every $u \in H_{\text{per}}^1(Q)$, $\tau \in Q'$ we have the following estimates: 
    \begin{equation}\label{eqn:coercive_est1}
        |\tau| \left\| u \right\|_{L^2(Q)} \leq C \left\| \left(\nabla + i\tau\right) u \right\|_{L^2(Q;\C^d)},
    \end{equation}
    \begin{equation}\label{eqn:coercive_est2}
        \left\| \nabla u \right\|_{L^2(Q;\C^d)} \leq C \left\lVert \left(\nabla + i\tau\right)  u \right\|_{L^2(Q;\C^d)},
    \end{equation}
    \begin{equation}\label{eqn:coercive_est3}
        \left\| u - \int_Q u \right\|_{L^2(Q)} \leq C \left\| \left(\nabla + i\tau\right) u \right\|_{L^2(Q;\C^d)}.
    \end{equation}
\end{prop}

\begin{proof}
    Fix $u \in H_{\text{per}}^1(Q)$ and consider its Fourier series decomposition:
    \begin{align*}
        u = \sum_{k \in \Z^d} a_k e^{2\pi i k \cdot y}, \quad  
        \nabla u = \sum_{k\in \Z^d} (2 \pi i k) a_k e^{2\pi i k \cdot y}, \quad 
        u - \int_Q u= \sum_{k \in \Z^d \setminus \{0\}} a_k e^{2\pi i k \cdot y}.
    \end{align*}

    Plancherel's formula yields
    \begin{equation}
        \left\| u \right\|_{L^2(Q)}^2 = \sum_{k\in \Z^d} |a_k|^2, \quad 
        \left\| \nabla u \right\|_{L^2(Q;\C^d)}^2 = \sum_{k\in \Z^d} |2 \pi |^2|a_k |^2| k|^2, \quad \left\| u - \int_Q u\right\|_{L^2(Q)}^2 = \sum_{k\in \Z^d \setminus \{0\}} |a_k|^2.
    \end{equation}

    Furthermore, we have
    \begin{equation}
            (\nabla + i\tau) u = \sum_{k\in \Z^d} \left(2 \pi i k + i\tau \right) a_k e^{2\pi i k \cdot y}.
    \end{equation}

    Now we calculate
    \begin{equation}
        \left\| (\nabla + i\tau) u \right\|_{L^2(Q;\C^d)}^2 
        = \sum_{k\in \Z^d} |a_k \left(2 \pi i k + i\tau \right)|^2 
        = \sum_{k\in \Z^d \setminus \{0\}} |a_k |^2| 2 \pi i k + i\tau |^2 + |a_0 |^2|  \tau |^2.
    \end{equation}

    Since $\tau \in Q' = [-\pi,\pi)^d$, if at least one $(k)_j \geq 1$, it is clear that 
    \begin{equation}
        |2 \pi i k + i \tau|^2 \geq C,
    \end{equation}
    where the constant $C>0$ does not depend on $\tau$ and $k \in \Z^d\setminus \{0\}$. This gives us 
    \begin{equation}
        \begin{split}
            \left\| (\nabla + i\tau) u \right\|_{L^2(Q;\C^d)}^2  
            &= \sum_{k\in \Z^d \setminus \{0\}} |a_k |^2| 2 \pi i k + i\tau |^2 + |a_0 |^2|  \tau |^2 \\
            &\geq \sum_{k\in \Z^d \setminus \{0\}} C|a_k |^2 
            = C\left\| u - \int_Y u\right\|_{L^2(Q)}^2.
        \end{split}
    \end{equation}

    Moreover, we have  
    \begin{equation}
         \left\| (\nabla + i\tau) u \right\|_{L^2(Q;\C^d)}^2 
         =  \sum_{k\in \Z^d \setminus \{0\}} |a_k |^2| 2 \pi i k + i\tau |^2 + |a_0 |^2|  \tau |^2 
         \geq  \sum_{k\in \Z^d} C|\tau|^2|a_k |^2 
         = C|\tau|^2 \left\| u \right\|_{L^2(Q)}^2,
    \end{equation}
    and also
    \begin{equation}
        \begin{split}
            \left\| (\nabla + i\tau) u \right\|_{L^2(Q;\C^d)}^2     
            \geq \sum_{k\in \Z^d} C|a_k |^2 |2 \pi|^2 |k|^2 
            = C\left\| \nabla u\right\|_{L^2(Q;\C^d)}^2 .
        \end{split}
    \end{equation}
    This concludes the proof.
\end{proof}

We are interested in using the above inequalities to estimate the eigenvalue $\mu_1^{\stls,(\tau)}$.

\begin{cor}\label{cor:stls_evalue1}
    There exist constants $C_1, C_2>0$ independent of $\tau$ such that
    \begin{align}\label{eqn:stls_evalue1}
        C_1 |\tau|^2 \leq - \mu_1^{\stls,(\tau)} \leq C_2 |\tau|^2.
    \end{align}
\end{cor}

\begin{proof}
    Fix $u \in H_\text{per}^1(Q_\stls)$, since $\partial Q_\stls = \Gamma_\ls$ is smooth, $u$ may be extended to a function in $H_\text{per}^1(Q)$, which we will still denote as $u$, such that
    \begin{align}
        \| (\nabla + i\tau) u \|_{L^2(Q;\C^d)} \leq C \| (\nabla + i\tau) u \|_{L^2(Q_\stls;\C^d)},
    \end{align}
    where the constant $C>0$ only depends on $Q_\stls$ \cite{extension_thm}. Combining this with \eqref{eqn:coercive_est1} and \eqref{eqn:coercive_est2} gives
    \begin{align}
        |\tau| \|u\|_{H^1(Q_\stls)} \leq C |\tau| \|u\|_{H^1(Q)}
        \leq C \|(\nabla + i\tau)u \|_{L^2(Q;\C^d)}
        \leq C \|(\nabla + i\tau)u \|_{L^2(Q_\stls;\C^d)}.
    \end{align}
    By the Trace theorem, this implies that
    \begin{align}
        |\tau| \| u|_{\Gamma_\ls} \|_{L^2(\Gamma_\ls)} \leq C \| (\nabla + i\tau) u \|_{L^2(Q_\stls;\C^d)},
    \end{align}
    which in turn gives us the lower bound of \eqref{eqn:stls_evalue1} by the min-max principle. The upper bound follows by testing $\mathbf{1}_{Q_\stls}$ in the variational characterization of $\mu_1^{\stls,(\tau)}$.
\end{proof}

\section{Proofs for Section \ref{sect:identifying_homo_op}}\label{appendix:krein_calc}

\begin{proof}[Proof of Proposition \ref{prop:r_eff_estimate_soft}.]
    Consider the generalized resolvent $P_\text{soft} ( \widehat{A}^{(\uptau)}_{\varepsilon, \mathcal{P}_\perp^{(\uptau)}, \mathcal{P}^{(\uptau)} } - z )^{-1} P_\text{soft}$, which we know is $O(\varepsilon^2)$ close to $R_\varepsilon^{(\uptau)}(z) = P_\text{soft}( A_\varepsilon^{(\uptau)} - z)^{-1} P_\text{soft}$ by Theorem \ref{thm:first_asymp_result}. This can be expressed as
    \begin{align}\label{eqn:soft_pperpp_soft}
        &P_\text{soft} ( \widehat{A}^{(\uptau)}_{\varepsilon, \mathcal{P}_\perp^{(\uptau)}, \mathcal{P}^{(\uptau)} } - z )^{-1} P_\text{soft} \nonumber \\
        &= P_\text{soft} (A_{\varepsilon,0}^{(\uptau)} - z )^{-1} P_\text{soft} - P_\text{soft} 
        S_\varepsilon^{(\tau)}(z)
        \left(\overline{\mathcal{P}_\perp^{(\uptau)} + \mathcal{P}^{(\uptau)} M_\varepsilon^{(\uptau)}(z)  } \right)^{-1} \mathcal{P}^{(\uptau)}
        \left( S_\varepsilon^{(\uptau)}(\bar{z}) \right)^* P_\text{soft} \nonumber\\
        &= (A_0^{\text{soft},(\uptau)} - z)^{-1} - S^{\text{soft},(\uptau)}(z)
        \left(\overline{\mathcal{P}_\perp^{(\uptau)} + \mathcal{P}^{(\uptau)} M_\varepsilon^{(\uptau)}(z)  } \right)^{-1} \mathcal{P}^{(\uptau)}
        \left( S^{\text{soft},(\uptau)} (\bar{z}) \right)^* \nonumber\\
        &= (A_0^{\text{soft},(\uptau)} - z)^{-1} - S^{\text{soft},(\uptau)}(z)
        \mathcal{P}^{(\uptau)}
        \left( \mathcal{P}^{(\uptau)} M_\varepsilon^{\text{stiff}, (\uptau)}(z) \mathcal{P}^{(\uptau)} + \mathcal{P}^{(\uptau)} M^{\text{soft}, (\uptau)}(z) \mathcal{P}^{(\uptau)} \right)^{-1}
        \mathcal{P}^{(\uptau)}
        \left( S^{\text{soft},(\uptau)} (\bar{z}) \right)^*.
    \end{align}
    The second equality follows by the same reasoning as Proposition \ref{prop:r_eps_resolvent_alt}. The final equality uses Lemma \ref{lem:useful_b0b1_check}. On the other hand, by the Krein's formula, we have
    \begin{align}\label{eqn:r_eff_resolvent}
        R_{\varepsilon,\text{eff}}^{(\uptau)}(z) &= (A_0^{\text{soft},(\uptau)} - z)^{-1} \nonumber \\ &\quad - S^{\text{soft},(\uptau)}(z)
        \left( \overline{\mathcal{P}_\perp^{(\uptau)} + \mathcal{P}^{(\uptau)} M_\varepsilon^{\text{stiff}, (\uptau)}(z) \mathcal{P}^{(\uptau)} + \mathcal{P}^{(\uptau)} M^{\text{soft},(\uptau)}(z) } \right)^{-1}
        \mathcal{P}^{(\uptau)}
        \left( S^{\text{soft},(\uptau)} (\bar{z}) \right)^* \nonumber \\
        &= (A_0^{\text{soft},(\uptau)} - z)^{-1} \nonumber \\
        &\quad - S^{\text{soft},(\uptau)}(z)
        \mathcal{P}^{(\uptau)}
        \left( \mathcal{P}^{(\uptau)} M_\varepsilon^{\text{stiff}, (\uptau)}(z) \mathcal{P}^{(\uptau)} + \mathcal{P}^{(\uptau)} M^{\text{soft}, (\uptau)}(z) \mathcal{P}^{(\uptau)} \right)^{-1}
        \mathcal{P}^{(\uptau)}
        \left( S^{\text{soft},(\uptau)} (\bar{z}) \right)^*.
    \end{align}
    The second equality follows from the observation made before the proposition.
\end{proof}

\begin{proof}[Proof of Proposition \ref{prop:r_eff_fullspace}.]
    We will verify this entry-wise. The top left entry is done in Proposition \ref{prop:r_eff_estimate_soft}. For the remaining entries, we will compare this with $(\widehat{A}^{(\uptau)}_{\varepsilon, \mathcal{P}_\perp^{(\uptau)}, \mathcal{P}^{(\uptau)}}  - z )^{-1}$ since it is $O(\varepsilon^2)$ close to $(A_\varepsilon^{(\uptau)} - z )^{-1}$ by Theorem \ref{thm:first_asymp_result}. In the computations below, we will use
    \begin{align}
        \text{ran}\left(S^{\text{soft},(\uptau)}(z) \mathcal{P}^{(\uptau)} \right) \subset \mathcal{D}(A_0^{\text{soft},(\uptau)}) \dot{+} \text{ran}(\Pi^{\soft,(\tau)}\mathcal{P}^{(\tau)}),
    \end{align}
    which is a consequence of the identity $S(z) = (I + z(A_0 - z)^{-1}) \Pi$. The argument for $a_{21}$ and $a_{31}$ are the same. For $a_{21}$, we have
    \begin{align*}
        &P_{\text{stiff-int}} (\widehat{A}^{(\uptau)}_{\varepsilon, \mathcal{P}_\perp^{(\uptau)}, \mathcal{P}^{(\uptau)}}  - z )^{-1}  P_{\text{soft}} \\
        &= - S_\varepsilon^{\text{stiff-int}, (\uptau)}(z) \mathcal{P}^{(\uptau)} \left( \mathcal{P}^{(\uptau)} M_\varepsilon^{\text{stiff},(\uptau)}(z) \mathcal{P}^{(\uptau)} + \mathcal{P}^{(\uptau)} M^{\text{soft},(\uptau)}(z) \mathcal{P}^{(\uptau)} \right)^{-1} \mathcal{P}^{(\uptau)} \left( S^{\text{soft},(\uptau)}(\bar{z}) \right)^* \\
        &\qquad\qquad \text{ By a similar argument to (\ref{eqn:soft_pperpp_soft}).} \\
        &= - S_\varepsilon^{\text{stiff-int}, (\uptau)}(z) \Gamma_0^{\text{soft},(\uptau)} S^{\text{soft},(\uptau)}(z) \mathcal{P}^{(\uptau)} \left( \mathcal{P}^{(\uptau)} M_\varepsilon^{\text{stiff},(\uptau)}(z) \mathcal{P}^{(\uptau)} + \mathcal{P}^{(\uptau)} M^{\text{soft},(\uptau)}(z) \mathcal{P}^{(\uptau)} \right)^{-1} \mathcal{P}^{(\uptau)} \left( S^{\text{soft},(\uptau)}(\bar{z}) \right)^* \\
        &\qquad\qquad \text{ By Proposition \ref{prop:auxops_properties}(2), $\Gamma_0^{\text{soft}, (\uptau)} S^{\text{soft},(\uptau)}(z) = I$.} \\
        &= - S_\varepsilon^{\text{stiff-int}, (\uptau)}(z) k^{(\uptau)}(z) \left[ S^{\text{soft},(\uptau)}(z) \mathcal{P}^{(\uptau)} \left( \mathcal{P}^{(\uptau)} M_\varepsilon^{\text{stiff},(\uptau)}(z) \mathcal{P}^{(\uptau)} + \mathcal{P}^{(\uptau)} M^{\text{soft},(\uptau)}(z) \mathcal{P}^{(\uptau)} \right)^{-1} \mathcal{P}^{(\uptau)} \left( S^{\text{soft},(\uptau)}(\bar{z}) \right)^* \right] \\
        &= S_\varepsilon^{\text{stiff-int}, (\uptau)}(z) k^{(\uptau)}(z) \left[ R_{\varepsilon,\text{eff}}^{(\uptau)}(z) - (A_0^{\text{soft}, (\uptau)} - z)^{-1} \right] \\
        &\qquad\qquad \text{ By (\ref{eqn:r_eff_resolvent}).} \\
        &= \Pi^{\text{stiff-int},(\uptau)} k^{(\uptau)}(z) \left[ R_{\varepsilon,\text{eff}}^{(\uptau)}(z) - (A_0^{\text{soft}, (\uptau)} - z)^{-1} \right] + O(\varepsilon^2) \\
        &\qquad\qquad \text{ By Lemma \ref{lem:soln_op_est}. The remaining terms equals $\mathcal{P}^{(\uptau)}\mathbb{A}^{-1}\mathcal{P}^{(\uptau)} \left( S^{\text{soft},(\uptau)}(\bar{z}) \right)^*$, which is $O(1)$.}
    \end{align*}

    The argument for $a_{12}$ and $a_{13}$ are the same. For $a_{12}$, we have
    \begin{align*}
        &P_{\text{soft}} (\widehat{A}^{(\uptau)}_{\varepsilon, \mathcal{P}_\perp^{(\uptau)}, \mathcal{P}^{(\uptau)}}  - z )^{-1}  P_{\text{stiff-int}} \\
        &=  -S^{\text{soft},(\uptau)}(z) \mathcal{P}^{(\uptau)} \left( \mathcal{P}^{(\uptau)} M_\varepsilon^{\text{stiff},(\uptau)}(z) \mathcal{P}^{(\uptau)} + \mathcal{P}^{(\uptau)} M^{\text{soft},(\uptau)}(z) \mathcal{P}^{(\uptau)} \right)^{-1} \mathcal{P}^{(\uptau)} \left( S_\varepsilon^{\text{stiff-int},(\uptau)}(\bar{z}) \right)^* \\
        &\qquad \text{Similarly to (\ref{eqn:soft_pperpp_soft}). The decoupling term $(A_{\varepsilon,0}^{(\uptau)} - z)^{-1}$ vanishes as $L^2(Q_\text{soft})$ and } \\
        &\qquad \text{$L^2(Q_\text{stiff-int})$ are orthogonal, and are invariant subspaces for $A_{\varepsilon,0}^{(\uptau)}$.} \\
        &=\left( - \mathcal{P}^{(\uptau)} \left(\mathcal{P}^{(\uptau)} M_\varepsilon^{(\uptau)}(\bar{z}) \mathcal{P}^{(\uptau)} \right)^{-1} \mathcal{P}^{(\uptau)} \left(S^{\text{soft},(\uptau)}(z)\right)^* \right)^* \left( S_\varepsilon^{\text{stiff-int},(\uptau)}(\bar{z}) \right)^*\\
        &\qquad \text{Take adjoint twice. Use Proposition \ref{prop:auxops_properties}(5).} \\
        &=\left( k^{(\uptau)}(\bar{z}) S^{\text{soft},(\uptau)}(\bar{z}) \mathcal{P}^{(\uptau)} \left(\mathcal{P}^{(\uptau)} M_\varepsilon^{(\uptau)}(\bar{z}) \mathcal{P}^{(\uptau)} \right)^{-1} \mathcal{P}^{(\uptau)} \left(S^{\text{soft},(\uptau)}(z)\right)^* \right)^* \left( S_\varepsilon^{\text{stiff-int},(\uptau)}(\bar{z}) \right)^*\\
        &\qquad \text{ By Proposition \ref{prop:auxops_properties}(2), $\Gamma_0^{\text{soft}, (\uptau)} S^{\text{soft},(\uptau)}(\bar{z}) = I$, and $\Gamma_0^{\text{soft}, (\uptau)}$ may be restricted to $k^{(\uptau)}(\bar{z})$.} \\
        &=\left( k^{(\uptau)}(\bar{z}) \left[ R_{\varepsilon,\text{eff}}^{(\uptau)}(\bar{z}) - (A_0^{\text{soft},(\uptau)} - \bar{z})^{-1} \right] \right)^* \left( S_\varepsilon^{\text{stiff-int},(\uptau)}(\bar{z}) \right)^* \\
        &= \left( k^{(\uptau)}(\bar{z}) \left[ R_{\varepsilon,\text{eff}}^{(\uptau)}(\bar{z}) - (A_0^{\text{soft},(\uptau)} - \bar{z})^{-1} \right] \right)^* \left( \Pi^{\text{stiff-int},(\uptau)} \right)^* + O(\varepsilon^2)
    \end{align*}
    The second last equality follows by (\ref{eqn:r_eff_resolvent}). The final equality follows by Lemma \ref{lem:soln_op_est}.

    The argument for $a_{22}$ is the same as $a_{33}$ are the same. For $a_{22}$, we have
    \begin{align*}
        &P_{\text{stiff-int}} (\widehat{A}^{(\uptau)}_{\varepsilon, \mathcal{P}_\perp^{(\uptau)}, \mathcal{P}^{(\uptau)}}  - z )^{-1}  P_{\text{stiff-int}} \\
        &=(A_{\varepsilon,0}^{\text{stiff-int}, (\uptau)} - z )^{-1} - S_\varepsilon^{\text{stiff-int},(\uptau)}(z) \mathcal{P}^{(\uptau)} \left(\mathcal{P}^{(\uptau)} M_\varepsilon^{(\uptau)}(z) \mathcal{P}^{(\uptau)} \right)^{-1} \mathcal{P}^{(\uptau)} \left( S_\varepsilon^{\text{stiff-int},(\uptau)}(\bar{z}) \right)^* \\
        &\qquad \text{Similarly to (\ref{eqn:soft_pperpp_soft}).} \\
        &=(A_{\varepsilon,0}^{\text{stiff-int}, (\uptau)} - z )^{-1} - S_\varepsilon^{\text{stiff-int},(\uptau)}(z) k^{(\uptau)}(z) 
        S^{\text{soft},(\uptau)}(z) \mathcal{P}^{(\uptau)} \left(\mathcal{P}^{(\uptau)} M_\varepsilon^{(\uptau)}(z) \mathcal{P}^{(\uptau)} \right)^{-1} \mathcal{P}^{(\uptau)} \left( S_\varepsilon^{\text{stiff-int},(\uptau)}(\bar{z}) \right)^* \\
        &\qquad \text{By Proposition \ref{prop:auxops_properties}(2), $\Gamma_0^{\text{soft}, (\uptau)} S^{\text{soft},(\uptau)}(z) = I$, and $\Gamma_0^{\text{soft}, (\uptau)}$ may be restricted to $k^{(\uptau)}(z)$.} \\
        &= (A_{\varepsilon,0}^{\text{stiff-int}, (\uptau)} - z )^{-1} + 
        S_\varepsilon^{\text{stiff-int},(\uptau)}(z) k^{(\uptau)}(z)
        \left( k^{(\uptau)}(\bar{z}) \left[  R_{\varepsilon,\text{eff}}^{(\uptau)}(\bar{z}) - (A_0^{\text{soft},(\uptau)} - \bar{z})^{-1} \right] \right)^* \left( S_\varepsilon^{\text{stiff-int},(\uptau)}(\bar{z}) \right)^* \\
        &\qquad \text{By the arguments of $a_{12}$.} \\
        &= \Pi^{\text{stiff-int},(\uptau)}(z) k^{(\uptau)}(z)
        \left( k^{(\uptau)}(\bar{z}) \left[  R_{\varepsilon,\text{eff}}^{(\uptau)}(\bar{z}) - (A_0^{\text{soft},(\uptau)} - \bar{z})^{-1} \right] \right)^* \left( \Pi^{\text{stiff-int},(\uptau)}(\bar{z}) \right)^* + O(\varepsilon^2) \\
        &\qquad \text{By Lemma \ref{lem:soln_op_est} and Proposition \ref{prop:decoupling_properties}.}
    \end{align*}
    
    Finally, the argument for $a_{23}$ and $a_{32}$ is similar to that of $a_{22}$, the only difference being that the decoupling term $(A_0 - z)^{-1}$ is now absent. This completes the proof.
\end{proof}

The following proof is an extension of the arguments provided in \cite[Theorem 4.4]{eff_behavior}:

\begin{proof}[Proof of Theorem \ref{thm:self_adjoint}.]
    We will start with the \textbf{top left entry} of (\ref{eqn:ahom_block_matrix}). To qualify as a resolvent of $\mathcal{A}_{\text{hom}}$ at $z$, the operator on the top left entry must take any given $f \in L^2(Q_\text{soft})$ to $u$, where $u$ is the first entry of $(u, \widehat{u})^T \in \mathcal{D}(\mathcal{A}_\text{hom})$, and $(u, \widehat{u})$ is the unique solution to the problem
    \begin{align}\label{eqn:ahom_system0}
        \left( \mathcal{A}_\text{hom} - z \right) \begin{pmatrix}
            u \\ \widehat{u}
        \end{pmatrix} = \begin{pmatrix}
            f \\ 0
        \end{pmatrix} \Longleftrightarrow \begin{cases}
            \breve{\widehat{A}^{\text{soft}}} u - zu = f, \\
            -(\breve{\Pi}^{\text{stiff }*})^{-1} \breve{\Gamma}_1^{\text{soft}} u + \mathcal{B} \widehat{u} - z \widehat{u} = 0.
        \end{cases}
    \end{align}
    But we may rearrange the second line of the latter system: 
    \begin{align}\label{eqn:ahom_system}
        &\quad (\breve{\Pi}^{\text{stiff }*})^{-1} \breve{\Gamma}_1^{\text{soft}} u - (\mathcal{B} - z) \widehat{u} = 0 \nonumber\\
        &\Leftrightarrow  \breve{\Gamma}_1^{\text{soft}} u - (\breve{\Pi}^{\text{stiff}})^* (\mathcal{B} - z) \widehat{u} = 0 \nonumber\\
        &\Leftrightarrow \underbrace{I}_{\beta_1 := } \breve{\Gamma}_1^{\text{soft}} u \underbrace{ - (\breve{\Pi}^{\text{stiff}})^* (\mathcal{B} - z) \breve{\Pi}^{\text{stiff}}}_{\beta_0 := } \breve{\Gamma}_0^{\text{soft}} u = 0.
    \end{align}
    That is, the mapping $f \mapsto u$ as described above, is precisely that of $(\breve{\widehat{A}^{\text{soft}}_{\beta_0, \beta_1}} - z)^{-1} \equiv R(z)$, \textit{provided it exist}. This means to check the conditions on $\beta_0 + \beta_1 \breve{M}^{\text{soft}}(z)$ so that Theorem \ref{thm:kreins_formula} applies: The condition on the domains of $\beta_0$ and $\beta_1$ are immediate as these are bounded operators. The boundedness also implies that $\beta_0 + \beta_1 \breve{M}^{\text{soft}}(z)$ (with its maximal domain) is closed. 
    
    As for the boundedness of the inverse, it suffices to check that it is bounded below, since we are working on a finite dimensional space $\breve{\mathcal{H}}^{\text{stiff}}$. Just like in Theorem \ref{thm:m_inverse_est}, it suffices to show that the imaginary part is bounded below: Let $\phi \in \breve{\mathcal{E}}$, then
    \begin{alignat}{2}
        &\left| \left( \text{Im}\left[- (\breve{\Pi}^{\text{stiff}})^* (\mathcal{B} - z) \breve{\Pi}^{\text{stiff}} + \breve{M}^{\text{soft}}(z) \right] \phi, \phi \right)_{\breve{\mathcal{E}}}  \right| \nonumber\\
        &= \left| \left( \text{Im}\left[z (\breve{\Pi}^{\text{stiff}})^* \breve{\Pi}^{\text{stiff}} + \breve{M}^{\text{soft}}(z) \right] \phi, \phi \right)_{\breve{\mathcal{E}}}  \right| 
        &&\text{as $(\breve{\Pi}^{\text{stiff}})^* \mathcal{B} \breve{\Pi}^{\text{stiff}}$ is self-adjoint.} \nonumber\\
        &= \left| \text{Im}z \left( (\breve{S}^{\text{soft}}(\bar{z}))^* \breve{S}^{\text{soft}}(\bar{z}) \phi + (\breve{\Pi}^{\text{stiff}})^* \breve{\Pi}^{\text{stiff}} \phi, \phi  \right)_{\breve{\mathcal{E}}} \right| 
        &&\text{by Proposition \ref{prop:auxops_properties}(7). } \nonumber\\
        &= |\text{Im} z| \left( (\breve{S}^{\text{soft}}(\bar{z}))^* \breve{S}^{\text{soft}}(\bar{z}) \phi, \phi \right)_{\breve{\mathcal{E}}} + |\text{Im} z| \left( (\breve{\Pi}^{\text{stiff}})^* \breve{\Pi}^{\text{stiff}} \phi, \phi \right)_{\breve{\mathcal{E}}} \quad
        &&\text{both operators are positive.} \nonumber\\
        &\geq |\text{Im} z| \text{ } \| \breve{\Pi}^{\text{stiff}} \phi \|_{\breve{\mathcal{H}}^{\text{stiff}}}^2 
        &&\text{both operators are positive.} \nonumber \\
        &\geq |\text{Im} z| \text{ } \| (\breve{\Pi}^{\text{stiff}})^{-1} \|^{-2}_{\breve{\mathcal{H}}^{\text{stiff}} \rightarrow \breve{\mathcal{E}} } \| \phi \|^2_{\breve{\mathcal{E}}}.
    \end{alignat}
    We have shown that $R(z)$ exists, and equals $P_{\text{soft}} (\mathcal{A}_\text{hom} - z)^{-1} P_\text{soft}$. Next, we check the \textbf{bottom left entry} of (\ref{eqn:ahom_block_matrix}). This is the mapping $f \mapsto \widehat{u}$, where $(u,\widehat{u})$ solves the system (\ref{eqn:ahom_system}). But we defined $\widehat{u} = \breve{\Pi}^{\text{stiff}} \breve{\Gamma}_0^{\text{soft}} u$ and have just shown that $u = R(z)f$, therefore
    \begin{align*}
        \widehat{u} &= \breve{\Pi}^{\text{stiff}} \breve{\Gamma}_0^{\text{soft}} R(z) f \\
        &= \breve{\Pi}^{\text{stiff}} \breve{\Gamma}_0^{\text{soft}} \left[R(z) - (A_0^{\text{soft}} - z)^{-1} \right] f
        \quad \text{as $\mathcal{D}(A_0^{\text{soft}}) = \text{ker}(\breve{\Gamma}_0^{\text{soft}})$ by definition.} \\
        &= \breve{\Pi}^{\text{stiff}} k(z) \left[R(z) - (A_0^{\text{soft}} - z)^{-1} \right] f
    \end{align*}
    
    The final equality holds by exactly the same argument as Proposition \ref{prop:r_eff_fullspace} (the term $a_{21}$), since the Krein's formula is now applicable to $R(z)$. We have shown that $P_{\breve{\mathcal{H}}^{\text{stiff}}} (\mathcal{A}_\text{hom} - z)^{-1} P_\text{soft}$ equals the bottom left entry of (\ref{eqn:ahom_block_matrix}).

    Next, we discuss the \textbf{top right entry} of (\ref{eqn:ahom_block_matrix}). Similarly to $R(z)$, this must take any given $\widehat{f} \in \breve{\mathcal{H}}^{\text{stiff}}$ to $u$, where $u$ is the first entry of $(u,\widehat{u})^T \in \mathcal{D}(\mathcal{A}_\text{hom})$, and $(u,\widehat{u})$ is the unique solution to the problem
    \begin{align}\label{eqn:ahom_system2}
        \left( \mathcal{A}_\text{hom} - z \right) \begin{pmatrix}
            u \\ \widehat{u}
        \end{pmatrix} = \begin{pmatrix}
            0 \\ \widehat{f}
        \end{pmatrix} \Longleftrightarrow \begin{cases}
            \breve{\widehat{A}^{\text{soft}}} u - zu = 0, \\
             \underbrace{I}_{\beta_1 = } \breve{\Gamma}_1^{\text{soft}} u = \underbrace{ (\breve{\Pi}^{\text{stiff}})^* (\mathcal{B} - z) \breve{\Pi}^{\text{stiff}}}_{-\beta_0 = } \breve{\Gamma}_0^{\text{soft}} u - (\breve{\Pi}^{\text{stiff}})^* \widehat{f}.
        \end{cases}
    \end{align}
    We would like to put the system into the form (\ref{eqn:abstract_system_b0b1}), forcing us address the term $(\breve{\Pi}^{\text{stiff}})^* \widehat{f}$. Using Proposition \ref{prop:trunc_aux_properties}(5) (``furthermore" part), we find some $v_{\widehat{f}} \in \mathcal{D}(A_0^{\text{soft}})$ satisfying 
    \begin{align*}
        \breve{\Gamma}_1^{\text{soft}} v_{\widehat{f}} = (\breve{\Pi}^{\text{stiff}})^* \widehat{f}.
    \end{align*}
    Now consider the function $v = R(z)(A_0^{\text{soft}} - z) v_{\widehat{f}}$. By applying Theorem \ref{thm:kreins_formula} to $R(z)$, we know that $v \in \mathcal{D}(\breve{\widehat{A}^{\text{soft}, (\uptau)}})$, and furthermore $v$ solves the following system uniquely:
    \begin{align*}
        \begin{cases}
            ( \breve{\widehat{A}^{\text{soft}}} - z ) v = (A_0^{\text{soft}} - z) v_{\widehat{f}}, \\
            \beta_0 \breve{\Gamma}_0^{\text{soft}} v + \beta_1 \breve{\Gamma}_1^{\text{soft}} v = 0.
        \end{cases}
    \end{align*}
    Using the first line of the system, $v_{\widehat{f}} \in \mathcal{D}(A_0^{\text{soft}})$, and $A_0^{\text{soft}} \subset \breve{\widehat{A}^{\text{soft}}}$,
    \begin{align*}
        (\breve{\widehat{A}^{\text{soft}}} - z) (v - v_{\widehat{f}}) = (A_0^{\text{soft}} - z) v_{\widehat{f}} - (\breve{\widehat{A}^{\text{soft}}} - z) v_{\widehat{f}} = 0.
    \end{align*}
    Using the second line of the system, $v_{\widehat{f}} \in \mathcal{D}(A_0^{\text{soft}})$, and $\mathcal{D}(A_0^{\text{soft}}) = \text{ker}(\breve{\Gamma}_0^{\text{soft}})$,
    \begin{align}
        \beta_0 \breve{\Gamma}_0^{\text{soft}} (v - v_{\widehat{f}} ) + \breve{\Gamma}_1^{\text{soft}} (v - v_{\widehat{f}} ) = - \beta_0 \breve{\Gamma}_0^{\text{soft}} v_{\widehat{f}} - \breve{\Gamma}_1^{\text{soft}} v_{\widehat{f}} = - \breve{\Gamma}_1^{\text{soft}} v_{\widehat{f}} = - (\breve{\Pi}^{\text{stiff}})^* \widehat{f}.
    \end{align}
    In other words, if we define $u \in \mathcal{D}(\breve{\widehat{A}^{\text{soft}, (\uptau)}})$ by $u := v - v_{\widehat{f}} = R(z)(A_0^{\text{soft}} - z) v_{\widehat{f}} - v_{\widehat{f}}$, then $u$ is a solution to the system (\ref{eqn:ahom_system2}). To show that the solution is unique, we consider the fully homogeneous case, $f=0$ and $\widehat{f}=0$. But this can be viewed as a special case of (\ref{eqn:ahom_system0}), for which uniqueness has been established by Theorem \ref{thm:kreins_formula}. We further compute
    \begin{alignat}{2}\label{eqn:ahom_btmright}
        u &= \left[ R(z) - (A_0^{\text{soft}} - z)^{-1} \right] (A_0^{\text{soft}} - z) v_{\widehat{f}} \nonumber\\
        &= - S^{\text{soft}}(z) \mathcal{P}^{(\uptau)} \left( \beta_0 + \breve{M}^{\text{soft}}(z) \right)^{-1} \mathcal{P}^{(\uptau)} \left( S^{\text{soft}}(\bar{z}) \right)^* (A_0^{\text{soft}} - z) v_{\widehat{f}}
        \quad &&\text{By Krein's formula.} \nonumber\\
        &= - S^{\text{soft}}(z) \mathcal{P}^{(\uptau)} \left( \beta_0 + \breve{M}^{\text{soft}}(z) \right)^{-1} \mathcal{P}^{(\uptau)} \breve{\Gamma}_1^{\text{soft}} v_{\widehat{f}}
        &&\text{By Proposition \ref{prop:auxops_properties}(4)} \nonumber\\
        &= - S^{\text{soft}}(z) \mathcal{P}^{(\uptau)} \left( \beta_0 + \breve{M}^{\text{soft}}(z) \right)^{-1} \mathcal{P}^{(\uptau)} (\breve{\Pi}^{\text{stiff}})^* \widehat{f}
        &&\text{Definition of $v_{\widehat{f}}$.} \nonumber\\
        &= \left( k(\bar{z}) \left[ R(\bar{z}) - (A_0^{\text{soft}} - \bar{z})^{-1} \right]  \right)^* (\breve{\Pi}^{\text{stiff}})^* \widehat{f}
    \end{alignat}
    Where the last equality is proven in the same way as in Proposition \ref{prop:r_eff_fullspace}(the term $a_{12}$). Therefore the expression for $u$ does not depend on the choice $v_{\widehat{f}}$, and coincides with the top right entry of (\ref{eqn:ahom_block_matrix}).

    Next, we discuss the \textbf{bottom right entry} of (\ref{eqn:ahom_block_matrix}). Similarly to the bottom left entry, we use $\widehat{u} = \breve{\Pi}^{\text{stiff}} \breve{\Gamma}_0^{\text{soft}} u$, and that $u$ is given by (\ref{eqn:ahom_btmright}) to see that 
    \begin{align*}
        \widehat{u} &= \breve{\Pi}^{\text{stiff}} \breve{\Gamma}_0^{\text{soft}} \left( k(\bar{z}) \left[ R(\bar{z}) - (A_0^{\text{soft}} - \bar{z})^{-1} \right]  \right)^* (\breve{\Pi}^{\text{stiff}})^* \widehat{f} \\
        &= \breve{\Pi}^{\text{stiff}} k(z) \left( k(\bar{z}) \left[ R(\bar{z}) - (A_0^{\text{soft}} - \bar{z})^{-1} \right]  \right)^* (\breve{\Pi}^{\text{stiff}})^* \widehat{f}
        \quad \text{as $\mathcal{D}(A_0^{\text{soft}}) = \text{ker}(\breve{\Gamma}_0^{\text{soft}})$.}
    \end{align*}

    Finally the show the \textbf{self-adjointness of $\mathcal{A}_{\text{hom}}$.} Notice that the arguments provided above implies that $(\mathcal{A}_{\text{hom}} - z)$ is surjective for all $z \in \mathbb{C} \setminus \mathbb{R}$. As $\mathcal{A}_{\text{hom}}$ is symmetric by Lemma \ref{lem:ahom_symmetric}, the conclusion follows from \cite[Proposition 3.11]{konrad_book}.
\end{proof}



\renewcommand{\bibname}{References} 
\printbibliography

@article{eff_behavior,
    author={K. D. Cherednichenko and Y. Y. Ershova and A. V. Kiselev},
    title={Effective Behaviour of Critical-Contrast {PDEs}: Micro-resonances, Frequency Conversion, and Time Dispersive Properties. {I}},
    journal={Communications in Mathematical Physics.},
    year={2020},
    %month={May},
    day={01},
    volume={375},
    number={3},
    pages={1833--1884},
    issn={1432-0916},
    %doi={10.1007/s00220-020-03696-2},
    %url={https://doi.org/10.1007/s00220-020-03696-2}
}

@article{split_second,
    author={K. D. Cherednichenko and A. V. Kiselev and L. O. Silva},
    title={Operator-norm resolvent asymptotic analysis of continuous media with low-index inclusions.},
    year={2020},
    eprint={2010.13318},
    archivePrefix={arXiv},
    primaryClass={math.AP}
}

@article{time_dispersive,
    author={K. D. Cherednichenko and Y. Y. Ershova and A. V. Kiselev},
    title={Time-Dispersive Behavior as a Feature of Critical-Contrast Media},
    journal={SIAM Journal on Applied Mathematics.},
    year={2019},
%    month={Apr},
    day={16},
    volume={79},
    number={2},
    pages={690--715}
}

@article{birman_suslina_2004,
  author  = {M. Sh. Birman and T. A. Suslina}, 
  title   = {Second order periodic differential operators. Threshold properties and homogenisation.},
  journal = {St. Petersburg. Math. J.},
  year    = 2004,
  number  = 5,
  pages   = {639--714},
  volume  = 15
}

@article{suslina_2013_dirichlet, 
  author  = {T. A. Suslina}, 
  title   = {Homogenisation of the Dirichlet problem for elliptic systems: $L^2$-operator error estimates.},
  journal = {Mathematika.},
  year    = 2013,
  number  = 2,
  pages   = {463–-476},
  volume  = 59
}

@article{suslina_2013_neumann,
  author  = {T. A. Suslina},
  title   = {Homogenization of the Neumann problem for elliptic systems with periodic coefficients.},
  journal = {SIAM Journal on Mathematical Analysis.},
  year    = 2013,
  number  = 6,
  pages   = {3453--3493},
  volume  = 45
}

@article{suslina_2018_perforated,
  author  = {T. A. Suslina},
  title   = {Spectral Approach to Homogenization of Elliptic Operators in a Perforated Space.},
  journal = {Ludwig Faddeev Memorial Volume.},
  year    = 2018,
  pages   = {481--537},
}

@article{zhikov_pastukhova_2005_operator, 
  author  = {V. V. Zhikov and S. E. Pastukhova}, 
  title   = {On operator estimates for some problems in homogenization theory.},
  journal = {Russian Journal of Mathematical Physics.},
  year    = 2005,
  number  = 4,
  pages   = {515--524},
  volume  = 12
}

@article{zhikov_pastukhova_2016_opsurvey, 
  author  = {V. V. Zhikov and S. E. Pastukhova}, 
  title   = {Operator estimates in homogenization theory.},
  journal = {Russian Mathematical Surveys.},
  year    = 2016,
  number  = 3,
  pages   = {417-511},
  volume  = 71
}

@article{griso_2004,
    author={G. Griso},
    title={Error estimate and unfolding for periodic homogenization},
    journal={Asymptotic Analysis.},
    year={2004},
    volume={40},
    number={3--4},
    pages={269--286}
}

@article{griso_2006,
    author={G. Griso},
    title={Interior error estimate for periodic homogenization},
    journal={Analysis and Applications.},
    year={2006},
    volume={4},
    number={1},
    pages={61--79}
}

@article{kenig_lin_shen_2012,
    author={C. E. Kenig and F. Lin and Z. Shen},
    title={Convergence rates in $L^2$ for elliptic homogenization},
    journal={Archive for Rational Mechanics and Analysis.},
    year={2012},
    volume={203},
    pages={1009--1036}
}

@article{cooper_waurick_2019,
    author = {S. Cooper and M. Waurick},
    title = {Fibre homogenisation},
    journal = {Journal of Functional Analysis},
    volume = {276},
    number = {11},
    pages = {3363--3405},
    year = {2019}
}

@article{conca_vanninathan1997,
    author = {C. Conca and M. Vanninathan},
    title = {Homogenization of Periodic Structures via Bloch Decomposition},
    journal = {SIAM Journal on Applied Mathematics.},
    volume = {57},
    number = {6},
    pages = {1639--1659},
    year = {1997}
}

@article{zhikov1989,
    author = {V. V. Zhikov},
    title = {Spectral approach to asymptotic problems in diffusion},
    journal = {Differential Equations},
    volume = {25},
    number = {1},
    pages = {33--39},
    year = {1989}
}

@article{arbogast_douglas_hornung_1990,
    author = {T. Arbogast and J. Douglas and U. Hornung},
    title = {Derivation of the double porosity model of single phase ﬂow via homogenization theory},
    journal = {SIAM Journal on Mathematical Analysis},
    volume = {21},
    number = {4},
    pages = {823--836},
    year = {1990}
}

@article{zhikov2000,
    author = {V. V. Zhikov},
    title = {On an extension of the method of two-scale convergence and its applications},
    journal = {Sbornik: Mathematics.},
    volume = {191},
    number = {7},
    pages = {973--1014},
    year = {2000}
}

@article{zhikov2005,
    author = {V. V. Zhikov},
    title = {On spectrum gaps of some divergent elliptic operators with periodic coefficients},
    journal = {St. Petersburg. Math. J.},
    volume = {16},
    number = {5},
    pages = {773--790},
    year = {2005}
}

@article{ryzhov2009,
  author        =   {V. Ryzhov}, 
  title         =   {Spectral Boundary Value Problems and their Linear Operators.},
  year          =   {2009},
  eprint        =   {0904.0276v1},
  archivePrefix =   {arXiv},
  primaryClass  =   {math-ph}
}

@article{kochubei_bdrytriples,
    author={A. N. Kochubei},
    title={Extensions of symmetric operators and symmetric binary relations},
    journal={Mathematical notes of the Academy of Sciences of the USSR.},
    year={1975},
    %month={Jan},
    volume={17},
    %number={1},
    pages={25--28},
    %doi={10.1007/BF01093837}
}

@article{bruk_bdrytriples,
    author={V. M. Bruk},
    title={On a class of boundary value problems with a spectral parameter in the boundary condition.},
    journal={Mathematics of the USSR-Sbornik},
    year={1976},
    volume={29},
    number={2},
    pages={186--192}
}

@article{ardent_mazzeo,
    author = {W. Arendt and R. Mazzeo},
    title = {Spectral properties of the Dirichlet-to-Neumann operator on Lipschitz domains},
    year = {2007},
    volume = 12,
    %number = {}
    pages = {23--37},
    journal = {Ulmer Seminare.}
    % This is a seminar series from ulm university. There is no issue number.
}

@article{ardent_elst,
    author  = {W. Arendt and A.F.M. {ter Elst}},
    title   = {The Dirichlet-to-Neumann operator on rough domains},
    journal = {Journal of Differential Equations},
    volume  = {251},
    number  = {8},
    pages   = {2100--2124},
    year    = {2011},
    %issn    = {0022-0396},
    %doi     = {https://doi.org/10.1016/j.jde.2011.06.017},
    %url     = {https://www.sciencedirect.com/science/article/pii/S0022039611002488},
}

@online{muthukumar,
  author = {T. Muthukumar},
  title = {Bloch-Floquet Transform},
  year = 2014,
  url = {http://home.iitk.ac.in/~tmk/courses/minicourse/FAPDE/Bloch.pdf},
  urldate = {2022-02-09}
}

@article{friedlander2002,
    author={L. Friedlander},
    title={On the density of states of periodic media in the large coupling limit},
    journal={Communications in Partial Differential Equations.},
    year={2002},
    volume={27},
    number={1-2},
    pages={355--380},
    %doi={10.1081/PDE-120002790},
    %url={http://dx.doi.org/10.1081/PDE-120002790}
}

@article{hempellienau,
    author = {R. Hempel and K. Lienau},
    title = {Spectral properties of periodic media in the large coupling limit},
    journal = {Communications in Partial Differential Equations.},
    volume = {25},
    number = {7-8},
    pages = {1445-1470},
    year  = {1999},
}

@book{cioranescu_donato,
  author    = {D. Cioranescu and P. Donato}, 
  title     = {An Introduction to Homogenization},
  publisher = {Oxford University Press},
  year      = 1999,
  series    = {Oxford Lecture Series in Mathematics and Its Applications}
}

@book{zhikov,
  author    = {V. V. Zhikov and S. M. Kozlov and O. A. Oleinik}, 
  title     = {Homogenization of Differential Operators and Integral Functionals},
  publisher = {Springer-Verlag, Berlin},
  year      = 1994,
  isbn      = {978-3-642-84661-8}
}

@book{bensoussan_lions_papanicolaou,
  author    = {A. Bensoussan and J.-L. Lions and G. Papanicolaou}, 
  title     = {Asymptotic Analysis for Periodic Structures},
  publisher = {North-Holland Publishing Co., Amsterdam--New York},
  year      = 1978
}

@book{bakhvalov_panasenko,
  author    = {N. Bakhvalov and G. Panasenko}, 
  title     = {Homogenisation: Averaging Processes in Periodic Media},
  publisher = {Springer Dordrecht},
  year      = 1989,
  series    = {Mathematics and its Applications (Soviet series)},
  isbn      = {978-94-010-7506-0}
}

@book{david_borthwick,
  author    = {D. Borthwick}, 
  title     = {Spectral Theory},
  publisher = {Springer Nature Switzerland AG},
  year      = 2020,
  series    = {Graduate Texts in Mathematics},
  isbn      = {978-3-030-38001-4} 
}

@book{konrad_book,
  author    = {K. Schmüdgen}, 
  title     = {Unbounded Self-adjoint Operators on Hilbert Space},
  publisher = {Springer, Dordrecht},
  year      = 2012,
  series    = {Graduate Texts in Mathematics},
  isbn      = {978-94-007-9741-3} 
}

@book{tretter,
  author    = {C. Tretter}, 
  title     = {Spectral Theory of Block Operator Matrices and Applications},
  publisher = {Imperial College Press, London},
  year      = 2008,
  isbn      = {978-1-86094-768-1}
}

@book{strongly_elliptic_sys,
  author    = {W. McLean}, 
  title     = {Strongly Elliptic Systems and Boundary Integral Equations},
  publisher = {Cambridge University Press, Cambridge},
  year      = 2000,
  isbn      = {0-521-66332-6}
}

@book{gilbarg_trudinger,
  author    = {D. Gilbarg and N. S. Trudinger}, 
  title     = {Elliptic Partial Differential Equations of Second Order},
  publisher = {Springer-Verlag, Berlin Heidelberg},
  year      = {2001},
  series    = {Classics in Mathematics},
  edition   = {2nd edn.},
  isbn      = {978-3-642-61798-0}
}

@book{engel_nagel_short,
  author    = {K-J. Engel and R. Nagel}, 
  title     = {A Short Course on Operator Semigroups},
  publisher = {Springer, New York, NY},
  year      = {2006},
  series    = {Universitext},
  isbn      = {978-0-387-31341-2}
}

@book{gorbachuk_gorbachuk_book,
  author    = {V. I. Gorbachuk and M. L. Gorbachuk}, 
  title     = {Boundary Value Problems for Operator Differential Equations},
  publisher = {Springer, Dordrecht},
  year      = 1991,
  volume    = 48,
  series    = {Mathematics and its Applications (Soviet series)},
  isbn      = {978-0-7923-0381-7} 
}

@book{behrndt_hassi_desnoo_book,
  author    = {J. Behrndt and S. Hassi and H. de Snoo}, 
  title     = {Boundary Value Problems, Weyl Functions, and Differential Operators},
  publisher = {Birkhäuser Cham},
  year      = 2020,
  volume    = 108,
  series    = {Monographs in Mathematics},
  isbn      = {978-3-030-36716-9} 
}

@book{ambrosio_tilli_book,
  author    = {L. Ambrosio and P. Tilli}, 
  title     = {Topics on Analysis in Metric Spaces},
  publisher = {Oxford University Press},
  year      = 2004,
  series    = {Oxford Lecture Series in Mathematics and Its Applications}
}

@book{shen_book,
  author    = {Z. Shen}, 
  title     = {Periodic Homogenization of Elliptic Systems},
  publisher = {Birkhäuser, Cham},
  year      = 2018,
  series    = {Operator Theory: Advances and Applications},
  isbn      = {978-3-319-91213-4} 
}

@book{qshomo_book,
  author    = {S. Armstrong and T. Kuusi and J.-C. Mourrat}, 
  title     = {Quantitative Stochastic Homogenization and Large-Scale Regularity},
  publisher = {Springer Cham},
  year      = 2019,
  series    = {Grundlehren der mathematischen Wissenschaften},
  isbn      = {978-3-030-15544-5}
}

@book{kato_book,
  author    = {T. Kato}, 
  title     = {Perturbation Theory for Linear Operators},
  publisher = {Springer Berlin, Heidelberg},
  year      = 1995,
  volume    = 132,
  edition   = 2,
  series    = {Grundlehren der mathematischen Wissenschaften},
  isbn      = {978-3-540-58661-6}
}

@book{reed_simon1,
  author    = {M. Reed and B. Simon}, 
  title     = {Methods of Modern Mathematical Physics \text{I}: Functional analysis},
  publisher = {Academic Press, New York},
  year      = {1980}
}

@book{reed_simon4,
  author    = {M. Reed and B. Simon}, 
  title     = {Methods of Modern Mathematical Physics \text{IV}: Analysis of Operators},
  publisher = {Academic Press, New York},
  year      = {1978},
}

@thesis{ys_thesis,
  author       = {Y-S. Lim}, 
  title        = {Propagation and dispersion of waves in composite media with resonant inclusions.},
  type         = {Ph.D.\,thesis},
  school       = {University of Bath},
  year         = 2023
}

@article{simplified_method,
  author        =   {Y-S. Lim and J. Žubrinić}, 
  title         =   {An operator-asymptotic approach to periodic homogenization for equations of linearized elasticity.},
  year          =   {2023},
  eprint        =   {2308.00594v3},
  archivePrefix =   {arXiv},
  primaryClass  =   {math.AP}
}

@article{extension_thm,
  author  = {E. Acerbi and V. {Chiadò Piat} and G. {Dal Maso} and D. Percivale},
  title   = {An extension theorem from connected sets, and homogenization in general periodic domains.},
  journal = {Nonlinear Analysis: Theory, Methods \& Applications.},
  year    = 1992,
  number  = 5,
  pages   = {481--496},
  volume  = 18
}
\addcontentsline{toc}{section}{\refname}

\end{document}